\pdfoutput=1

\documentclass[aos,preprint]{imsart}

\RequirePackage{amsthm,amsmath,amsfonts,amssymb} \RequirePackage[authoryear]{natbib} \RequirePackage[colorlinks,citecolor=blue,linkcolor=blue,urlcolor=blue]{hyperref} \usepackage{cleveref}

\startlocaldefs \numberwithin{equation}{section}

\theoremstyle{plain} \newtheorem{theorem}{Theorem} \newtheorem{proposition}{Proposition} \newtheorem{corollary}{Corollary} \newtheorem{lemma}{Lemma}[section] \newtheorem{prop}[lemma]{Proposition} \theoremstyle{definition} \newtheorem{remark}{Remark}[section]

\crefname{theorem}{Theorem}{Theorems} \Crefname{theorem}{Theorem}{Theorems} \crefname{proposition}{Proposition}{Propositions} \Crefname{proposition}{Proposition}{Propositions} \crefname{corollary}{Corollary}{Corollaries} \Crefname{corollary}{Corollary}{Corollaries} \crefname{lemma}{Lemma}{Lemmas} \Crefname{lemma}{Lemma}{Lemmas} \crefname{remark}{Remark}{Remarks} \Crefname{remark}{Remark}{Remarks} \crefname{figure}{Figure}{Figures} \Crefname{figure}{Figure}{Figures} \crefname{section}{Section}{Sections} \Crefname{section}{Section}{Sections} \crefname{prop}{Proposition}{Propositions} \Crefname{prop}{Proposition}{Propositions} \crefformat{equation}{(#2#1#3)} \Crefformat{equation}{(#2#1#3)} \crefmultiformat{equation}{(#2#1#3)}{ and (#2#1#3)}{, (#2#1#3)}{, and (#2#1#3)} \Crefmultiformat{equation}{(#2#1#3)}{ and (#2#1#3)}{, (#2#1#3)}{, and (#2#1#3)} \crefrangeformat{equation}{(#3#1#4) to (#5#2#6)} \Crefrangeformat{equation}{(#3#1#4) to (#5#2#6)}

\newcommand{\sfV}{\mathsf{V}}                 % vertex set
\newcommand{\ind}{\mathbf{1}}                 % indicator
\newcommand{\E}{\mathbb{E}}
\newcommand{\bbR}{\mathbb{R}}
\DeclareMathOperator{\pa}{pa}
\DeclareMathOperator{\ch}{ch}
\DeclareMathOperator{\depth}{depth}
\DeclareMathOperator{\conv}{conv}
\DeclareMathOperator{\diam}{diam}
\newcommand{\Nup}{N^{\uparrow}_{T}}           % ancestor covering number
\newcommand{\kalg}{k_{\mathrm{alg}}}          % algorithmic crossing index
\def\journal@name{}
\endlocaldefs

\begin{document}

\begin{frontmatter}
\title{The Minimax Rate of Denoising Flows on Trees}
\runtitle{The Minimax Rate of Denoising Flows on Trees}

\begin{aug}
\author[A]{\fnms{Sichen}~\snm{Wang}\ead[label=e1]{wsc@smbu.edu.cn}}
\address[A]{Shenzhen MSU-BIT University\printead[presep={,\ }]{e1}}
\end{aug}
\runauthor{Sichen Wang}

\begin{abstract}
Isotonic regression and mean estimation over a simplex are among the most classical denoising problems under geometric constraints. Chatterjee and Lafferty generalized both to the recovery of a flow on a rooted tree from Gaussian noise, with isotonic regression on a path and the simplex on a star. Their results reach only special trees; on a general tree the minimax risk remained unknown for a decade. We settle the question in full. On every tree with $n$ vertices, at every budget $V$ and noise level $\sigma$, the minimax risk is of order $\min\{V^2H_T,\ \sigma^2k_0\}$, with $H_T$ the diameter and $k_0$ the crossing index of a truncated ancestor profile. A deterministic estimator, an exponentially weighted aggregate over integer states, attains it in $O(n^2\log(n+1))$ arithmetic operations. Least squares at a known budget, the natural convex program for the model, is suboptimal in the worst case by a factor of order $(\log n)^{2/5}$. One functional, the truncated ancestor profile, carries the rate, the estimator, and the least squares comparison. It serves as covering radius, as packing entropy in the positive cone, and as the range of the one-dimensional state behind the exact computation. Two classical theories become one on trees, sharp and computable throughout.
\end{abstract}

\begin{keyword}[class=MSC]
\kwdgroup[type=primary]{\kwd{62G05}
\kwd{62C20}}
\kwdgroup[type=secondary]{\kwd{62G20}
\kwd{68Q25}}
\end{keyword}

\begin{keyword}
\kwd{Flow polytope}
\kwd{isotonic regression}
\kwd{minimax risk}
\kwd{metric entropy}
\kwd{aggregation}
\kwd{least squares}
\end{keyword}

\end{frontmatter}

%%%%%%%%%%%%%%%%%%%%%%%%%%%%%%%%%%%%%%%%%%%%%%%%%%%%%%%%%%%%%%%%%%%%%%%%%%%%%%
\section{Introduction}\label{sec:intro}
%%%%%%%%%%%%%%%%%%%%%%%%%%%%%%%%%%%%%%%%%%%%%%%%%%%%%%%%%%%%%%%%%%%%%%%%%%%%%%

Two classical problems anchor Gaussian denoising under geometric constraints. Isotonic regression, the estimation of a monotone sequence, obeys the cube-root phase diagram of shape-constrained estimation \citep{zhang2002,chatterjee2015}. Mean estimation over a scaled simplex, the nonnegative part of an $\ell_1$ ball, obeys the $\sqrt{\log}$ diagram of sparse estimation \citep{donoho1994,birgemassart2001}. The two diagrams, each carrying decades of theory, are unlike in kind: the monotone problem pays a power of its depth, the sparse problem a logarithm of its width. Between the sequence and the simplex lies every geometry in which depth and branching interact, and there the picture has remained incomplete.

One model generalizes both. A \emph{tree flow} \citep{chatterjeelafferty2018} assigns to the vertices of a rooted tree nonnegative values such that each vertex passes to its children at most what it receives: a mass enters at the root, travels downward, and may leak at any vertex. A path, depth without branching, is bounded isotonic regression; a star, branching without depth, is mean estimation over the scaled simplex. The same inequalities arise wherever hierarchical totals are measured with noise. In the consistency post-processing of differentially private counts they hold with equality \citep{hay2010,abowd2022}; in profiling a call tree the leak carries real mass, the self time of a procedure \citep{graham1982}. Chatterjee and Lafferty studied the least squares estimator, the projection onto the flow cone and the natural tuning-free program for the model, and discovered that it behaves unlike its isotonic counterpart: its rate of convergence is not monotone in the depth of the tree, and it can miss the minimax rate by a power of $n$. Where the projection fell short, the minimax rate was attained only by least squares over an exponentially large net, and their closing sentence poses the challenge of ``closing the gap between the LSE and the minimax lower bound'' with an efficient estimator. A decade later the challenge stood, and with it three gaps: the minimax risk was uncharacterized beyond two regimes of trees, the worst case of least squares at a fixed budget was unquantified, and no efficient estimator was known to attain the minimax rate where the projection did not.

This paper closes all three. We ask: for every finite rooted tree and every budget and noise level, what is the minimax risk of denoising a tree flow, and is it attained by a computationally efficient estimator? One combinatorial functional of the tree organizes the answers, stated in \Cref{sec:results}: a rate formula; an exact and efficient estimator; adaptation to the budget and, in an explicit regime, to the noise level; and the worst-case price of least squares. In the local entropy program for convex constraints \citep{neykov2023}, the flow polytopes thereby form a nonsymmetric, combinatorially explicit family.

\subsection{Our Results}\label{sec:results}

Let $T$ be a finite rooted tree with vertex set $\sfV$, root $o$, and $n:=|\sfV|$ vertices. Write $u\preceq v$, and call $u$ an \emph{ancestor} of $v$, when $u$ lies on the path from $o$ to $v$, so that every vertex is an ancestor of itself; let $d_T$ denote graph distance, and let $H_T:=\max_{u,w\in\sfV}d_T(u,w)$ denote the diameter of $T$. Each vertex $u$ carries its root-path indicator $p_u\in\{0,1\}^{\sfV}$, defined by $p_u(v)=\ind\{v\preceq u\}$. The \emph{flow cone} of $T$ is
\begin{equation}\label{eq:cone}
\mathcal F(T):=\Bigl\{\textstyle\sum_{u\in\sfV}s_up_u:\ s_u\ge0
\text{ for all }u\Bigr\},
\end{equation}
whose elements are exactly the monotone flows on $T$ (\Cref{lem:normal}). Fix a budget $V>0$ and a noise level $\sigma>0$. The parameter set of this paper is the budget-$V$ slice of the cone, the \emph{flow polytope}
\begin{equation}\label{eq:body}
\mathcal F_V(T):=\bigl\{\mu\in\mathcal F(T):\ \mu(o)=V\bigr\}
\ =\ V\conv\{p_u:u\in\sfV\}.
\end{equation}
The cone \Cref{eq:cone} is the flow cone of \citet{chatterjeelafferty2018}; they cap the root value, $\mu(o)\le V$, where we fix it. For $n\ge2$ the two constraints carry the same minimax risk up to universal constants: the capped body contains the slice, raising the root value of a capped signal to $V$ changes no other coordinate, and the cap therefore adds only the bounded scalar $\mu(o)$, at a cost $\min\{V^2,\sigma^2\}$ that the two-point bound \Cref{eq:twopoint} absorbs; the estimator of \Cref{thm:efficient} transfers (\Cref{ssec:profile}). We observe $Y=\mu+\sigma Z$ with $Z\sim N(0,I_n)$ and $\mu\in\mathcal F_V(T)$, and study the minimax risk
\begin{equation}\label{eq:risk}
R^*_T(V,\sigma):=\inf_{\widehat\mu}\ \sup_{\mu\in\mathcal F_V(T)}
\E_\mu\bigl\|\widehat\mu(Y)-\mu\bigr\|_2^2,
\end{equation}
the infimum running over all measurable maps $\widehat\mu:\bbR^{\sfV}\to\bbR^{\sfV}$; estimators are not required to take values in $\mathcal F_V(T)$.

An \emph{ancestor $r$-net} is a set $R\subseteq\sfV$ such that every vertex $u$ has an ancestor $a\in R$ with $d_T(a,u)\le r$, and $\Nup(r)$ denotes the minimum cardinality of an ancestor $r$-net. The \emph{truncated ancestor profile} of $T$ is
\begin{equation}\label{eq:profile}
\alpha_k(T):=\sup_{r>0}\ r\,\min\Bigl\{k,\
\Bigl[\log\frac{\Nup(r)}{k}\Bigr]_+\Bigr\},
\qquad
\delta_k(T):=\frac{\alpha_k(T)}{k}
\qquad(k=1,2,\dots),
\end{equation}
where $[x]_+:=\max\{x,0\}$, and the \emph{crossing index} is
\begin{equation}\label{eq:crossing}
k_0=k_0(T,V,\sigma):=\min\Bigl\{k\ge1:\
V^2\delta_k(T)\le A_0\sigma^2k\ \text{ or }\ k>n/C_\star\Bigr\},
\end{equation}
where $A_0$ and $C_\star$ are fixed universal constants; one may take $A_0=144$ and $C_\star=e^2$. The truncation is the one nonstandard feature: once describing the branching below a vertex costs more than the information budget $k$, the cluster is blurred at its own diameter rather than charged per descendant (\Cref{rem:star}).

The crossing index determines the minimax risk: \Cref{thm:rate}, proved in \Cref{sec:rate}, shows that
\[
R^*_T(V,\sigma)\ \asymp\ \min\bigl\{V^2H_T,\ \sigma^2k_0(T,V,\sigma)\bigr\}
\]
for every finite rooted tree and all $V,\sigma>0$. The minimum has three readings: a diameter branch $V^2H_T$, where the body is small and the constant estimator $Vp_o$ is already optimal; an information branch $\sigma^2k_0$, the cost of $k_0$ effective coordinates; and, when the crossing clause never fires, a dimension branch $\sigma^2k_0\asymp\sigma^2n$, where the identity estimator is optimal. The formula is computable: a surrogate index evaluates it to within a factor of two in $O(n\log(n+1))$ operations (\Cref{lem:toolkit}).

Behind the formula stands the entropy theory of Gaussian estimation, from global metric entropy to the local entropy fixed point of bounded convex constraints \citep{yangbarron1999,neykov2023}. That theory identifies the governing quantity; computing it on a given body, and constructing the packings it promises when the body is an asymmetric cone, are the separate problems solved in \Cref{sec:rate}. For $t>0$ let $H^c_T(t)$ denote the metric entropy of the flow polytope at squared radius $t$: the logarithm of the smallest number of Euclidean balls of radius $\sqrt t$, centered anywhere, whose union contains $\mathcal F_V(T)$. Define the associated fixed point
\begin{equation}\label{eq:fixedpoint}
\rho_T(V,\sigma)\ :=\ \inf_{t>0}\ \bigl[t+\sigma^2H^c_T(t)\bigr].
\end{equation}
Free centers are the convenient convention, because the covers of \Cref{sec:rate} have centers outside the body; restricting centers to the body leaves \Cref{eq:fixedpoint} unchanged (\Cref{ssec:profile}). \Cref{prop:fixedpoint} gives $R^*_T(V,\sigma)\asymp\min\{\rho_T(V,\sigma),\ V^2H_T,\ \sigma^2n\}$; \Cref{thm:rate} thus evaluates this entropy expression in closed combinatorial form.

\begin{table}[t]
\caption{The rate formula on three families
(\Cref{cor:path,cor:star,cor:binary}), with $n$ the number of
vertices and universal constants in $\asymp$; the path $P_L$ has $L$
edges, the star $S_m$ has $m$ leaves, and the complete binary tree
$B_h$ has height $h$.
The first two rows recover the classical phase diagrams; the third
carries a full logarithm that occurs at neither endpoint.}
\label{tab:phases}
\centering
{\footnotesize
\begin{tabular}{@{}llll@{}}
\hline
Tree & $R^*\asymp$ & Classical counterpart & Phenomenon\\
\hline
path $P_L$ &
$\min\bigl\{V^2L,\ V^{2/3}\sigma^{4/3}L^{1/3},\ \sigma^2(L{+}1)\bigr\}$ &
bounded isotonic & cube root\\
star $S_m$ &
$\min\bigl\{V^2,\ \sigma V\sqrt{1+[\log(m\sigma/V)]_+},\ \sigma^2m\bigr\}$ &
simplex mean & $\sqrt{\log}$\\
binary $B_h$ &
$\min\bigl\{V^2h,\ \sigma V\bigl(1+[\log(n\sigma/V)]_+\bigr),\ \sigma^2n\bigr\}$ &
no direct analogue & full logarithm\\
\hline
\end{tabular}
}
\end{table}

On three benchmark families the formula specializes in closed form; \Cref{tab:phases} collects the three rates, proved as \Cref{cor:path,cor:star,cor:binary} in \Cref{sec:crossing}. The first two are external checks: the path recovers the bounded isotonic diagram \citep{zhang2002,chatterjeelafferty2018}; the star recovers the simplex diagram of sparse estimation \citep{donoho1994,birgemassart2001}. The complete binary tree resembles neither: branching entropy repeats across $\Theta(\log n)$ depth scales, making the profile $\alpha_k$ quadratic in $\log(n/k)$ once $k$ exceeds that logarithm, and the crossing takes a square root, leaving the single full logarithm.

An exactly evaluable estimator attains the rate. \Cref{thm:efficient}, proved in \Cref{sec:estimation}, constructs a deterministic estimator with worst-case risk $O(R^*_T(V,\sigma))$ on every tree, computed exactly in $O(n^2\log(n+1))$ arithmetic operations: an exponentially weighted aggregate over a class of integer states. Together with \Cref{thm:lse}, it answers the challenge of \citet{chatterjeelafferty2018}: an efficient estimator closes the gap between least squares and the minimax risk.

The estimator of \Cref{thm:efficient} takes $V$ and $\sigma$ as inputs; the next two theorems remove them. \Cref{thm:budget}, proved in \Cref{sec:adaptation}, removes the budget: a deterministic polynomial-time estimator free of $V$ attains $C(R^*_T(V,\sigma)+\sigma^2)$, and no budget-free estimator stays below an $\Omega(\sigma^2)$ floor on adjacent budgets. On the cone, where the original model carries no budget parameter, this adaptation is the return to the model as \citet{chatterjeelafferty2018} posed it. \Cref{thm:full} removes the noise level as well, in an explicit regime $Vw_T\le c\sigma n$, with $w_T$ a width functional that bounds how much of the tree's difference structure the signal can contaminate, equal to $1$ on a path and $\Theta(\log n)$ on a balanced binary tree. Within the regime the noise dominates a majority of the tree's difference statistics and a median estimates $\sigma$; on families of bounded width, the regime sits within a $\sqrt{\log n}$ factor of a ceiling that no noise estimator can pass (\Cref{ssec:adaptation}).

Every estimator above is built from the profile. The default estimator ignores it: the Euclidean projection onto the body, the tuning-free convex program of least squares at a known budget. \Cref{thm:lse}, proved in \Cref{sec:lse}, determines its worst case: on every tree and at all $V,\sigma>0$ the projection loses at most a factor $C(1+\log(en))^{2/5}$ over the minimax risk, and on an explicit one-parameter family of brooms, at a single budget and noise level, it loses at least $c(\log n)^{2/5}$. The separation is an exact power of the logarithm, matched from both sides, where the separations of \citet{chatterjee2014,kur2024} are polynomial in the sample size or the dimension. The equivalence of the fixed and capped budgets recorded above concerns minimax risks; \Cref{thm:lse} concerns the projection itself.

\subsection{Technical Overview}\label{sec:overview}

The upper bound in \Cref{thm:rate} is soft: least squares over a covering net attains the entropy fixed point \Cref{eq:fixedpoint}, and two trivial estimators cap the risk at $V^2H_T$ and $\sigma^2n$. Three tasks remain. The entropy must be \emph{realized}, by hypotheses legal in a positive body (the lower bound); \emph{searched}, exactly and quickly (the algorithm); and \emph{measured} against the natural convex program (the least squares question). Two features of the flow polytope obstruct the classical routes to all three: positivity, the body being a simplex rather than a symmetric ball, so that a hypothesis moves only the leak mass it holds and sign-toggling lower bounds, whose one base point would have to fund every signed perturbation, fail; and overlap, root paths through a common vertex sharing coordinates, so that the contrast blocks realizing the entropy can overlap and are therefore priced (\Cref{sec:packing}).

The lower bound realizes the entropy inside the body. Ancestor-covering numbers force a separated vertex set (\Cref{lem:bridge}); groups drawn from it, in subtrees with disjoint edge sets, share the budget exactly; a multiscale code fills each group; and a product form of Fano's inequality (\Cref{lem:fano}) prices the overlap between scales once, through a normalized Gram row sum, returning $R^*_T\ge c\min\{V^2r/q,\ \sigma^2q\log(eM/q)\}$ for $M$ vertices at pairwise distance $r$ in $q$ groups; matching $q$ to the profile recovers its truncated logarithm. Where \citet{chatterjeelafferty2018} allocate mass across vertex-disjoint paths, the same budget discipline runs on arbitrarily overlapping routes.

The upper bound builds one coded cover for all signals. Dyadic ancestor nets and their cells depend only on the tree and the information budget $k$; for each signal, heavy cells split and light cells collapse to their roots, sparse sampling restoring the light exits; every subtree is an interval of the depth-first preorder, so one scalar, the running floor of cumulative masses along it, rounds all subtree sums at once. The resulting approximant is within squared error $O(V^2\delta_k)$ of its signal and carries a fixed root sum, a code length $O(k)$, and flow states confined to an $O(k)$ range; counting the class at $e^{O(k)}$ gives the entropy bound \Cref{eq:entropycover}. An integer-leak net already appears in the upper bound of \citet{chatterjeelafferty2018}, at cost $\log n$ per unit of mass and at the diameter scale; the activation charges lower this to $O(1)$ per unit and the profile scale, and these reductions are what the estimator consumes.

The two halves meet at the crossing. In the interior case the nonincreasing profile scale $V^2\delta_k$ first falls below $A_0\sigma^2k$ at $k_0$; the cover gives the rate $\sigma^2k_0$ there, the packing gives $\sigma^2(k_0-1)$ one step earlier, and the corner regimes are collected by a two-point bound and the dimension cap. This proves \Cref{thm:rate,prop:fixedpoint} as four cases of one argument.

The estimator replaces the search over the cover by an average, an exponentially weighted aggregate under a Kraft-type prior built from the code length. Its risk analysis composes two exact identities, Stein's and the Gibbs variational identity, with no union bound and approximation constant one (\Cref{lem:oracle}). The confinement is what makes the computation possible: one integer state in $[0,3k]$ per vertex turns the partition function into a tree recursion of polynomial messages, evaluated exactly; one reverse-mode sweep then reads off the posterior mean at every coordinate. The recursion is dynamic programming over distributions, not values.

Adaptation reuses the cover. Removing the quantization from the approximant of \Cref{prop:cover} leaves amplitude-free profile subspaces; a Kraft-weighted penalized selection over them, joined with finite grids of candidate budgets, attains $C(R^*_T+\sigma^2)$ without knowing $V$, the additive $\sigma^2$ forced on adjacent budgets by a two-point test at the root. For unknown $\sigma$, a median of tree difference statistics survives contamination when $Vw_T\le c\sigma n$, and the selection runs at the rounded estimate.

The broom of \Cref{thm:lse}, a handle of $L$ edges feeding $e^{\Theta(L^{5/2})}$ terminal leaves, makes the blindness of the projection quantitative: the extreme leaf noise repays the quadratic cost of the full-budget route through the handle, forcing least squares to lose $L^{3/2}$, while the profile sees one ancestor covering the terminal star at every radius $r\ge1$ and the minimax risk stays at order $\sqrt L$, that of the handle alone; with $\log n\asymp L^{5/2}$, the ratio is the $2/5$ power of the logarithm. For the matching upper bound, the localized-width analysis of \citet{chatterjee2014} runs on the coded cover with centers chosen after the noise, and a width branch and a height branch meet at the same power.

Three devices are stated for reuse: the overlapping-block Fano inequality (\Cref{lem:fano}), which refers to no tree and prices overlap inside an arbitrary convex constraint; the positive packing of \Cref{sec:packing}, a template for bodies in which sign toggling is illegal; and the aggregate of \Cref{sec:estimation}, whose oracle inequality holds over any finite class and whose evaluation is exact whenever the class is dynamically
programmable. \Cref{sec:interfaces,%
sec:rate,sec:estimation,sec:lse} prove the results in turn; the remaining details are in the appendices.

\subsection{Related Work}\label{sec:related}

The closest work is \citet{chatterjeelafferty2018}. They introduced tree flows and analyzed the projection onto the flow cone in two regimes: for trees of bounded depth they proved matching risk bounds in the principal range of budgets, and on trees consisting of many long disjoint paths they determined the exponents of both the projection and the minimax risk, which disagree. Their projection takes no budget input, and over part of that family its risk exceeds the minimax risk polynomially in $n$ at a fixed budget, by overfitting the zero signal where the cone has large statistical dimension; the projection of \Cref{thm:lse} receives the budget, which removes that failure and leaves exactly the $2/5$ power of the logarithm. The constraint family itself predates the statistics: \citet{benabbas2011} smooth hierarchical data under the same inequalities in $\ell_1$. Isotonic regression on richer orders \citep{han2019} shares the ambient structure but not the geometry: it acts on differences, while the flow constraint transports mass, thereby making the body a simplex. On the path the model meets the isotonic tradition \citep{zhang2002,chatterjee2015,bellec2018,guntuboyinasen2018}, in its bounded form \citep{lussrosset2017}, and \Cref{cor:path} returns its phase diagram. The obstacle on trees is therefore the budget slice, the simplex geometry, not monotonicity.

Metric entropy has determined minimax rates since \citet{yangbarron1999} tied them to the global entropy of the function class; for bounded convex constraints in the Gaussian sequence model, \citet{neykov2023} gives the exact rate through a local entropy fixed point. On the flow polytopes the two halves of that program take concrete form: the profile computes the rate that the fixed point governs, in closed combinatorial form and in near-linear time, and the packing of \Cref{sec:packing} realizes the entropy inside a cone with no symmetry, one-sided perturbations replacing sign toggling; compare the cone geometry of testing in \citet{wei2019}. On the star, positivity leaves the sparse rate unchanged; what it changes is the lower bound, which must be built one-sided. The covering geometry has an operator-theoretic precedent: the body is $V$ times the image of the positive face of the $\ell_1$ ball under the tree's summation operator, and \citet{lifshitslinde2011,lifshitslinde2011critical} bound the entropy numbers of such operators through ancestor covering numbers; the critical case stops a heavy--light partition at its light sets and counts the outcomes, the combinatorial core that \Cref{sec:cover} shares. Compactness of the symmetric image is their question; the positive slice, the minimax risk, and the certificates behind the estimator are this paper's.

That the projection onto a convex body can miss the minimax rate by a large factor was shown by \citet{zhang2013} on an ellipsoid and by \citet{chatterjee2014} on a designed body, in both cases by a factor as large as $\sqrt n$; the fixed-point analysis of least squares in the latter is the engine of our upper bound in \Cref{sec:lse}. Separations of polynomial order in the sample size or the dimension are known for several natural programs \citep{kur2024,vaskevicius2023}, and on an $\ell_p$ ball \citet{aolaritei2025} separate the projection from the minimax rate by a power of the logarithm. \citet{prasadan2025} characterize the boundary exactly: least squares is minimax optimal if and only if the local Gaussian width map is Lipschitz at the relevant scales. By that characterization, \Cref{thm:lse} shows the Lipschitz bound cannot hold uniformly along the brooms, and it goes past the optimal-or-not dichotomy by bounding the worst-case ratio of least squares to the minimax rate from both sides. The upper bound speaks for the projection as much as against it: a loss beyond the $2/5$ power of the logarithm is impossible on any tree.

The estimator of \Cref{thm:efficient} belongs to the aggregation tradition. \citet{leungbarron2006} proved an exact Stein-based oracle inequality for exponentially weighted aggregates, and \Cref{lem:oracle} descends from it; sharp PAC-Bayes bounds, optimal sparse aggregation, and unknown-variance mixing follow in \citet{dalalyantsybakov2008,rigollettsybakov2011,giraud2008}. Exact mixtures over combinatorial families have a computing tradition of their own: context-tree weighting \citep{willems1995} and circuit differentiation for marginals \citep{darwiche2003}, without squared-loss risk guarantees. The class of integer states sits in both traditions: its code length ties the Kraft prior, in the tradition of minimum description length \citep{barroncover1991}, to the minimax rate, and its confined states make the posterior mean a tree recursion, evaluated exactly. Shape-constrained solvers compute hard projections \citep{robertson1988,kyng2015}, the object that \Cref{thm:lse} measures; the recursion here averages where those optimize. A recent line attains polynomial-time near-optimality on symmetric bodies under further regularity and oracle hypotheses \citep{neykov2026}; the flow polytope is a simplex. The trade is structure for exactness: a combinatorial family of bodies, in exchange for the exact rate and the exact estimator.

%%%%%%%%%%%%%%%%%%%%%%%%%%%%%%%%%%%%%%%%%%%%%%%%%%%%%%%%%%%%%%%%%%%%%%%%%%%%%%
\section{The Flow Polytope and the Ancestor Profile}\label{sec:interfaces}
%%%%%%%%%%%%%%%%%%%%%%%%%%%%%%%%%%%%%%%%%%%%%%%%%%%%%%%%%%%%%%%%%%%%%%%%%%%%%%

This section records the interfaces used by the rest of the paper: the two normal forms of the flow polytope, the Euclidean realization of the tree metric, and the working properties of the ancestor profile, including the surrogate index through which the profile is computed. Proofs are collected in \Cref{ssec:profile}.

\subsection{The Flow Polytope}\label{sec:polytope}

We complete the tree notation of \Cref{sec:results}. For $v\ne o$ we write $\pa(v)$ for the parent of $v$ and identify $v$ with its parent edge $(\pa(v),v)$, of which $v$ is the \emph{lower endpoint}; $\ch(v)$ is the set of children of $v$, and $T_v:=\{u\in\sfV:v\preceq u\}$ is the subtree rooted at $v$. The depth of a vertex is its distance from the root, and the height of $T$ is $h_T:=\max_{v\in\sfV}\depth(v)$, so that $h_T\le H_T\le2h_T$. Universal constants $c,C,c',C',\dots$ may change value from occurrence to occurrence, while subscripted constants such as $A_0$ and $C_\star$ keep their values fixed once introduced. Logarithms are natural. Complexity statements count arithmetic operations in the real-RAM model: arithmetic, comparisons, integer part, and evaluations of $\exp$ and $\log$ on real registers are exact unit-cost operations, and two polynomials of degree $m$ can be multiplied in $O(m\log m)$ operations by the fast Fourier transform \citep{vonzurgathen2013}; what exactness means for the estimator is stated in \Cref{rem:exact}.

The polytope has a barycentric normal form and a flow normal form: the first is the legality certificate behind every lower-bound construction of \Cref{sec:rate}, the second identifies the body as a set of flows and explains its name.

\begin{lemma}[Normal forms]\label{lem:normal}
For $\mu\in\bbR^{\sfV}$ the following are equivalent:
\begin{enumerate}
\item\label{it:body} $\mu\in\mathcal F_V(T)$;
\item\label{it:bary} $\mu=\sum_{u\in\sfV}s_up_u$ with $s_u\ge0$ for all
$u$ and $\sum_{u\in\sfV}s_u=V$;
\item\label{it:flow} $\mu(o)=V$ and
$\mu(v)\ge\sum_{c\in\ch(v)}\mu(c)$ for every $v\in\sfV$.
\end{enumerate}
The coefficients in \textup{(\ref{it:bary})} are unique, and they are tied to the coordinates by the subtree sums
\[
\mu(v)\ =\ \sum_{u\succeq v}s_u,
\qquad
s_v\ =\ \mu(v)-\sum_{c\in\ch(v)}\mu(c)
\qquad(v\in\sfV).
\]
\end{lemma}

Form (\ref{it:flow}) exhibits $\mathcal F_V(T)$ as the set of monotone flows at budget $V$, with $s_v$ the \emph{leak} at $v$, the mass that stops there \citep{chatterjeelafferty2018}. In a linear order of $\sfV$ that places ancestors before descendants, the matrix with columns $(p_u)_{u\in\sfV}$ is upper triangular with unit diagonal (\Cref{ssec:profile}), so the $p_u$ are linearly independent and the body is a nondegenerate $(n-1)$-dimensional simplex with vertex set $\{Vp_u:u\in\sfV\}$. Finally, $\mu(o)=V$ identically on the body: the root coordinate carries no statistical content, and all estimation error lives on the nonroot coordinates.

The second interface converts tree combinatorics into Euclidean geometry: squared Euclidean distance between root-path indicators is tree distance, which is why the ancestor profile, defined by the tree metric, controls Euclidean covering and packing of the body.

\begin{lemma}[Isometry and edge supports]\label{lem:isometry}
\leavevmode
\begin{enumerate}
\item\label{it:iso} $\|p_u-p_w\|_2^2=d_T(u,w)$ for all $u,w\in\sfV$;
consequently $\diam^2\mathcal F_V(T)=V^2H_T$.
\item\label{it:edge} For $a\preceq b$, $p_b-p_a=\ind_{(a,b]}$, the
indicator of $\{v:a\prec v\preceq b\}$; in general, the support of
$p_u-p_w$ is exactly the set of lower endpoints of the edges on the
path from $u$ to $w$. In particular, differences $p_u-p_w$ taken
within edge-disjoint connected subtrees have disjoint supports.
\end{enumerate}
\end{lemma}

Behind (\ref{it:iso}) is the ancestor count $\langle p_u,p_w\rangle=\depth(u\wedge w)+1$, where $u\wedge w$ denotes the deepest common ancestor of $u$ and $w$.

\subsection{The Ancestor Profile}\label{sec:profile}

The profile $\alpha_k,\delta_k$ and the crossing index $k_0$ were defined in \Cref{eq:profile,eq:crossing}; this subsection collects the properties used below. A star shows why the truncation is there.

\begin{remark}[Why the profile is truncated]\label{rem:star}
On the star with $m$ leaves rooted at the hub, a leaf is ancestor-covered only by itself or by the hub, so $\Nup(r)=m+1$ for $0\le r<1$ and $\Nup(r)=1$ for $r\ge1$; the expression in \Cref{eq:profile} evaluates exactly to
\[
\alpha_k\ =\ \min\Bigl\{k,\ \Bigl[\log\frac{m+1}{k}\Bigr]_+\Bigr\}.
\]
Without the truncation, a star with $\lceil e^{2k^2}\rceil$ leaves would be assigned $\alpha_k\ge k^2$, hence $V^2\alpha_k/k\ge A_0\sigma^2k$ at $V^2=A_0\sigma^2$, and the profile-to-risk half of \Cref{sec:rate} would assert a minimax risk of order at least $\sigma^2k$; for large $k$ this contradicts the bound $R^*_T\le V^2H_T=2A_0\sigma^2$ furnished by the constant estimator $Vp_o$, since the star has $H_T=2$.
\end{remark}

The root covers every vertex at radius $h_T$, so $\Nup(r)=1$ for $r\ge h_T$ and the supremum in \Cref{eq:profile} is effectively over $r<h_T$. Evaluating the profile exactly involves every integer radius below $h_T$; a factor-of-two surrogate needs only geometrically spaced ones, and it is the surrogate that the estimator of \Cref{sec:estimation} consumes. Define
\begin{equation}\label{eq:surrogate}
\overline\delta_k(T):=\frac2k\,
\max_{\substack{q=2^\ell-1,\ \ell\in\{0,1,2,\dots\}\\ 2^\ell\le h_T}}\
(q+1)\,\min\Bigl\{k,\ \Bigl[\log\frac{\Nup(q)}{k}\Bigr]_+\Bigr\},
\end{equation}
with $\overline\delta_k(T):=0$ if $h_T=0$, and the \emph{algorithmic crossing index}
\begin{equation}\label{eq:kalg}
\kalg=\kalg(T,V,\sigma):=\min\Bigl\{k\ge1:\
V^2\overline\delta_k(T)\le2A_0\sigma^2k\ \text{ or }\
k>n/C_\star\Bigr\}.
\end{equation}

\begin{lemma}[Profile toolkit]\label{lem:toolkit}
For every finite rooted tree $T$:
\begin{enumerate}
\item\label{it:mono} $\delta_{k+1}(T)\le\delta_k(T)\le h_T$ for every
$k\ge1$;
\item\label{it:discrete} for $n\ge2$ and every $k\ge1$,
\begin{equation}\label{eq:discrete}
\alpha_k(T)=\max_{0\le q<h_T}\ (q+1)\,\min\Bigl\{k,\
\Bigl[\log\frac{\Nup(q)}{k}\Bigr]_+\Bigr\},
\end{equation}
the maximum running over integer radii, and $\alpha_k(T)=0$ when
$n=1$;
\item\label{it:surrogate} $\delta_k\le\overline\delta_k\le2\delta_k$
for every $k\ge1$, the map $k\mapsto\overline\delta_k(T)$ is
nonincreasing, and
\[
\kalg\ \le\ k_0\ \le\ 2\,\kalg .
\]
Moreover, a minimum ancestor $q$-net can be computed exactly in $O(n)$
operations for each integer $q\ge0$, by a one-pass
\emph{residual-depth greedy}; consequently $\kalg$ is computable in
$O(n\log(n+1))$ operations, and $k_0$ in $O(n\,h_T)$ operations.
\end{enumerate}
\end{lemma}

The greedy selects a center $q$ levels above a deepest uncovered vertex, or the root when fewer remain; its exactness and the remaining proofs are in \Cref{ssec:profile}.

\Cref{sec:rate} now converts the profile into the minimax rate: an upper bound through a coded cover, and a matching lower bound through a positive packing.

%%%%%%%%%%%%%%%%%%%%%%%%%%%%%%%%%%%%%%%%%%%%%%%%%%%%%%%%%%%%%%%%%%%%%%%%%%%%%%
\section{The Rate Formula}\label{sec:rate}
%%%%%%%%%%%%%%%%%%%%%%%%%%%%%%%%%%%%%%%%%%%%%%%%%%%%%%%%%%%%%%%%%%%%%%%%%%%%%%

The main theorem of the paper is the rate formula.

\begin{theorem}[Rate formula]\label{thm:rate}
There are universal constants $c,C>0$ such that for every finite rooted tree $T$ and all $V,\sigma>0$,
\[
c\,R^*_T(V,\sigma)\ \le\ \min\bigl\{V^2H_T,\ \sigma^2k_0(T,V,\sigma)\bigr\}
\ \le\ C\,R^*_T(V,\sigma).
\]
\end{theorem}

The formula is insensitive to the choice of $(A_0,C_\star)$ in \Cref{eq:crossing}: any pair for which \Cref{thm:rate} holds yields a formula $\asymp R^*_T$. The entropy reading is the following.

\begin{proposition}[Entropy fixed point]\label{prop:fixedpoint}
For every finite rooted tree $T$ and all $V,\sigma>0$,
\[
R^*_T(V,\sigma)\ \asymp\
\min\bigl\{\rho_T(V,\sigma),\ V^2H_T,\ \sigma^2n\bigr\}
\]
with universal constants.
\end{proposition}

\Cref{sec:cover} converts the profile into a coded cover of the flow polytope, yielding the upper bounds and the certificates that the estimators of \Cref{sec:estimation} later consume; \Cref{sec:packing} converts the same profile into legal packings and the matching lower bounds; and \Cref{sec:crossing} joins the two at the crossing index, proves both statements above, and derives the benchmark corollaries. The assembly is given here; the proofs of the two halves are in \Cref{ssec:cover,ssec:packing}.

\subsection{The Coded Cover}\label{sec:cover}

Fix an integer $2\le k\le n/C_\star$. The construction runs at the \emph{working scale}
\[
\bar\alpha\ :=\ k\,\overline\delta_k(T),
\]
built from the surrogate \Cref{eq:surrogate}; the centers, and ultimately the estimator, owe their computability to this one choice. Everything rests on a fixed hierarchy of nets. Set $m_j:=2^j$ for $0\le j\le J$, where $m_J$ is the largest power of two strictly below $k$, let $S_j$ be the minimum ancestor $\lfloor\bar\alpha/m_j\rfloor$-net computed by the residual-depth greedy of \Cref{lem:toolkit}(\ref{it:surrogate}), and put $R_j:=S_0\cup\dots\cup S_j$, so that $R_0\subseteq\dots\subseteq R_J$ and each $R_j$ is an ancestor $(\bar\alpha/m_j)$-net containing the root. The nets depend only on $(T,k)$, not on any signal; every counting statement below rests on this. Their sizes are controlled by the profile itself: $\Nup(\lfloor\bar\alpha/m_j\rfloor)\le ke^{m_j}$, as $m_j<k$ and a larger count would make the term of \Cref{eq:profile} at that radius exceed $\bar\alpha\ge\alpha_k(T)$; consequently $|R_j|\le2ke^{m_j}$.

For $v\in\sfV$ let $a_j(v)$ be the deepest ancestor of $v$ in $R_j$, and call the fibers $B_{a,j}:=\{v:a_j(v)=a\}$, $a\in R_j$, the level-$j$ \emph{cells}. At each level the cells partition $\sfV$; each cell is \emph{rooted-connected}, containing with any of its vertices the whole segment from its root $a$ down to that vertex; the cells of level $j+1$ refine those of level $j$; and every vertex of a cell is within tree distance $\bar\alpha/m_j$ of its root. These are elementary consequences of the net property (\Cref{ssec:cover}).

Define the \emph{active support}, the first-appearance weight, and the \emph{activation charge}
\[
A_k:=R_J,
\qquad
\omega(v):=\min\{m_j:\ v\in R_j\},
\qquad
\widetilde\omega(v):=\omega(v)\,\ind\{v\notin R_0\},
\]
and, for an integer vector $z$ supported on $A_k$, the \emph{code length}
\begin{equation}\label{eq:codelength}
\Gamma(z)\ :=\ \sum_{v\in A_k}\Bigl(|z_v|
+\widetilde\omega(v)\,\ind\{z_v\ne0\}\Bigr):
\end{equation}
coefficient mass, plus a charge on each support vertex equal to its dyadic discovery cost, the coarsest level being free. \Cref{sec:estimation} converts exactly this additive functional into a prior.

The next proposition is the key technical result of the upper half. Beyond the covering statement it certifies three additive budgets simultaneously: a fixed root sum, a bounded code length, and a confined flow state. It is through these certificates that the cover later becomes an estimator.

\begin{prop}[Coded cover]\label{prop:cover}
There are universal constants $C_0,C_1$ such that for every finite
rooted tree, every $V>0$, and every integer $2\le k\le n/C_\star$ the
following hold.
\begin{enumerate}
\item\label{it:approx} For every $\mu\in\mathcal F_V(T)$ there is an
integer vector $z$ supported on $A_k$ whose flow states
$x_v:=\sum_{u\succeq v}z_u$ satisfy
\begin{gather*}
\sum_{v}z_v=k,
\qquad
\Gamma(z)\le9k,
\qquad
0\le x_v\le3k\ \text{ for all }v\in\sfV,
\\
\Bigl\|\mu-\frac Vk\sum_vz_vp_v\Bigr\|_2^2\ \le\ C_0V^2\delta_k(T).
\end{gather*}
\item\label{it:count} The integer vectors supported on $A_k$ that
satisfy the first three constraints displayed in
\textup{(\ref{it:approx})} number at most $e^{C_1k}$.
\end{enumerate}
Since $A_k$ and $\widetilde\omega$ depend only on $(T,k)$, the centers
built from these vectors form a class
\[
\mathcal D_k\ :=\ \Bigl\{\tfrac Vk\textstyle\sum_vz_vp_v:\
z\in\mathbb Z^{A_k},\ \sum_vz_v=k,\ \Gamma(z)\le9k,\
0\le x_v\le3k\ \forall v\Bigr\},
\]
the \emph{coded cover}, which does not depend on the signal and covers
$\mathcal F_V(T)$ at squared radius $C_0V^2\delta_k(T)$; in particular
\begin{equation}\label{eq:entropycover}
H^c_T\bigl(C_0V^2\delta_k(T)\bigr)\ \le\ C_1k
\qquad\text{for every integer }2\le k\le n/C_\star .
\end{equation}
\end{prop}

Members of $\mathcal D_k$ may carry negative coefficients $z_v$, so the cover is external to the body, which is why $H^c_T$ was taken with free centers; in flow coordinates, every member is nonnegative with root coordinate exactly $V$ and no coordinate above $3V$.

\begin{proof}
By homogeneity take $V=1$ and write $f=\sum_u\lambda_up_u$ in barycentric form (\Cref{lem:normal}), with $\lambda(B):=\sum_{u\in B}\lambda_u$. The construction stops a refinement of the cells at the signal's own masses, collapses each stopped cell to its root, rounds the resulting skeleton, and samples the light exits; the budgets are then read off and the class is counted. The cell properties above, the collapse and exit estimates, the charge and skeleton accounting, and the counting computation are carried out in \Cref{ssec:cover}.

\emph{Stopped refinement.} Call a level-$j$ cell $B$ \emph{heavy} when $\lambda(B)>m_j/k$ and \emph{light} otherwise. Starting from level $0$, stop at every light cell, refine every heavy cell, and stop everything at level $J$. The \emph{terminal} cells partition $\sfV$; write $\mathcal L$ for their collection, and for $L\in\mathcal L$ let $a_L$ be its root, $m(L)$ its level weight $m_j$, and $\lambda_L:=\lambda(L)$. Heaviness is inherited upward through the refinement, so the maximal heavy cells are pairwise disjoint (\Cref{ssec:cover}) and each has mass exceeding its own $m_j/k$; since their masses total at most one,
\begin{equation}\label{eq:heavyweight}
\sum_{C\ \text{maximal heavy}}m(C)\ <\ k .
\end{equation}
A heavy terminal cell sits at level $J$ and is maximal, so $m_J\ge k/2$ leaves room for at most one.

\emph{Collapse.} Let $g:=\sum_{L\in\mathcal L}\lambda_L\,p_{a_L}$. The error blocks $e_L:=\sum_{u\in L}\lambda_u(p_u-p_{a_L})$ are supported in $L\setminus\{a_L\}$ by rooted-connectedness and \Cref{lem:isometry}(\ref{it:edge}), so they have pairwise disjoint supports; each is controlled by the radius of its cell and, on a light cell, by its mass, and summing the disjoint blocks gives
\begin{equation}\label{eq:collapse}
\|f-g\|_2^2\ \le\ \frac{3\bar\alpha}k
\end{equation}
(\Cref{ssec:cover}).

\emph{Skeleton and exits.} Let $U$ consist of $R_0$ together with the roots of all heavy cells. Attach each terminal cell to $U$: set $h(L):=a_L$ if $L$ is a level-$0$ cell or the heavy terminal cell; otherwise the parent cell of $L$ is heavy, since $L$ was reached by the refinement, and $h(L)$ is its root, so that $h(L)\in U$ and $h(L)\preceq a_L$. With $q_L:=p_{a_L}-p_{h(L)}$, \Cref{lem:isometry}(\ref{it:iso}) and the parent cell's radius give $\|q_L\|_2^2=d_T(h(L),a_L)\le2\bar\alpha/m(L)$, and
\[
g\ =\ g_0+\sum_L\lambda_L\,q_L,
\qquad
g_0:=\sum_{h\in U}b_h\,p_h,
\quad
b_h:=\sum_{L:\,h(L)=h}\lambda_L .
\]
The nonzero $q_L$ are the \emph{light exits}.

\emph{Preorder rounding.} The skeleton masses are rounded through one scalar. List $U$ as $h_1,\dots,h_s$ in a fixed depth-first preorder of $T$, write $b_i:=b_{h_i}$ for the masses in this order, set $B_0:=0$ and $B_i:=\sum_{\ell\le i}b_\ell$, and put
\[
Q_i\ :=\ \lfloor kB_i\rfloor-\lfloor kB_{i-1}\rfloor\ \in\
\mathbb Z_{\ge0},
\qquad\text{so that}\qquad
\Bigl|\sum_{i\in I}Q_i-k\sum_{i\in I}b_i\Bigr|\ <\ 1
\]
for every interval $I$ of the list, the left side being a difference of two floor errors. Write $Q$ for the vector with entry $Q_i$ at $h_i$ and zero elsewhere, a nonnegative integer leak vector of total mass $k$ supported on $U$ (\Cref{ssec:cover}). The point of the preorder is that every subtree $T_v$ is a contiguous block of it, so $U\cap T_v$ is an interval of the list; coordinates being subtree sums (\Cref{lem:normal}), the skeleton and its rounding $\bar g_0:=\frac1k\sum_{h\in U}Q_h\,p_h$ differ by less than $1/k$ at every vertex simultaneously. Both signals vanish outside the minimal rooted subtree spanning $U$, which has $O(\bar\alpha k)$ edges (a chaining argument over the nets; \Cref{ssec:cover}), whence
\begin{equation}\label{eq:roundingerror}
\|g_0-\bar g_0\|_2^2\ \le\ C\,\frac{\bar\alpha}k .
\end{equation}

\emph{Exit quantization.} A light exit has $\lambda_L\le m(L)/k$, so the independent variables $Z_L:=\frac{m(L)}k\,\xi_L$ with $\xi_L\sim\mathrm{Bernoulli}(k\lambda_L/m(L))$ are well defined and unbiased. Independence makes the cross terms vanish in expectation no matter how the exits overlap, so $\E\|\sum_L(Z_L-\lambda_L)q_L\|_2^2=\sum_L\operatorname{Var}(Z_L)\|q_L\|_2^2\le2\bar\alpha/k$. The selection weight $W:=\sum_Lm(L)\ind\{Z_L\ne0\}$ has $\E W\le k$; by Markov's inequality, some realization satisfies both $W\le2k$ and squared exit error at most $8\bar\alpha/k$ (\Cref{ssec:cover}). Fix one and define, with $e_v$ the coordinate vector at $v$, $z:=Q+\sum_{L:\,Z_L\ne0}m(L)\bigl(e_{a_L}-e_{h(L)}\bigr)$. Each transfer has coefficient sum zero, so $\sum_vz_v=k$; and $z$ is supported on $A_k$ with $\|z\|_1\le5k$.

\emph{The flow states.} Since $x_v=\sum_uz_up_u(v)$, the flow states of $z$ are the coordinates of $\sum_uz_up_u$. The skeleton $Q$ is a nonnegative leak vector of total mass $k$, so its states lie in $[0,k]$; and each selected transfer contributes $m(L)(p_{a_L}-p_{h(L)})=m(L)\ind_{(h(L),a_L]}\ge0$ by \Cref{lem:isometry}(\ref{it:edge}), a downward transport of mass that never creates a negative state and adds at most $m(L)$ to any state. Hence $0\le x_v\le k+W\le3k$ at every vertex.

\emph{Code length and error.} The code length collects the activation charges: the selected exit roots contribute at most $W\le2k$, and the heavy roots at most $2k$, because the level weights along a chain of heavy cells total less than twice that of the maximal cell it descends to, which \Cref{eq:heavyweight} caps; so $\Gamma(z)\le9k$ (\Cref{ssec:cover}). The error separates along the three stages of the construction,
\[
f-\frac1k\sum_vz_vp_v\ =\ (f-g)+(g_0-\bar g_0)+\sum_L(\lambda_L-Z_L)q_L,
\]
so \Cref{eq:collapse,eq:roundingerror} and the fixed realization give $\bigl\|f-\frac1k\sum_vz_vp_v\bigr\|_2^2\le C\bar\alpha/k =C\overline\delta_k(T)\le2C\delta_k(T)$ by \Cref{lem:toolkit}(\ref{it:surrogate}). Restoring the amplitude proves (\ref{it:approx}).

\emph{Counting.} A vector obeying the three budgets has support split into a free part inside $R_0$ and a charged part of total activation charge at most $9k$; the net sizes bound each charge group, so a generating-function estimate controls the number of supports, and on a fixed support, $\|z\|_1\le9k$ bounds the number of coefficient vectors. Multiplying the counts gives (\ref{it:count}) (\Cref{ssec:cover}). The covering statement and \Cref{eq:entropycover} follow from (\ref{it:approx}) and (\ref{it:count}).
\end{proof}

Two displays close the upper half. The first is generic. For any $t>0$, least squares over a covering of the body at squared radius $t$ with $\log$-cardinality $H^c_T(t)$, or the single center itself when one ball suffices, has risk at most $C(t+\sigma^2H^c_T(t))$ uniformly over the body, by the oracle inequality for least squares over a finite set of candidates (\Cref{ssec:packing}). We optimize over $t$ and adjoin the constant estimator $Vp_o$, whose risk is at most $\sup_\mu\|\mu-Vp_o\|_2^2\le V^2h_T$ by convexity and \Cref{lem:isometry}, and the identity estimator $Y$, whose risk is $\sigma^2n$. This gives
\begin{equation}\label{eq:genericupper}
R^*_T(V,\sigma)\ \le\
C\bigl(\rho_T(V,\sigma)\wedge V^2H_T\wedge\sigma^2n\bigr),
\end{equation}
the upper half of \Cref{prop:fixedpoint}. The second display is where the profile enters: evaluating the infimum \Cref{eq:fixedpoint} at $t=C_0V^2\delta_k(T)$ and inserting \Cref{eq:entropycover},
\begin{equation}\label{eq:rhocover}
\rho_T(V,\sigma)\ \le\ C\bigl(V^2\delta_k(T)+\sigma^2k\bigr)
\qquad\text{for every integer }2\le k\le n/C_\star .
\end{equation}
\Cref{eq:genericupper} knows nothing of the profile; \Cref{eq:rhocover} is what the crossing argument of \Cref{sec:crossing} consumes.

\subsection{The Packing Lower Bound}\label{sec:packing}

The lower bound must realize the same entropy through hypotheses that live in the polytope, and here the two structural features of the body bind: a hypothesis may only move leak mass that it actually possesses, and root paths through a common vertex share coordinates, so the hypotheses cannot be taken orthogonal on an arbitrary tree. The construction of this subsection allocates disjoint mass budgets exactly, and measures the overlap it cannot remove.

\begin{lemma}[Cover to packing]\label{lem:bridge}
For every $r>0$ there exist at least $\Nup(r)$ vertices of $T$ whose pairwise tree distances are all at least $r/9$.
\end{lemma}

\begin{proof}[Proof sketch]
Under the isometry of \Cref{lem:isometry}(\ref{it:iso}), a maximal $(\sqrt r/3)$-separated subset of $\{p_u:u\in\sfV\}$ is a packing at pairwise tree distances at least $r/9$ and, by maximality, a cover of the same set at Euclidean radius $\sqrt r/3$. Each covering ball yields one ancestor: the deepest common ancestor of the vertices it captures lies within tree distance $\tfrac49\,r$ of each of them, because every captured vertex below it meets some other captured vertex exactly there. Collecting one ancestor per ball gives an ancestor $r$-net, so the balls, and with them the separated vertices, number at least $\Nup(r)$ (\Cref{ssec:packing}).
\end{proof}

A product form of Fano's inequality extracts risk from overlapping hypotheses. The construction in this subsection supplies all three assumptions: blockwise scalability is legality under shrinking toward a reference point, the aspect bound rules out degenerate alphabets, and the Gram row sum is the one quantity through which arbitrary overlap enters.

\begin{lemma}[Fano's inequality for overlapping blocks]\label{lem:fano}
Let $K\subseteq\bbR^N$ be convex, let $\mathcal A_1,\dots,\mathcal A_J$ be finite alphabets with $|\mathcal A_j|\ge2$, and let
\[
\mu_{\mathbf a}\ =\ \mu_0+\sum_{j=1}^Jg_j(a_j),
\qquad
\mathbf a\in\textstyle\prod_j\mathcal A_j,
\]
be a family of points of $K$. Write $d_j^2:=\min_{a\ne b}\|g_j(a)-g_j(b)\|_2^2$, $D_j^2:=\max_{a,b}\|g_j(a)-g_j(b)\|_2^2$, and $h_j:=\log|\mathcal A_j|$, and assume, for constants $A\ge1$ and $\gamma<1$:
\begin{enumerate}
\item\label{it:scal} $\mu_0+\sum_j\tau_jg_j(a_j)\in K$ for every
$\tau\in[0,1]^J$ and every $\mathbf a$;
\item\label{it:aspect} $0<D_j^2\le A\,d_j^2$ for every $j$;
\item\label{it:gram} every family of nonzero contrasts
$\Delta_j=g_j(a_j)-g_j(b_j)$ satisfies
$\ \sup_i\sum_{j\ne i}
\frac{|\langle\Delta_i,\Delta_j\rangle|}
{\|\Delta_i\|_2\,\|\Delta_j\|_2}\le\gamma$.
\end{enumerate}
Then the minimax risk over $K$ in the model $Y\sim N(\mu,\sigma^2I_N)$, $\mu\in K$, satisfies
\[
R^*(K,\sigma)\ \ge\ c_A(1-\gamma)\sum_{j=1}^J
\min\bigl\{d_j^2,\ \sigma^2h_j\bigr\},
\]
where $c_A>0$ depends only on $A$. The bound holds already for the Bayes risk of an explicit product prior on a blockwise-shrunken subfamily; and if every contrast $g_j(a)-g_j(b)$ is supported in a common coordinate set $S$, it holds for the loss restricted to $S$.
\end{lemma}

\begin{proof}[Proof sketch]
For a small universal $\eta$, shrink each block into its information budget by $\tau_j^2:=\min\{1,\,\eta\sigma^2h_j/D_j^2\}$, legal by (\ref{it:scal}), and place the uniform product prior on $\mathbf a$. The scale-invariant Gram bound (\ref{it:gram}) separates any two shrunken members by $(1-\gamma)\sum_j\tau_j^2\|\Delta_j\|_2^2$, overlap paid for once through $\gamma$. Revealing all other labels reduces block $j$ to a finite Gaussian test of Kullback--Leibler diameter $O(\eta h_j)$, on which Fano's inequality \citep{tsybakov2009}, or a direct binary test, makes every decoder err with probability at least $\tfrac14$. Decoding to the nearest shrunken member makes an error at block $j$ cost squared loss of order $(1-\gamma)\tau_j^2d_j^2\ge(\eta/A)(1-\gamma)\min\{d_j^2,\sigma^2h_j\}$ by (\ref{it:aspect}). Taking expectations assembles the bound; the per-block constants, the binary-alphabet test, and the restricted-loss variant are in \Cref{ssec:packing}.
\end{proof}

The next proposition realizes these hypotheses inside the flow polytope.

\begin{prop}[Positive packing]\label{prop:packing}
Let $U\subseteq\sfV$ consist of $M\ge2$ vertices with pairwise tree
distances at least $r>0$. Then for every integer $1\le q\le M/2$,
\[
R^*_T(V,\sigma)\ \ge\ c\,\min\Bigl\{\frac{V^2r}q,\
\sigma^2q\log\frac{eM}q\Bigr\}.
\]
\end{prop}

\begin{proof}[Proof sketch]
The construction distributes the budget exactly and lets \Cref{lem:fano} price the overlap; the proof is completed in \Cref{ssec:packing}.

\emph{Disjoint budgets.} A bin-packing pass from the leaves of the minimal subtree spanning $U$ upward, sealing a connected component whenever it has collected on the order of $M/q$ marked vertices, extracts from $U$ marked sets $U_1,\dots,U_{q'}$ of sizes $M_i\asymp M/q$ lying in connected subtrees $\mathcal C_1,\dots,\mathcal C_{q'}$ with pairwise disjoint edge sets, where $q'\asymp q$ (\Cref{ssec:packing}). Give every group the budget $a:=V/q'$; the budgets sum to $V$.

\emph{A multiscale code in each group.} Fix a group with marked set $U_i$. Set $r_\ell:=16^\ell r$ and let $N_\ell\subseteq N_{\ell-1}$ be nested maximal $r_\ell$-separated subsets of $N_0:=U_i$; assigning each point of $N_{\ell-1}$ to a point of $N_\ell$ within distance $r_\ell$ organizes $U_i$ into a hierarchy with a singleton top. Descend from the top along children of maximal weight, the weight of a node being the number of members of $U_i$ at or below it, so that the child counts along the chain multiply to at least $M_i$. Discarding single-child levels and keeping one residue class $I$ of the rest modulo four produces alphabets $\mathcal A_\ell$, $\ell\in I$, the children of the chain node at level $\ell$, with $h_\ell:=\log|\mathcal A_\ell|$ and total entropy $H:=\sum_{\ell\in I}h_\ell\ge\tfrac14\log M_i$. Two points of $\mathcal A_\ell$ are $r_{\ell-1}$-separated and within $2r_\ell$ of each other, so each block has aspect at most $32$ (\Cref{lem:isometry}(\ref{it:iso})). The per-scale masses $a_\ell:=a\sqrt{h_\ell/r_{\ell-1}}/S$, with normalizer $S:=\sum_{\ell'\in I}\sqrt{h_{\ell'}/r_{\ell'-1}}$, spend $\sum_{\ell\in I}a_\ell=a$ exactly, and Cauchy--Schwarz against the geometric radii gives
\[
\sum_{\ell\in I}\min\{a_\ell^2r_{\ell-1},\,\sigma^2h_\ell\}\ \ge\ c\min\{a^2r,\,\sigma^2\log M_i\}
\]
(\Cref{ssec:packing}).
Fix a reference letter $u^0_\ell\in\mathcal A_\ell$ at each retained scale; the group's hypotheses are $\mu_{\mathbf u}=\mu_0+\sum_{\ell\in I}g_\ell(u_\ell)$ with $\mu_0:=\sum_{\ell\in I}a_\ell\,p_{u^0_\ell}$ and $g_\ell(u):=a_\ell(p_u-p_{u^0_\ell})$, so that $\mu_{\mathbf u}=\sum_{\ell\in I}a_\ell p_{u_\ell}$ places mass $a_\ell$ at the letter chosen from $\mathcal A_\ell$.

\emph{Overlap.} A level-$\ell$ contrast has entries bounded by $a_\ell$ and support of size at most $32\,r_{\ell-1}$ (\Cref{lem:isometry}(\ref{it:edge})), so its normalized inner product with a level-$\ell'$ contrast, $\ell<\ell'$ in $I$, is at most $32\cdot4^{-(\ell'-\ell)}$; retained levels differ by at least four, so every normalized Gram row sums to at most $\gamma=64/255<\tfrac12$.

\emph{Assembly.} Each blockwise contraction $\mu_0+\sum_\ell\tau_\ell g_\ell(u_\ell)$ replaces $p_{u^0_\ell}$ by a convex combination of $p_{u^0_\ell}$ and $p_{u_\ell}$, so it remains a nonnegative leak combination of total mass $a$; \Cref{lem:fano} therefore applies with $A=32$ and $\gamma=64/255$ and bounds the group's Bayes risk from below by $c\min\{a^2r,\sigma^2\log M_i\}$ in the group's own coordinates. Sums of one member per group are legal points of $\mathcal F_V(T)$ (\Cref{lem:normal}), and the groups' contrasts occupy the pairwise disjoint edge supports of the $\mathcal C_i$ (\Cref{lem:isometry}(\ref{it:edge})), so the joint Bayes risk dominates the sum of the isolated group bounds (\Cref{ssec:packing}). With $q'\asymp q$ and $M_i\asymp M/q$, the group bounds sum to $R^*_T(V,\sigma)\ge c\,q\min\{V^2r/q^2,\,\sigma^2\log(eM/q)\}$, which is the claim.
\end{proof}

The lower half of the crossing now follows by matching the free parameter $q$ to the profile.

\begin{prop}[Profile to risk]\label{prop:profilerisk}
There is a universal constant $c_0>0$ such that for every finite
rooted tree, every integer $k\ge1$, and all $V,\sigma>0$,
\[
V^2\delta_k(T)\ \ge\ A_0\sigma^2k
\qquad\Longrightarrow\qquad
R^*_T(V,\sigma)\ \ge\ c_0\,\sigma^2k .
\]
\end{prop}

\begin{proof}
The supremum in \Cref{eq:profile} need not be attained, so choose $r>0$ whose term is at least half the supremum:
\[
V^2r\,\min\Bigl\{1,\ \tfrac1k\Bigl[\log\tfrac{\Nup(r)}k\Bigr]_+\Bigr\}
\ \ge\ \tfrac12\,V^2\delta_k(T)\ \ge\ \tfrac{A_0}2\,\sigma^2k\ >\ 0 .
\]
Positivity forces $\Nup(r)>k$. \Cref{lem:bridge} supplies $M\ge\Nup(r)>k$ vertices with pairwise tree distances at least $r':=r/9$, and the displayed inequality survives with $(r',M)$ in place of $(r,\Nup(r))$ at the cost of a factor $18$, absorbed into $A_0$. It remains to choose $q$ in \Cref{prop:packing}: the mass branch $V^2r'/q$ shrinks and the entropy branch $\sigma^2q\log(eM/q)$ grows as $q$ increases, and at $q\asymp\max\bigl\{1,\ k/(1+\log(M/k))\bigr\}$ both sit above $\sigma^2k$ times a universal constant. This scalar matching, including the boundary cases, is proved in \Cref{ssec:packing}.
\end{proof}

\subsection{The Crossing and the Benchmarks}\label{sec:crossing}

The two halves now meet. Throughout, $K:=\lfloor n/C_\star\rfloor$, so that $k_0\le K+1$ by \Cref{eq:crossing}.

\begin{proof}[Proof of \Cref{thm:rate,prop:fixedpoint}]
The upper half of \Cref{prop:fixedpoint} is \Cref{eq:genericupper}; it remains to prove
\[
\text{(I)}\quad
R^*_T\ \ge\ c\,\bigl(\rho_T\wedge V^2H_T\wedge\sigma^2n\bigr),
\qquad\qquad
\text{(II)}\quad
R^*_T\ \asymp\ \min\bigl\{V^2H_T,\ \sigma^2k_0\bigr\}.
\]
For $n=1$ the body is the single point $Vp_o$ and $H_T=0$, so $R^*_T=0$ and both sides of (I) and (II) vanish; assume $n\ge2$. Testing two points at distance $\min\{V\sqrt{H_T},\sigma\}$ along the diameter segment of the body, which lies in $\mathcal F_V(T)$ by \Cref{lem:isometry}(\ref{it:iso}), records a two-point bound
\begin{equation}\label{eq:twopoint}
R^*_T(V,\sigma)\ \ge\ c\,\min\bigl\{V^2H_T,\ \sigma^2\bigr\}
\qquad(n\ge2)
\end{equation}
(\Cref{ssec:packing}). The proof of (I) and (II) now splits into four exhaustive cases, according to how the crossing index fires.

\emph{Case $n<C_\star$.} Here $k_0=1$ through the dimension clause, and $\sigma^2\ge\sigma^2 n/C_\star$, so \Cref{eq:twopoint} gives (I) and the lower half of (II), while \Cref{eq:genericupper} gives the upper half (\Cref{ssec:packing}).

\emph{Case $n\ge C_\star$ and $V^2\delta_1\le A_0\sigma^2$.} Here $k_0=1$ through the crossing clause. \Cref{prop:cover} starts at $k=2$, so we supply the missing entropy bound at $k=1$. Every ancestor $r$-net with $r<h_T$ has at least two elements: the root lies in every net, being its own only ancestor, and it covers no deepest vertex. Hence $\Nup(r)\ge2$ for all $r<h_T$, and letting $r\uparrow h_T$ in \Cref{eq:profile} gives $\alpha_1\ge h_T\log2$. On the other hand the single center $Vp_o$ covers the body at squared radius $V^2h_T$, since $\|\mu-Vp_o\|_2\le\max_u\|Vp_u-Vp_o\|_2=V\sqrt{h_T}$ by convexity, so $H^c_T(V^2h_T)=0$ and
\[
\rho_T\ \le\ V^2h_T\ \le\ \frac{V^2\delta_1}{\log2}\ \le\
\frac{A_0}{\log2}\,\sigma^2 .
\]
Both (I) and (II) follow, with \Cref{eq:twopoint} for the lower bounds and \Cref{eq:genericupper} for the upper ones (\Cref{ssec:packing}).

\emph{Case $2\le k_0\le K$.} The crossing clause fired at $k_0$: $V^2\delta_{k_0}\le A_0\sigma^2k_0$, and $k_0\le n/C_\star$, so \Cref{eq:rhocover} at $k_0$ gives
\begin{equation}\label{eq:rhointerior}
\rho_T\ \le\ C\bigl(V^2\delta_{k_0}+\sigma^2k_0\bigr)\ \le\
C(A_0+1)\,\sigma^2k_0 .
\end{equation}
By minimality neither clause holds at $k_0-1$; in particular $V^2\delta_{k_0-1}>A_0\sigma^2(k_0-1)$, so \Cref{prop:profilerisk} at $k_0-1$ gives
\begin{equation}\label{eq:packinterior}
R^*_T\ \ge\ c_0\,\sigma^2(k_0-1)\ \ge\ \frac{c_0}2\,\sigma^2k_0 .
\end{equation}
Combining \Cref{eq:rhointerior,eq:packinterior} proves (I). For (II) the same two inequalities identify the information branch: $V^2H_T\ge V^2h_T\ge V^2\delta_{k_0-1}>A_0\sigma^2(k_0-1)\ge \sigma^2k_0$ by \Cref{lem:toolkit}(\ref{it:mono}), so $\min\{V^2H_T,\sigma^2k_0\}=\sigma^2k_0$, which \Cref{eq:packinterior} bounds below and \Cref{eq:genericupper,eq:rhointerior} bound above.

\emph{Case $k_0=K+1$ with $n\ge C_\star$.} No clause fired on $[1,K]$; in particular $V^2\delta_K>A_0\sigma^2K$, so \Cref{prop:profilerisk} at $K$ gives $R^*_T\ge c_0\sigma^2K\ge c\,\sigma^2n$, using $K=\lfloor n/C_\star\rfloor\ge n/(2C_\star)$. This is (I), the three-term minimum being at most $\sigma^2n$; and (II) follows because $\sigma^2k_0\asymp\sigma^2n$ while $V^2H_T\ge V^2\delta_K>A_0\sigma^2K$ (\Cref{ssec:packing}).

The four cases are exhaustive: for $n\ge C_\star$, either the crossing clause fires at $k=1$, or it first fires at some $2\le k_0\le K$, or it never fires on $[1,K]$ and the dimension clause fires at $K+1$. This proves (I) and (II). In the interior case the computation also pins the fixed point itself, $\rho_T\asymp\sigma^2k_0$ (\Cref{ssec:packing}).
\end{proof}

The rate formula specializes in closed form on the three families of \Cref{tab:phases}; in each corollary the equivalence holds for all $V,\sigma>0$ with universal constants.

\begin{corollary}[Path]\label{cor:path}
Let $P_L$ be the path with $L\ge1$ edges, rooted at one endpoint, so that $\mathcal F_V(P_L)=\{\mu:V=\mu_0\ge\mu_1\ge\cdots\ge\mu_L\ge0\}$. Then
\[
R^*_{P_L}(V,\sigma)\ \asymp\
\min\bigl\{V^2L,\ \ V^{2/3}\sigma^{4/3}L^{1/3},\ \ \sigma^2(L+1)\bigr\}.
\]
\end{corollary}

\begin{corollary}[Star]\label{cor:star}
Let $S_m$ be the star with $m\ge2$ leaves rooted at the hub; the hub coordinate is pinned at $V$, and in leaf coordinates $\mathcal F_V(S_m)=\{\mu\in\bbR_{\ge0}^m:\sum_{i\le m}\mu_i\le V\}$. Then
\[
R^*_{S_m}(V,\sigma)\ \asymp\
\min\Bigl\{V^2,\ \ \sigma V\sqrt{1+\bigl[\log(m\sigma/V)\bigr]_+},\ \
\sigma^2m\Bigr\}.
\]
\end{corollary}

\begin{corollary}[Complete binary tree]\label{cor:binary}
Let $B_h$ be the complete binary tree of height $h\ge1$: every internal vertex has two children and all $2^h$ leaves lie at depth $h$, so that $n=2^{h+1}-1$. Then
\[
R^*_{B_h}(V,\sigma)\ \asymp\
\min\bigl\{V^2h,\ \ \sigma V\bigl(1+[\log(n\sigma/V)]_+\bigr),\ \
\sigma^2n\bigr\}.
\]
\end{corollary}

The three corollaries follow one recipe: the ancestor-covering counts determine the profile through \Cref{eq:discrete}, and the crossing selects the active branch. Interval covering gives the path counts $\lceil(L+1)/(q+1)\rceil$; on the star only the radius $q=0$ contributes, and the profile is the exact expression of \Cref{rem:star}; on the binary tree the counts decay geometrically in the radius, and, for $k\le n/C_\star$, the profile $\alpha_k$ is quadratic in $\log(n/k)$ wherever $k$ exceeds that logarithm. The counts, the profiles, and the scalar optimizations that locate $k_0$ and match the branches in every parameter regime are carried out in \Cref{ssec:benchmarks}.

The crossing consumed \Cref{prop:cover} only through \Cref{eq:entropycover}; the further certificates are spent in \Cref{sec:estimation}, the code length becoming a prior, the root sum an exact constraint, and the confined states a one-dimensional computation.

%%%%%%%%%%%%%%%%%%%%%%%%%%%%%%%%%%%%%%%%%%%%%%%%%%%%%%%%%%%%%%%%%%%%%%%%%%%%%%
\section{Efficient and Adaptive Estimation}\label{sec:estimation}
%%%%%%%%%%%%%%%%%%%%%%%%%%%%%%%%%%%%%%%%%%%%%%%%%%%%%%%%%%%%%%%%%%%%%%%%%%%%%%

This section proves the three estimation theorems; the first is stated here, the adaptation pair in \Cref{sec:adaptation}.

\begin{theorem}[Efficient estimation]\label{thm:efficient}
There is a deterministic estimator $\widehat\mu=\widehat\mu(T,V,\sigma,Y)$, computed exactly in $O(n^2\log(n+1))$ arithmetic operations of the real-RAM model fixed in \Cref{sec:polytope}, such that for every finite rooted tree $T$ and all $V,\sigma>0$,
\[
\sup_{\mu\in\mathcal F_V(T)}\E_\mu\bigl\|\widehat\mu-\mu\bigr\|_2^2
\ \le\ C\,R^*_T(V,\sigma)
\]
with a universal constant $C$.
\end{theorem}

The estimator is an exponentially weighted aggregate over a class of integer states, evaluated exactly in the sense fixed by \Cref{rem:exact}. \Cref{sec:aggregation} constructs it and proves the risk bound, \Cref{sec:evaluation} gives the algorithm, and \Cref{sec:adaptation} removes first the budget and then the noise level.

\subsection{Aggregation over Integer States}\label{sec:aggregation}

Least squares over the coded cover converts \Cref{eq:entropycover} into the rate, but it is a search through as many as $e^{C_1k}$ centers (\Cref{prop:cover}). The estimator of \Cref{thm:efficient} replaces the search by an average under a prior built from the code length.

Fix an integer $2\le k\le n/C_\star$ and recall the active support $A_k$, the charges $\widetilde\omega$, and the code length \Cref{eq:codelength}. The \emph{integer states} at budget $k$ are the vectors
\begin{equation}\label{eq:class}
\mathcal X_k\ :=\ \Bigl\{x\in\{0,1,\dots,3k\}^{\sfV}:\ x_o=k
\ \text{ and }\
z_v:=x_v-\textstyle\sum_{c\in\ch(v)}x_c=0\ \text{ for all }v\notin
A_k\Bigr\},
\end{equation}
each with its leak vector $z$, its code length $\Gamma(x):=\Gamma(z)$, and its center $\nu_x:=\frac Vk\,x$; the leaks telescope to the root state, $\sum_vz_v=x_o=k$. The approximant of \Cref{prop:cover} is an integer state: its coordinates lie in $[0,3k]$, its root coordinate is $k$, and its leaks are supported on $A_k$. The class relaxes the coded cover in exactly one respect: the budget $\Gamma\le9k$ is dropped from the constraints, so that $\mathcal D_k=\{\nu_x:\ x\in\mathcal X_k,\ \Gamma(x)\le9k\}$. The dropped budget resurfaces as the prior: as a constraint it would add a second dimension to the recursions of \Cref{sec:evaluation}; as a soft penalty it costs only a normalizing sum. Their leaks may be negative, as for the coded cover.

The estimator at budget $k$ is the exponentially weighted aggregate
\begin{equation}\label{eq:aggregate}
\widehat\mu_k\ :=\
\frac{\sum_{x\in\mathcal X_k}\nu_x\,e^{-2\Gamma(x)}\,
e^{-\|Y-\nu_x\|_2^2/(4\sigma^2)}}
{\sum_{x\in\mathcal X_k}e^{-2\Gamma(x)}\,
e^{-\|Y-\nu_x\|_2^2/(4\sigma^2)}}\ ,
\end{equation}
the posterior mean at temperature $4\sigma^2$ under the Kraft-type prior $\pi_k(x)\propto e^{-2\Gamma(x)}$. The normalizing sum $\sum_{x\in\mathcal X_k}e^{-2\Gamma(x)}$, which cancels from \Cref{eq:aggregate} and is never computed, is at most $e^{k}$, by a generating-function estimate from the net sizes of \Cref{sec:cover} (\Cref{ssec:evaluation}). Every integer state has root coordinate $k$ and coordinates at most $3k$, so $\widehat\mu_k(o)=V$ identically and every coordinate of $\widehat\mu_k$ lies in $[0,3V]$.

The risk of an exponentially weighted aggregate is controlled through the prior mass at a single candidate; no cardinality of the class enters \citep[cf.][]{leungbarron2006}. In the next lemma the true mean is arbitrary.

\begin{lemma}[Aggregation oracle]\label{lem:oracle}
Let $F\subseteq\bbR^N$ be finite and nonempty, let $\pi$ be a probability vector on $F$ with $\pi_f>0$ for all $f$, and let $Y=\mu+\sigma Z$ with $Z\sim N(0,I_N)$ and $\mu\in\bbR^N$. Define
\[
\widehat f_\pi(Y)\ :=\ \sum_{f\in F}\widehat\pi_f(Y)\,f,
\qquad
\widehat\pi_f(Y)\ :=\ \frac{\pi_f\,e^{-\|Y-f\|_2^2/(4\sigma^2)}}
{\sum_{g\in F}\pi_g\,e^{-\|Y-g\|_2^2/(4\sigma^2)}}\ .
\]
Then
\[
\E_\mu\bigl\|\widehat f_\pi-\mu\bigr\|_2^2\ \le\ \min_{f\in F}
\Bigl\{\|f-\mu\|_2^2+4\sigma^2\log\frac1{\pi_f}\Bigr\}.
\]
\end{lemma}

\begin{proof}[Proof sketch]
Two exact identities. Differentiating the weights shows that the aggregate is smooth and bounded and identifies its Jacobian as $\frac1{2\sigma^2}$ times the posterior covariance of $f$, so Stein's identity \citep{stein1981} writes the risk through the divergence, a covariance trace; the bias--variance split of the posterior loss $\sum_f\widehat\pi_f(y)\,\|y-f\|_2^2$ produces the same trace with the opposite sign, and at this temperature the two cancel exactly, leaving $\E_\mu\|\widehat f_\pi-\mu\|_2^2=\E\sum_f\widehat\pi_f(Y)\,\|Y-f\|_2^2-N\sigma^2$. The weights are the Gibbs minimizer of the posterior loss penalized by $4\sigma^2$ times the relative entropy to $\pi$ \citep[see, e.g.,][Section~4.9]{boucheron2013}, so the value of the penalized objective at the point mass on any $f_0\in F$ bounds the posterior loss pathwise by $\|Y-f_0\|_2^2+4\sigma^2\log(1/\pi_{f_0})$; take expectations and insert $\E\|Y-f_0\|^2=\|f_0-\mu\|^2+N\sigma^2$. The differentiation, the domain of Stein's identity, and the two expansions are carried out in \Cref{ssec:evaluation}.
\end{proof}

Set $K:=\lfloor n/C_\star\rfloor$ as in \Cref{sec:crossing}. From here through \Cref{sec:evaluation}, set $k:=\kalg(T,V,\sigma)$, the algorithmic crossing index of \Cref{lem:toolkit}(\ref{it:surrogate}), so that $k\le K+1$. The root coordinate carries no statistical content (\Cref{sec:polytope}), so define the root-corrected observation $\widetilde Y$ by $\widetilde Y(o):=V$ and $\widetilde Y(v):=Y(v)$ for $v\ne o$. The estimator of \Cref{thm:efficient} is
\begin{equation}\label{eq:branches}
\widehat\mu\ :=\
\begin{cases}
Vp_o, & V^2H_T\le\sigma^2k,\\
\widetilde Y, & V^2H_T>\sigma^2k\ \text{ and }\ k>K,\\
Vp_o, & V^2H_T>\sigma^2k\ \text{ and }\ k=1\le K,\\
\widehat\mu_{k}, & V^2H_T>\sigma^2k\ \text{ and }\ 2\le k\le K,
\end{cases},
\end{equation}
four exclusive and exhaustive branches: when the diameter is not the active scale, the index $k$ either exceeds the dimension range, or is too small for \Cref{prop:cover}, or is interior.

The risk bound of \Cref{thm:efficient} follows. In the interior branch \Cref{prop:cover} applies, and evaluating the oracle of \Cref{lem:oracle} at the approximant $x^\star$ of a given $\mu\in\mathcal F_V(T)$, with $\log(1/\pi_k(x^\star))\le2\Gamma(x^\star)+k\le19k$, gives
\begin{equation}\label{eq:interior}
\sup_{\mu\in\mathcal F_V(T)}
\E_\mu\bigl\|\widehat\mu_k-\mu\bigr\|_2^2\ \le\
C_0V^2\delta_k(T)+C\sigma^2k\ \le\
C\,\sigma^2k\ =\ C\min\{V^2H_T,\ \sigma^2k\},
\end{equation}
the second inequality because the surrogate crossing clause fired at $k=\kalg\le K$, so that $V^2\delta_k\le V^2\overline\delta_k\le2A_0\sigma^2k$, and the equality by the branch condition. The three elementary branches obey the same bound by the diameter and first-budget facts of \Cref{sec:crossing} (\Cref{ssec:evaluation}). Since $\kalg\le k_0$ by \Cref{lem:toolkit}(\ref{it:surrogate}), $\min\{V^2H_T,\sigma^2\kalg\}\le\min\{V^2H_T,\sigma^2k_0\}$, which \Cref{thm:rate} bounds by $C\,R^*_T(V,\sigma)$.

\subsection{Exact Evaluation}\label{sec:evaluation}

It remains to evaluate \Cref{eq:aggregate}, a ratio of sums over a class of exponential size. Computing $\kalg$ and the nets behind $A_k$ costs $O(n\log(n+1))$ operations (\Cref{lem:toolkit}), and the elementary branches are linear; the content is the interior branch.

The confined states are the mechanism. Every coordinate of an integer state lies in $\{0,\dots,3k\}$, so subtree weights factorize over one integer state per vertex. Write $\varphi_v(x):=e^{-(Y_v-(V/k)x)^2/(4\sigma^2)}$ for the local weight, so that the summand of \Cref{eq:aggregate} at $x$ is $e^{-2\Gamma(x)}\prod_v\varphi_v(x_v)$, and define for each vertex the \emph{message}
\begin{equation}\label{eq:message}
P_v(\zeta)\ :=\ \sum_{x=0}^{3k}P_v[x]\,\zeta^x,
\end{equation}
a polynomial of degree at most $3k$ in an indeterminate $\zeta$, whose coefficient $P_v[x]$ is the total weight, local weights times prior kernel, of the admissible assignments on the subtree $T_v$ with state $x$ at $v$. The message at $v$ arises from the product $\prod_{c\in\ch(v)}P_c(\zeta)$ of its children's messages by pairwise polynomial multiplications and an $O(k)$ local pass, and correctness is an induction over subtrees; since every integer state has root coordinate $k$, the denominator of \Cref{eq:aggregate} is the single coefficient $P_o[k]$.

A child sum can exceed $3k$ where a state cannot. At an inactive vertex conservation kills such terms; at an active vertex the prior kernel is geometric in the leak, so they enter through one scalar, the child product evaluated at $e^{-2}$; evaluations multiply where messages multiply, so the scalar travels with each message, and no coefficient beyond degree $3k$ outlives a merge (\Cref{ssec:evaluation}).

The tree off the active support carries no state: a subtree disjoint from $A_k$ contributes a constant factor and is pruned, and a maximal chain of inactive vertices, each left with one child whose state it copies, compresses to a single edge whose diagonal kernel is generated in $O(k)$ from the chain's length and observation sum.

The numerator of \Cref{eq:aggregate} is recovered by differentiation. Treat the $3k+1$ local weights at each vertex as formal inputs of the arithmetic circuit that computes $P_o[k]$; every admissible assignment carries the factor $\varphi_v(x_v)$ exactly once, so
\begin{equation}\label{eq:marginal}
\widehat\mu_k(v)\ =\ \frac Vk\,\sum_{x=0}^{3k}x\cdot
\frac{\varphi_v(x)}{P_o[k]}\,
\frac{\partial P_o[k]}{\partial\varphi_v(x)}\ ,
\end{equation}
the ratio under the sum being the posterior probability of $\{x_v=x\}$. One reverse-mode sweep of the circuit computes all these derivatives at once, at the same asymptotic cost as the forward pass (\Cref{ssec:evaluation}).

Each vertex contributes $O(k)$ local work, each of the fewer than $n$ merges costs one multiplication of degree-$3k$ polynomials, $O(k\log(k+1))$ operations in the model of \Cref{sec:polytope}, and the reverse sweep matches the forward cost, so with the preprocessing the total is $O(n\log(n+1)+nk\log(k+1))\le O(n^2\log(n+1))$ operations, the count of \Cref{thm:efficient}. On structured trees the account improves: sibling messages that agree up to a dilation of the variable multiply in batch through power sums, and on stars the entire evaluation takes $O(n\log^2 n)$ operations. Every tie break is fixed, so the estimator is a deterministic function of $(T,V,\sigma,Y)$; together with the risk bound of \Cref{sec:aggregation}, this proves \Cref{thm:efficient}. The evaluation identities, the operation count, and the instance-sensitive refinements are proved in \Cref{ssec:evaluation}.

\begin{remark}[Computational model]\label{rem:exact}
``Exact'' refers to the estimator itself, not only to its risk order: in the arithmetic model fixed in \Cref{sec:polytope}, the messages, the overflow scalars, the compressed kernels, and the reverse sweep are algebraic identities, and the returned vector is \Cref{eq:aggregate} with $k=\kalg$, evaluated without error.
\end{remark}

\subsection{Adaptation}\label{sec:adaptation}

The budget is not external to the model: on the cone \Cref{eq:cone} it is the root value of the signal, observed directly as $Y_o=V+\sigma Z_o$, so adaptation to it is cheap but not free.

\begin{theorem}[Budget adaptation]\label{thm:budget}
There are universal constants $c,C>0$ and a deterministic polynomial-time estimator $\widehat\mu=\widehat\mu(T,\sigma,Y)$, not depending on $V$, such that for every finite rooted tree $T$, all $V,\sigma>0$, and every $\mu\in\mathcal F_V(T)$,
\[
\E_\mu\|\widehat\mu-\mu\|_2^2\ \le\
C\bigl(R^*_T(V,\sigma)+\sigma^2\bigr).
\]
Conversely, every estimator $\widehat\mu(T,\sigma,Y)$ satisfies, for every finite rooted tree $T$ and all $V,\sigma>0$,
\[
\max_{V'\in\{V,\,V+\sigma\}}\ \sup_{\mu\in\mathcal F_{V'}(T)}
\E_\mu\|\widehat\mu-\mu\|_2^2\ \ge\ c\,\sigma^2 .
\]
\end{theorem}

\begin{proof}[Proof of \Cref{thm:budget}, necessity]
The points $Vp_o$ and $(V+\sigma)p_o$ lie in the slices at budgets $V$ and $V+\sigma$ and are at distance $\sigma$. For any estimator $\widehat\mu(T,\sigma,Y)$, decode the nearer of the two points: a decoding error forces squared loss at least $\sigma^2/4$, and deciding between two Gaussian means at distance $\sigma$ is a one-dimensional shift of size one, on which every test errs with probability at least a universal constant under the uniform two-point prior. The larger of the two risks is therefore at least $c\,\sigma^2$.
\end{proof}

The sufficiency halves of \Cref{thm:budget,thm:full} rest on candidate families free of the unknown parameters. The construction behind \Cref{prop:cover} consumes the budget in exactly one place, the quantization of leaks in units of $V/k$; dropping the quantization leaves subspaces. For an integer $2\le k\le n/C_\star$, call $S\subseteq A_k\setminus R_0$ \emph{admissible} when $\sum_{v\in S}\widetilde\omega(v)\le4k$, and define the \emph{profile subspaces} $\mathcal V_{k,S}:=\operatorname{span}\{p_v:\ v\in R_0\cup S\}$, of dimension $O(k)$. Everything here depends on $(T,k)$ alone, and the admissible supports number at most $e^{C_1k}$, by the counting argument of \Cref{prop:cover}. Every body point is close to one of these subspaces at the profile scale:
\begin{equation}\label{eq:subspace}
\min_{S\ \mathrm{admissible}}\
\operatorname{dist}^2\bigl(\mu,\ \mathcal V_{k,S}\bigr)
\ \le\ C\,V^2\overline\delta_k(T)
\qquad\text{for every }\mu\in\mathcal F_V(T),\
2\le k\le n/C_\star .
\end{equation}
The unrounded skeleton and the sampled exits of the proof of \Cref{prop:cover} realize the bound, their support carrying total charge at most $4k$ (\Cref{ssec:adaptation}).

Selection is by penalized least squares with Kraft weights. \Cref{ssec:adaptation} proves the following oracle, of Birg\'e--Massart type \citep{birgemassart2001}, in the weighted and random-penalty form needed here: for Gaussian data on $\bbR^N$ with independent coordinates of variances at most $\sigma^2$ and mean $\theta$, a countable collection $\mathcal M$ of affine models $m$ with weights $\Delta_m$ obeying $\sum_me^{-\Delta_m}\le1$, penalties $C_2\sigma^2(\dim m+\Delta_m)$ for a sufficiently large universal $C_2$, and a full model penalized at $C_2\sigma^2N$, every minimizer $\widehat\theta$ of the criterion obeys
\begin{equation}\label{eq:selection}
\E\bigl\|\widehat\theta-\theta\bigr\|_2^2\ \le\
C\,\min\Bigl\{\ \inf_{m\in\mathcal M}
\bigl(\operatorname{dist}^2(\theta,m)
+\sigma^2(\dim m+\Delta_m+1)\bigr),\ \ \sigma^2N\Bigr\}
\ +\ C\sigma^2 .
\end{equation}
If the penalties are random, nonnegative, and bracketed on an event $E$ between fixed multiples of the displayed ones, the same right side bounds $\E\bigl[\|\widehat\theta-\theta\|_2^2\ind_E\bigr]$; this is how the estimated noise scale enters in \Cref{thm:full}.

For $n$ below a universal threshold the estimator of \Cref{thm:budget} returns $Y$, at risk $O(\sigma^2)$; above it, the estimator is one such selection, conditional on the root observation, over three groups of candidates: the amplitude-free subspaces $\mathcal V_{k,S}$ at dyadic budgets $2\le k\le k^\star:=\lfloor\log n/(2C_1)\rfloor$; the integer-state points $\tfrac{V'}k\,x$, $x\in\mathcal X_k$, at the $O(\log n)$ amplitudes $V'$ of a root-anchored geometric grid and a low-amplitude dyadic grid, with $\operatorname{span}(p_o)$ and the root points $V'p_o$ adjoined; and the full model, each candidate weighted by its dimension and code length. The system depends on $Y$ through $Y_o$ alone, so conditionally on $Y_o$ it is deterministic while the noise in the remaining coordinates is unchanged, and unconditioning costs $\E(Y_o-V)^2=\sigma^2$; the precise grids, weights, and thresholds are fixed in \Cref{ssec:adaptation}.

\begin{proof}[Proof sketch of \Cref{thm:budget}, sufficiency]
By \Cref{thm:rate} and the sandwich of \Cref{lem:toolkit}(\ref{it:surrogate}), $\min\{V^2H_T,\sigma^2\kalg\}\asymp R^*_T(V,\sigma)$, and there are three regimes, according to the candidate at which the oracle \Cref{eq:selection} is evaluated. In the diameter regime $V^2H_T\le\sigma^2\kalg$, the model $\operatorname{span}(p_o)$ is within squared distance $V^2h_T$ of $\mu$, at constant dimension and weight. In the entropy regime with small crossing, $\kalg\le k^\star/2$, the smallest dyadic $k\ge\max\{\kalg,2\}$ stays below $k^\star$, and \Cref{eq:subspace} with the surrogate clause fired at $\kalg$ gives an amplitude-free admissible subspace within squared distance $C\sigma^2\kalg$ of $\mu$, at dimension plus weight $O(\kalg)$, so no quantization is paid. In the entropy regime with large crossing, $\kalg>k^\star/2$, the grids take over: for $V\ge\sigma/2$, outside an event of loss contribution $O(\sigma^2)$, the anchored grid contains an amplitude $V'\ge V$ with $\E(V'/V)^2=O(1)$ and index price $O(\sigma^2)$, and, according to the regime that $V'$ sees, the approximant of \Cref{prop:cover} at the dyadic budget nearest its crossing, the root candidate, or the full model is a candidate whose oracle value is $C(R^*_T(V,\sigma)+\sigma^2)$ in expectation, by the scale bound $R^*_T(cV,\sigma)\le Cc^2R^*_T(V,\sigma)$ for $c\ge1$; for $V<\sigma/2$ the low-amplitude grid plays the same role, its index price absorbed because $\kalg>k^\star/2$ places the rate above $c\,\sigma^2\log n$ (\Cref{ssec:adaptation}). The additive $\sigma^2$ collects the root replacement, the anchored-grid prices, and the constant floor in \Cref{eq:selection}; the two families are complementary, amplitude quantization being expensive only at small crossing, where the subspaces are few enough to enumerate. The selection runs in deterministic polynomial time, each grid candidate's inner minimization over $\mathcal X_k$ being a min-plus analogue of \Cref{sec:evaluation}; the operation counts, the grid moment bounds, and the conditioning argument are completed in \Cref{ssec:adaptation}.
\end{proof}

It remains to remove the noise level, which enters the selection only through the penalties; these tolerate a constant-factor error, so a robust estimate suffices. The regime in which the estimate succeeds is described by a width functional of the tree. At each internal vertex of $T$ fix a child $c$ maximizing the number of vertices $u$ with $c\preceq u$, breaking ties by a fixed rule; call the edge to that child \emph{heavy}, every other child edge \emph{light}, and a vertex \emph{branching} when it has at least two children. Set $w_T:=w_{\mathrm{lt}}(T)+w_{\mathrm{br}}(T)$, where $w_{\mathrm{lt}}(T)$ is one plus the maximum number of light edges on a root-to-leaf path, and $w_{\mathrm{br}}(T)$ is the maximum number of branching vertices on such a path.

\begin{theorem}[Full adaptation]\label{thm:full}
There are universal constants $c,C>0$ and a deterministic polynomial-time estimator $\widehat\mu=\widehat\mu(T,Y)$, depending on neither $V$ nor $\sigma$, such that for every finite rooted tree $T$ and all $V,\sigma>0$,
\[
\sup_{\mu\in\mathcal F_V(T)}\E_\mu\|\widehat\mu-\mu\|_2^2\ \le\
C\bigl(R^*_T(V,\sigma)+\sigma^2\bigr)
\qquad\text{whenever}\quad Vw_T\le c\sigma n .
\]
\end{theorem}

For $n$ below a universal threshold the estimator behind \Cref{thm:full} returns $Y$, at risk $O(\sigma^2)$; above it, the estimator rounds a median. The tree supplies at least $(n-1)/2$ difference statistics: heavy-edge and sibling differences, each Gaussian of variance $2\sigma^2$ around a signal difference. Monotonicity caps the contamination collectively: heavy differences telescope along vertex-disjoint heavy paths to their tops, sibling excesses charge their branching vertices, and a root path meets at most $w_T$ of these sites; for every $\varepsilon>0$, at most $Vw_T/(\varepsilon\sigma)$ statistics therefore have mean magnitude at least $\varepsilon\sigma$. Under $Vw_T\le c\sigma n$, the normalized median of the absolute values of an independent subfamily then lands in $[\sigma/2,\,2\sigma]$ except with probability $2e^{-cn}$; rounded upward to a power of two, the estimate takes at most four values on that event, the penalties of \Cref{thm:budget} stay bracketed at the estimated scale, a constant weight shift pays for the union over the four values, and the off-event loss contributes $O(\sigma^2)$. The statistics, the contamination count, the concentration of the median, and the moment bookkeeping are carried out in \Cref{ssec:adaptation}.

Uniform recovery of the noise level has a ceiling: once monotone flows with root value of order $\sigma n\sqrt{\log n}$ are admitted, every estimator of $\sigma$ fails, with constant probability, to land strictly within a factor $\sqrt2$ of the truth (\Cref{ssec:adaptation}). On families of bounded width, stars and brooms among them, the scale-based route of \Cref{thm:full} thus operates within a $\sqrt{\log n}$ factor of its ceiling.

%%%%%%%%%%%%%%%%%%%%%%%%%%%%%%%%%%%%%%%%%%%%%%%%%%%%%%%%%%%%%%%%%%%%%%%%%%%%%%
\section{The Suboptimality of Least Squares}\label{sec:lse}
%%%%%%%%%%%%%%%%%%%%%%%%%%%%%%%%%%%%%%%%%%%%%%%%%%%%%%%%%%%%%%%%%%%%%%%%%%%%%%

The last theorem prices the default estimator. The body is compact and convex, so the data have a unique Euclidean projection onto it, the \emph{constrained least squares estimator}
\[
\widehat\mu_{\mathrm{LSE}}\ :=\ \Pi_{\mathcal F_V(T)}(Y)\ =\
\operatorname*{arg\,min}_{\theta\in\mathcal F_V(T)}\
\bigl\|Y-\theta\bigr\|_2^2,
\]
the solution of a tuning-free convex program; write
\[
R^{\mathrm{worst}}_{\mathrm{LSE}}(T,V,\sigma)\ :=\
\sup_{\mu\in\mathcal F_V(T)}
\E_\mu\bigl\|\widehat\mu_{\mathrm{LSE}}-\mu\bigr\|_2^2
\]
for its worst-case risk.

\begin{theorem}[Least squares]\label{thm:lse}
There is a universal constant $C$ such that for every finite rooted tree with $n\ge2$ and all $V,\sigma>0$,
\[
R^{\mathrm{worst}}_{\mathrm{LSE}}(T,V,\sigma)\ \le\
C\,\bigl(1+\log(en)\bigr)^{2/5}\,R^*_T(V,\sigma).
\]
Conversely, there are universal constants $c>0$ and $n_0$ such that the following holds for every integer $n\ge n_0$. Let $L$ be the largest integer with $e^{20L^{5/2}}\le\sqrt n$, and let $T_L$ be the broom on $n$ vertices whose handle $o=v_0,v_1,\dots,v_L$ is a path of $L$ edges and whose endpoint $v_L$ has $m:=n-L-1$ leaf children. Then, at budget $V=L^{1/4}$ and noise level $\sigma=1$,
\[
R^{\mathrm{worst}}_{\mathrm{LSE}}\bigl(T_L,L^{1/4},1\bigr)\ \ge\
c\,(\log n)^{2/5}\,R^*_{T_L}\bigl(L^{1/4},1\bigr).
\]
\end{theorem}

We prove the broom lower bound first: the broom isolates the mechanism, extreme-value forcing through a terminal star, and every numerical threshold in the forcing step is explicit (\Cref{ssec:benchmarks}). The universal upper bound shows the resulting separation extremal over all trees and parameters; its proof, sketched at the end of the section, runs the coded cover of \Cref{sec:cover} through a localized-width analysis.

The broom couples the two classical geometries in series, a path feeding a star. For integers $L\ge1$ and $m\ge1$, the broom $T_L$ consists of a \emph{handle}, the path $o=v_0,v_1,\dots,v_L$ of $L$ edges, and of $m$ \emph{terminal leaves} $w_1,\dots,w_m$, the children of $v_L$, so that $n=L+1+m$ and $h_{T_L}=H_{T_L}=L+1$; the budget and the noise level are fixed at $V:=L^{1/4}$ and $\sigma:=1$.

\begin{proof}[Proof of \Cref{thm:lse}, lower bound]
Fix $n\ge n_0$, with $n_0$ a universal constant large enough for the estimates below. The maximality of $L$ gives $e^{20L^{5/2}}\le\sqrt n<e^{20(L+1)^{5/2}}$; for $n\ge n_0$ it forces $L\ge4$ and $L+1\le\sqrt n$, whence $m=n-L-1\ge n-\sqrt n\ge\sqrt n\ge e^{20L^{5/2}}$. The true signal is $\mu=Vp_o$: all mass rests at the root, and every nonroot coordinate vanishes. We show that one extreme leaf noise forces the projection to carry more than half the budget through the entire handle, at squared loss exceeding $L^{3/2}/4$, on an event of probability at least $3/8$; and that the minimax risk of the broom is of order $\sqrt L$.

\emph{Coordinates.} Write $h_j$ for the coordinate of a body point at $v_j$ and $\ell_i$ for its coordinate at $w_i$. By \Cref{lem:normal}, nonnegativity of the leaks identifies the body, in these coordinates, with the chain
\begin{equation}\label{eq:broomfeasible}
V\ \ge\ h_1\ \ge\ h_2\ \ge\ \cdots\ \ge\ h_L\ \ge\
\sum_{i=1}^{m}\ell_i,
\qquad
\ell_i\ \ge\ 0\quad(1\le i\le m).
\end{equation}
The root coordinate equals $V$ on the whole body, and at the true signal the nonroot data are pure noise: $Y_{v_j}=Z_{v_j}$ and $Y_{w_i}=Z_{w_i}$. Expanding $\|Y-\theta\|_2^2$ and discarding the terms that do not depend on $\theta$ shows that the nonroot coordinates of the projection maximize over \Cref{eq:broomfeasible} the objective
\begin{equation}\label{eq:broomobjective}
\Psi(h,\ell)\ :=\ 2\sum_{j=1}^{L}Z_{v_j}h_j-\sum_{j=1}^{L}h_j^2
\ +\ 2\sum_{i=1}^{m}Z_{w_i}\ell_i-\sum_{i=1}^{m}\ell_i^2 .
\end{equation}

\emph{The forcing inequality.} Aggregate the noise into three statistics, local to this proof:
\[
W:=\sum_{j=1}^{L}Z_{v_j},
\qquad
Q:=\sum_{j=1}^{L}\bigl[Z_{v_j}\bigr]_+^2,
\qquad
M:=\max_{1\le i\le m}Z_{w_i}.
\]
Routing the full budget to a leaf attaining $M$, that is, taking $h_j=V$ for every $j$, $\ell_{i^\star}=V$ at one index $i^\star$ with $Z_{w_{i^\star}}=M$, and $\ell_i=0$ elsewhere, is feasible in \Cref{eq:broomfeasible}. Its objective value is $2VW+2VM-(L+1)V^2$. On the other side, suppose $M>0$; then every feasible point with $h_L\le V/2$ obeys
\begin{equation}\label{eq:broomcap}
\Psi(h,\ell)\ \le\ Q+MV,
\end{equation}
since $2Z_{v_j}h_j-h_j^2\le[Z_{v_j}]_+^2$ for every $h_j\ge0$, while discarding $-\sum_i\ell_i^2$ and replacing each $Z_{w_i}$ by $M$ bounds the leaf part by $2M\sum_i\ell_i\le2Mh_L\le MV$, the last two steps by \Cref{eq:broomfeasible} and by the hypothesis $h_L\le V/2$ (\Cref{ssec:benchmarks}). The routed candidate beats the cap \Cref{eq:broomcap} exactly when
\begin{equation}\label{eq:forcing}
M\ >\ (L+1)V-2W+\frac{Q}{V}\ .
\end{equation}
When $M>0$ and \Cref{eq:forcing} hold, no maximizer of $\Psi$ can have $h_L\le V/2$, so the coordinates $(\widehat h,\widehat\ell)$ of $\widehat\mu_{\mathrm{LSE}}$ satisfy $\widehat h_L>V/2$; the chain \Cref{eq:broomfeasible} propagates the bound up the handle, and since the true nonroot signal is zero,
\begin{equation}\label{eq:handleloss}
\bigl\|\widehat\mu_{\mathrm{LSE}}-Vp_o\bigr\|_2^2\ \ge\
\sum_{j=1}^{L}\widehat h_j^{\,2}\ >\ L\,\frac{V^2}{4}\ =\
\frac{L^{3/2}}{4}\ .
\end{equation}

\emph{The forcing event.} Let
\[
E\ :=\ \{W\ge-2\sqrt L\}\cap\{Q\le2L\}\cap\{M\ge4L^{5/4}\}.
\]
On $E$, with $V=L^{1/4}$, the right side of \Cref{eq:forcing} is less than $4L^{5/4}\le M$ for $L\ge4$, so the forcing inequality holds; Markov's inequality on the handle statistics and a Gaussian tail estimate on the leaf maximum, using $m\ge e^{20L^{5/2}}$, give $\mathbb P(E)\ge\tfrac38$ (\Cref{ssec:benchmarks}). Combining with \Cref{eq:handleloss},
\begin{equation}\label{eq:broomrisk}
\E_{Vp_o}\bigl\|\widehat\mu_{\mathrm{LSE}}-Vp_o\bigr\|_2^2\ \ge\
\frac{L^{3/2}}{4}\,\mathbb P(E)\ \ge\ \frac{3}{32}\,L^{3/2}.
\end{equation}

\emph{The minimax rate.} The ancestor-covering numbers of the broom are explicit: $\Nup(0)=n$, and for integers $1\le j\le L$,
\begin{equation}\label{eq:broomcover}
\Nup(j)\ =\ \Bigl\lceil\frac{L+2}{j+1}\Bigr\rceil .
\end{equation}
Both directions follow by counting along the root-to-leaf path $v_0,\dots,v_L,w_1$; the single center $v_{L+1-j}$ covers every leaf at distance exactly $j$ (\Cref{ssec:benchmarks}). Thus every radius $r\ge1$ sees a covering count independent of $m$: the terminal star enters the profile \Cref{eq:profile} only through radii $r<1$, where the truncation caps its term at $k$. What survives is path geometry: \Cref{ssec:benchmarks} turns \Cref{eq:broomcover} into the profile bounds
\begin{equation}\label{eq:broomprofile}
\delta_k(T_L)\ \le\ \max\Bigl\{1,\ \frac{2(L+2)}{ek^2}\Bigr\}
\quad\text{for }k\ge1,
\qquad\quad
\delta_k(T_L)\ \ge\ \frac{L+2}{2ek^2}
\quad\text{for }k\le\frac{L+2}{2e}.
\end{equation}
The crossing then solves at $k_0\asymp(V^2(L+2)/\sigma^2)^{1/3}\asymp\sqrt L$, with the dimension clause of \Cref{eq:crossing} silent and the diameter branch of \Cref{thm:rate} inactive, so
\begin{equation}\label{eq:broomminimax}
R^*_{T_L}\bigl(L^{1/4},1\bigr)\ \asymp\ k_0\ \asymp\ \sqrt L
\end{equation}
(\Cref{ssec:benchmarks}).

\emph{Conclusion.} The maximality of $L$ gives $\sqrt n<e^{20(L+1)^{5/2}}$, so $\log n<40(L+1)^{5/2}$ and $L\ge c(\log n)^{2/5}$. Dividing \Cref{eq:broomrisk} by \Cref{eq:broomminimax} gives a ratio at least $cL\ge c'(\log n)^{2/5}$ at this signal, and the worst-case risk dominates this single-signal risk.
\end{proof}

The upper half is a statement about every tree and every parameter pair at once; its proof has three blocks, completed in \Cref{ssec:lse}.

\begin{proof}[Proof sketch of \Cref{thm:lse}, upper bound]
\emph{A localized projection principle.} For $\mu\in\mathcal F_V(T)$ and $t>0$, let
\[
w_\mu(t)\ :=\ \E\,\sup\bigl\{\langle Z,\theta-\mu\rangle:\
\theta\in\mathcal F_V(T),\ \|\theta-\mu\|_2\le t\bigr\}
\]
be the localized Gaussian width of the body at the truth. If $t\ge\sigma$ and $\sigma\,w_\mu(s)\le s^2/4$ for every $s\ge t$, then $\E_\mu\|\widehat\mu_{\mathrm{LSE}}-\mu\|_2^2\le33\,t^2$. On the event $\|\widehat\mu_{\mathrm{LSE}}-\mu\|_2\ge s$, projection optimality forces the supremum defining $w_\mu(s)$ above twice its mean; the supremum is $s$-Lipschitz in $Z$, so Gaussian concentration bounds its probability, and integrating the tail from $t$ gives the claim (\Cref{ssec:lse}). Localizing least squares at a fixed point of the width is Chatterjee's argument \citep{chatterjee2014}.

\emph{The width of the body through the coded cover.} Fix an integer $2\le k\le n/C_\star$, write $\overline\delta:=\overline\delta_k(T)$ for the surrogate profile \Cref{eq:surrogate}, and let $\ell_n:=1+\log(en)$. \Cref{ssec:lse} proves the truth-localized width bound
\begin{equation}\label{eq:truthwidth}
w_\mu(t)\ \le\
C\,\bigl(t+V\sqrt{\overline\delta}\,\bigr)\sqrt k
\ +\ C\,V\Bigl[\sqrt{k\overline\delta}\,\log(2+\ell_n)
+\sqrt{\overline\delta\,\ell_n}\,\Bigr]
\qquad\text{for all }t>0 .
\end{equation}
The first term is the width of the coded cover $\mathcal D_k$ seen from the truth: recentering at a deterministic center $\nu_\mu\in\mathcal D_k$ (\Cref{prop:cover}) confines every center met by the local ball to distance $t+2CV\sqrt{\overline\delta}$, and the class has logarithmic cardinality $O(k)$, so the maximum of the recentered pairings obeys a $\sqrt k$ bound. The second term bounds the residual $\theta-\nu_\theta$, uniformly over the body, through the three components of the construction behind \Cref{prop:cover}: the collapse and rounding residuals, bounded by a capped order-statistics sum and a union of norm concentrations over $e^{Ck}$ subspaces at $CV\sqrt{k\overline\delta}$ (\Cref{ssec:lse}); and the exit quantization. Conditionally on $(\theta,Z)$, the pairing of $Z$ with the quantization error is centered, so on the sampling event of that construction, of probability at least $\tfrac14$, its conditional expectation is at most four conditional standard deviations. Some realization in the event meets this bound, and the randomized rounding becomes, for each noise outcome, a deterministic choice of a nearby center in $\mathcal D_k$.

\emph{Closure of the exponent.} Let $2\le k_0\le K$, where $R^*_T\asymp\sigma^2k_0$ (\Cref{sec:crossing}). The crossing \Cref{eq:crossing} and the sandwich $\overline\delta_{k_0}\le2\delta_{k_0}$ give $V\sqrt{\overline\delta_{k_0}}\le C\sigma\sqrt{k_0}$, hence also $V\sqrt{k_0\overline\delta_{k_0}}\le C\sigma k_0$. Inserting these into \Cref{eq:truthwidth} at $k=k_0$ and feeding the result to the localization principle at $t^2\asymp\sigma^2[\,k_0\log(2+\ell_n)+\sqrt{k_0\,\ell_n}\,]$ gives
\begin{equation}\label{eq:lseratioA}
\frac{R^{\mathrm{worst}}_{\mathrm{LSE}}(T,V,\sigma)}
{R^*_T(V,\sigma)}
\ \le\ C\,\Bigl[\log(2+\ell_n)+\sqrt{\ell_n/k_0}\,\Bigr].
\end{equation}
This bound deteriorates as $k_0$ shrinks, but a small crossing index strengthens the projection: the profile remembers the height, $\delta_k\ge cH_T/k^2$ for $k\le n/C_\star$ (\Cref{ssec:lse}), so the crossing forces $V^2H_T\le C\sigma^2k_0^3$, while $R^{\mathrm{worst}}_{\mathrm{LSE}}\le V^2H_T$ holds, the truth and its projection lying in the body. Hence
\begin{equation}\label{eq:lseratioB}
\frac{R^{\mathrm{worst}}_{\mathrm{LSE}}(T,V,\sigma)}
{R^*_T(V,\sigma)}
\ \le\ C\,k_0^2 .
\end{equation}
The two bounds meet at $k_0=\ell_n^{1/5}$: for $k_0\le\ell_n^{1/5}$, \Cref{eq:lseratioB} is at most $C\ell_n^{2/5}$; for $k_0>\ell_n^{1/5}$, \Cref{eq:lseratioA} is at most $C[\log(2+\ell_n)+\ell_n^{2/5}]\le C'\ell_n^{2/5}$. At the endpoints $k_0=K+1$ and $k_0=1$ the deterministic caps $R^{\mathrm{worst}}_{\mathrm{LSE}}\le\min\{V^2H_T,\ \sigma^2n\}$ close the argument against $R^*_T$ (\Cref{ssec:lse}).
\end{proof}

The broom is where the two geometries of this paper part ways. Euclidean geometry sees $m$ orthogonal leaf directions, and the extreme noise among them pulls the projection through the handle; the ancestor geometry sees, at every radius $r\ge1$, one vertex that covers the entire star, and the truncation caps its cost at the information budget rather than per leaf. The truncated ancestor profile is therefore the right geometry for this body. On every finite tree it determines the minimax rate and yields an estimator evaluated exactly, one that survives the removal of the budget at an additive cost of order $\sigma^2$; across all trees it prices the blindness of the default convex program at the power $2/5$ of the logarithm.

\clearpage
%%%%%%%%%%%%%%%%%%%%%%%%%%%%%%%%%%%%%%%%%%%%%%%%%%%%%%%%%%%%%%%%%%%%%%%%%%
\appendix
\renewcommand{\thesection}{S.\arabic{section}}
\setcounter{section}{0}
%%%%%%%%%%%%%%%%%%%%%%%%%%%%%%%%%%%%%%%%%%%%%%%%%%%%%%%%%%%%%%%%%%%%%%%%%%

The appendices below contain the complete proofs. References of the form Lemma~2.3, (1.4), or Section~2 point to the main body above; statements and equations proved here carry the prefix~S, and the notation of the main body remains in force.

\section{The Profile Toolkit}\label{ssec:profile}
%%%%%%%%%%%%%%%%%%%%%%%%%%%%%%%%%%%%%%%%%%%%%%%%%%%%%%%%%%%%%%%%%%%%%%%%%%%%%%

This section proves the statements deferred from Section~2, Lemmas~2.1 to~2.3, and verifies two comparisons stated in Section~1.1: the center convention behind (1.6), and the capped constraint of \citet{chatterjeelafferty2018}. The two normal forms and the isometry come first; the computational claims of Lemma~2.3 then rest on an exact linear-time algorithm for minimum ancestor nets.

\begin{proof}[Proof of Lemma~2.1]
\emph{Item \textup{(1)} is equivalent to item \textup{(2)}.} Definition (1.2) reads $\mu=V\sum_u\lambda_up_u$ with $\lambda$ a probability vector on $\sfV$, and the substitution $s_u=V\lambda_u$ turns this into $\mu=\sum_us_up_u$ with $s_u\ge0$ and $\sum_us_u=V$. For uniqueness, order $\sfV$ so that every vertex precedes its descendants, for instance by nondecreasing depth, and let $P$ be the matrix with columns $(p_u)_{u\in\sfV}$ in that order. Its entry at row $v$ and column $u$ is $\ind\{v\preceq u\}$, which vanishes when $v$ follows $u$, because a strict ancestor of $u$ precedes $u$; and $p_u(u)=1$ on the diagonal. So $P$ is upper triangular with unit diagonal, hence $\det P=1$ and $P$ is invertible, and the coefficient vector is determined by $\mu$.

\emph{Item \textup{(2)} implies item \textup{(3)} and both displayed identities.} Evaluating $\mu=\sum_us_up_u$ at $v$ gives
\begin{equation}\label{eq:subtreesum}
\mu(v)\ =\ \sum_{u\in\sfV}s_u\,\ind\{v\preceq u\}\ =\
\sum_{u\succeq v}s_u ,
\end{equation}
the first identity; at $v=o$ it reads $\mu(o)=\sum_us_u=V$, every vertex being a descendant of the root. The strict descendants of $v$ are partitioned by the subtrees $T_c$, $c\in\ch(v)$, so subtracting \Cref{eq:subtreesum} at the children from \Cref{eq:subtreesum} at $v$ leaves the single term $u=v$:
\[
\mu(v)-\sum_{c\in\ch(v)}\mu(c)\ =\
\sum_{u\succeq v}s_u-\sum_{c\in\ch(v)}\sum_{u\succeq c}s_u\ =\ s_v .
\]
This is the second identity; since $s\ge0$ it gives $\mu(v)\ge\sum_{c\in\ch(v)}\mu(c)$, which with $\mu(o)=V$ is item (3).

\emph{Item \textup{(3)} implies item \textup{(2)}.} Set $s_v:=\mu(v)-\sum_{c\in\ch(v)}\mu(c)$, nonnegative by hypothesis. We prove \Cref{eq:subtreesum} by induction from the leaves upward. At a leaf the sum over children is empty, so $s_v=\mu(v)$, which is the claim. At an internal vertex, the same partition of the strict descendants of $v$ and the inductive hypothesis at the children give
\[
\sum_{u\succeq v}s_u\ =\ s_v+\sum_{c\in\ch(v)}\sum_{u\succeq c}s_u
\ =\ s_v+\sum_{c\in\ch(v)}\mu(c)\ =\ \mu(v).
\]
At $v=o$ this yields $\sum_us_u=\mu(o)=V$, and reading \Cref{eq:subtreesum} at every coordinate says exactly $\mu=\sum_us_up_u$.
\end{proof}

\begin{proof}[Proof of Lemma~2.2]
\emph{Part \textup{(1)}.} A vertex is a common ancestor of $u$ and $w$ precisely when it is an ancestor of their deepest common ancestor $u\wedge w$, and a vertex $x$ has exactly $\depth(x)+1$ ancestors, namely the vertices of the segment from $o$ to $x$. Hence
\[
\langle p_u,p_w\rangle\ =\ \#\{v:\ v\preceq u\ \text{and}\
v\preceq w\}\ =\ \depth(u\wedge w)+1 ,
\]
and in particular $\|p_x\|_2^2=\depth(x)+1$. Expanding the square,
\[
\|p_u-p_w\|_2^2\ =\ \depth(u)+\depth(w)-2\depth(u\wedge w)\ =\
d_T(u,w),
\]
the last equality because the $u$-to-$w$ path is the concatenation of the segments from $u$ up to $u\wedge w$ and from $u\wedge w$ down to $w$. The diameter of a convex hull equals that of its generating set, so $\diam^2\mathcal F_V(T)=\max_{u,w}\|Vp_u-Vp_w\|_2^2=V^2H_T$ by (1.2).

\emph{Part \textup{(2)}.} For $a\preceq b$ every ancestor of $a$ is an ancestor of $b$, so
\[
(p_b-p_a)(v)\ =\ \ind\{v\preceq b\}-\ind\{v\preceq a\}\ =\
\ind\{a\prec v\preceq b\},
\]
that is, $p_b-p_a=\ind_{(a,b]}$. For general $u,w$ put $m:=u\wedge w$ and subtract the two vertical cases,
\[
p_u-p_w\ =\ (p_u-p_m)-(p_w-p_m)\ =\ \ind_{(m,u]}-\ind_{(m,w]} .
\]
The sets $(m,u]$ and $(m,w]$ are disjoint, since a common element would be a common ancestor of $u$ and $w$ strictly below $m$; their union is the set of lower endpoints of the edges on the $u$-to-$w$ path, by the description of that path just used.

For the last assertion, let $u$ and $w$ lie in a connected subtree $C$. The $u$-to-$w$ path then runs inside $C$, so the support of $p_u-p_w$ consists of lower endpoints of edges of $C$. Distinct edges have distinct lower endpoints, an edge being determined by its lower endpoint, so differences taken within edge-disjoint connected subtrees occupy disjoint coordinate sets.
\end{proof}

Fix an integer radius $q\ge0$. The \emph{residual-depth greedy} computes a minimum ancestor $q$-net in one postorder pass, maintaining one number per vertex. Process the vertices in postorder, with a fixed child order, and at each vertex $v$ compute, from the values $r_c$ already assigned at its children,
\begin{equation}\label{eq:residual}
d_v:=\max\bigl(\{0\}\cup\{r_c+1:\ c\in\ch(v),\ r_c\ne-\infty\}\bigr);
\end{equation}
if $d_v=q$, select $v$ as a center and set $r_v:=-\infty$; otherwise set $r_v:=d_v$. After the root has been processed, select it as well if $r_o\ne-\infty$. The number $r_v$ tracks the depth below $v$ of the deepest vertex of $T_v$ that no selected center covers, with $-\infty$ standing for none; the proof below verifies this reading.

\begin{lemma}\label{slem:greedy}
For every integer $q\ge0$, the residual-depth greedy outputs a minimum ancestor $q$-net of $T$ in $O(n)$ operations, and the output is a deterministic function of $(T,q)$.
\end{lemma}

\begin{proof}
We first verify feasibility, then exhibit the selected centers, in a suitable order, as a run of a conceptually simpler procedure, the \emph{deepest-uncovered greedy}, and finally prove by an exchange argument that every run of that procedure is minimum.

\emph{Feasibility.} Call a vertex \emph{uncovered} at a given moment when no center selected so far is an ancestor of it within distance $q$. We claim that after $v$ is processed, $r_v$ is exactly the maximal depth below $v$ of an uncovered vertex of $T_v$, with $r_v=-\infty$ when none exists, and that $r_v<q$. When $v$ is examined, every selected center is an already processed vertex, and no ancestor of $v$ has been processed, so $v$ is uncovered. Moreover, an ancestor of a vertex of $T_c$, for a child $c$ of $v$, lies in $T_c$ or strictly above $c$, and no center selected after $c$ was processed is of either kind: the subtree $T_c$ is finished, and the vertices above $c$ are not yet processed. The uncovered part of $T_c$ is therefore still the one described by $r_c$, and the uncovered vertices of $T_v$ are $v$ itself together with those of the children's subtrees, one level deeper. Hence $d_v$ in \Cref{eq:residual} is exactly the maximal depth below $v$ of an uncovered vertex of $T_v$, and $d_v\le q$ by the inductive bound $r_c<q$. If $d_v=q$, selecting $v$ covers every uncovered vertex of $T_v$, all being descendants of $v$ within distance $q$, and $r_v=-\infty$ is correct; otherwise no center is added and $r_v=d_v<q$ is correct. When the root has been processed, either $r_o=-\infty$ and nothing remains uncovered, or the final selection covers the remaining uncovered vertices, which lie within $r_o<q$ of the root. Each vertex is examined once at cost proportional to its number of children, so the pass costs $O(n)$; and the child order being fixed, the algorithm makes no arbitrary choices, so the output is a deterministic function of $(T,q)$.

\emph{Reduction to the deepest-uncovered greedy.} The comparison procedure repeats the following step until every vertex is covered: pick an uncovered vertex $u$ of maximum depth and select its ancestor at distance exactly $\min\{q,\depth(u)\}$. To exhibit the algorithm's output as a run of this procedure, attach to every center $g$ selected through the test $d_g=q$ a \emph{witness} $u_g\in T_g$ realizing $d_g$, so that $d_T(g,u_g)=q$; if the root is selected at the end, its witness is a deepest vertex left uncovered, at distance $r_o<q$. Order the selected centers by nonincreasing depth, ties broken by processing order, with the final root selection last, and let $P$ denote the set of centers preceding a given center in this order.

Fix a center $g$ with $d_T(g,u_g)=q$. First, $u_g$ is uncovered when only $P$ is in place: a center covering $u_g$ is an ancestor of $u_g$ of depth at least $\depth(u_g)-q=\depth(g)$, so it lies on the vertical segment between $g$ and $u_g$, inside $T_g$, and was processed, and selected, before $g$; had it covered $u_g$, then $u_g$ would already be covered when $g$ was examined and could not realize $d_g$. Second, every vertex $w$ strictly deeper than $u_g$ is covered by $P$: the complete output is feasible, so some selected center $c$ satisfies $c\preceq w$ and $d_T(c,w)\le q$, whence
\[
\depth(c)\ \ge\ \depth(w)-q\ >\ \depth(u_g)-q\ =\ \depth(g),
\]
and every selected center strictly deeper than $g$ precedes it. Thus at $g$'s turn the vertex $u_g$ is a deepest uncovered vertex, and $g$ is its ancestor at distance exactly $q=\min\{q,\depth(u_g)\}$.

For the final root selection with witness $u$: no nonroot center covers the witness, which is still uncovered after the whole tree has been processed; every $w$ strictly deeper than $u$ is covered by a nonroot center, since otherwise $w$ would be an uncovered vertex deeper than $u$ at that moment, contradicting the maximality defining $r_o$; and the root is the ancestor of $u$ at distance $\depth(u)=\min\{q,\depth(u)\}$, since $\depth(u)=r_o<q$.

\emph{Optimality by exchange.} Suppose, inductively, that some minimum ancestor $q$-net $O$ contains the first $t-1$ centers of a run of the deepest-uncovered greedy; let $u$ be the uncovered vertex of maximum depth picked at step $t$ and $g$ its selected ancestor. Feasibility of $O$ provides $c\in O$ with $c\preceq u$ and $d_T(c,u)\le q$; being an ancestor of $u$ of depth at least $\depth(u)-q$, the center $c$ lies on the segment between $g$ and $u$, and $c$ is not among the first $t-1$ centers, which do not cover $u$. Set $O':=(O\setminus\{c\})\cup\{g\}$. A vertex $x$ covered by $c$ but not by $g$ is a descendant of $g$, as $g\preceq c\preceq x$, with $d_T(g,x)>q$, so that
\[
\depth(x)\ >\ \depth(g)+q\ \ge\ \depth(u),
\]
and $x$ is covered by the first $t-1$ centers, which lie in $O'$. Thus $O'$ is feasible, contains the first $t$ centers, and $|O'|\le|O|$, so $O'$ is again a minimum net. By induction the complete run is contained in a minimum net, and being itself feasible, it is minimum.
\end{proof}

\begin{proof}[Proof of Lemma 2.3]
Throughout, write $c_q:=\min\{k,[\log(\Nup(q)/k)]_+\}$ for the truncated factor of the profile at integer radius $q$, the dependence on the fixed index $k$ suppressed from the notation.

\emph{Part (1).} Dividing the truncated factor in (1.4) by $k$ and using $\min\{k,x\}/k=\min\{1,x/k\}$ gives
\begin{equation}\label{eq:deltasup}
\delta_k(T)=\sup_{r>0}\ r\,\min\Bigl\{1,\
\frac1k\Bigl[\log\frac{\Nup(r)}{k}\Bigr]_+\Bigr\}.
\end{equation}
Fix $r$. As $k$ increases, $[\log(\Nup(r)/k)]_+$ is nonnegative and nonincreasing, and so is $1/k$; a product of two nonnegative nonincreasing functions is nonincreasing, and the cap by $1$ preserves this. Every term of the supremum in \Cref{eq:deltasup} is therefore nonincreasing in $k$, and so is $\delta_k$. For the upper bound, the root is an ancestor of every vertex within distance $h_T$, so $\Nup(r)=1$ and the term vanishes for $r\ge h_T$, while for $r<h_T$ the term is at most $r<h_T$; hence $\delta_k\le h_T$.

\emph{Part (2).} Tree distances take integer values, so the covering constraint $d_T(a,u)\le r$ is equivalent to $d_T(a,u)\le\lfloor r\rfloor$, and $\Nup(r)=\Nup(\lfloor r\rfloor)$ for every $r\ge0$. The supremum in (1.4) therefore splits along integer radii:
\[
\alpha_k(T)
=\sup_{q\ge0}\ \sup_{r\in[q,q+1)\cap(0,\infty)}\ r\,c_q
=\sup_{q\ge0}\ (q+1)\,c_q ,
\]
the inner supremum being approached as $r\uparrow q+1$. Radii $q\ge h_T$ have $\Nup(q)=1$ and $c_q=0$, so the outer supremum is a maximum over the integers $0\le q<h_T$, an empty maximum being zero; this also covers $n=1$, where $h_T=0$.

\emph{Part (3): the two-sided comparison.} Write $G:=\{2^\ell-1:\ \ell\in\{0,1,2,\dots\},\ 2^\ell\le h_T\}$ for the grid, so that (2.1) reads $\overline\delta_k=\tfrac2k\max_{q\in G}(q+1)c_q$. The grid consists of integer radii below $h_T$, so part (2) gives $\max_{q\in G}(q+1)c_q\le\alpha_k$, that is, $\overline\delta_k\le2\delta_k$. Conversely, fix an integer $0\le q<h_T$ and set $s:=q+1\in[1,h_T]$; the largest power of two $p\le s$ satisfies $p>s/2$, and $q':=p-1$ belongs to $G$ with $q'\le q$. Every ancestor $q'$-net is an ancestor $q$-net, so $\Nup(q')\ge\Nup(q)$ and $c_{q'}\ge c_q$, whence
\[
(q'+1)\,c_{q'}\ =\ p\,c_{q'}\ \ge\ \frac s2\,c_q
\ =\ \frac{q+1}2\,c_q .
\]
Maximizing over $q$ gives $\max_{q\in G}(q+1)c_q\ge\alpha_k/2$, that is, $\overline\delta_k\ge\delta_k$. When $h_T=0$ both sides vanish by convention. Monotonicity of $k\mapsto\overline\delta_k$ follows as in part (1), applied to the finite maximum over $G$.

\emph{Part (3): the sandwich.} Put $K:=\lfloor n/C_\star\rfloor$. The dimension clause $k>n/C_\star$ first holds at $k=K+1$, so the minima in (1.5) and (2.2) are over nonempty sets and $k_0,\kalg\le K+1$. For the lower half: if $k_0\le K$, the dimension clause fails at $k_0$, so the profile clause holds there, $V^2\overline\delta_{k_0}\le2V^2\delta_{k_0}\le 2A_0\sigma^2k_0$, and $\kalg\le k_0$; if $k_0=K+1$, then $\kalg\le K+1=k_0$ directly. For the upper half: if $\kalg=K+1$, the lower half forces $k_0=K+1=\kalg$. Otherwise $\kalg\le K$, so the dimension clause of (2.2) fails at $\kalg$, its profile clause holds, and $V^2\delta_{\kalg}\le V^2\overline\delta_{\kalg}\le2A_0\sigma^2\kalg$. If $k_0=\kalg$ there is nothing to prove. If $k_0>\kalg$, every integer $k$ with $\kalg\le k\le k_0-1$ fails both clauses of (1.5) by the minimality of $k_0$; in particular $V^2\delta_k>A_0\sigma^2k$, and combining this with the monotonicity of part (1),
\[
A_0\sigma^2k\ <\ V^2\delta_k\ \le\ V^2\delta_{\kalg}
\ \le\ 2A_0\sigma^2\kalg ,
\]
so $k<2\kalg$. At $k=k_0-1$ this reads $k_0-1<2\kalg$, that is, $k_0\le2\kalg$.

\emph{Part (3): complexity.} \Cref{slem:greedy} computes $\Nup(q)$ exactly in $O(n)$ operations for each integer radius. The grid $G$ has $O(\log(h_T+1))$ elements, so storing $\{\Nup(q)\}_{q\in G}$ costs $O(n\log(h_T+1))$, after which each evaluation of $\overline\delta_k$ costs $O(\log(h_T+1))$. The predicate in (2.2) is monotone in $k$: in its profile clause the left side is nonincreasing by the monotonicity proved above and the right side is increasing, and its dimension clause is monotone; $\kalg$ is therefore located by binary search over $\{1,\dots,K+1\}$ with $O(\log n)$ predicate evaluations. Since $h_T<n$, the total is $O(n\log(n+1))$. For $k_0$, running \Cref{slem:greedy} at every integer radius $0\le q<h_T$ costs $O(n\,h_T)$ and stores all counts; each $\alpha_k$ then costs $O(h_T)$ by (2.3), and scanning $k=1,2,\dots$ until the first index satisfying (1.5), at most $K+1\le n$ values, costs $O(n\,h_T)$ in total.
\end{proof}

Finally, we verify the center-convention comparison stated with (1.6). Write $H^{c,\mathrm{int}}_T(t)$ for the covering entropy of $\mathcal F_V(T)$ at squared radius $t$ with centers restricted to the body, and
\[
\rho^{\mathrm{int}}_T(V,\sigma)\ :=\
\inf_{t>0}\ \bigl[t+\sigma^2H^{c,\mathrm{int}}_T(t)\bigr].
\]
Since every cover with centers in the body is in particular a cover with free centers, $H^c_T(t)\le H^{c,\mathrm{int}}_T(t)$ for every $t>0$. Conversely, fix $t>0$ and a cover of the body by $\exp H^c_T(t)$ balls of radius $\sqrt t$ with arbitrary centers, and replace the center $y$ of each ball by its Euclidean projection $\pi_y:=\Pi_{\mathcal F_V(T)}(y)$ onto the body. The variational inequality of the projection, $\langle y-\pi_y,\mu-\pi_y\rangle\le0$ for every $\mu\in\mathcal F_V(T)$, gives
\[
\|\mu-\pi_y\|_2^2\ \le\ \|\mu-y\|_2^2-\|y-\pi_y\|_2^2\ \le\
\|\mu-y\|_2^2 ,
\]
so every $\mu$ within $\sqrt t$ of $y$ is within $\sqrt t$ of $\pi_y$. The projected balls, of the same radius and number, cover the body with internal centers, so $H^{c,\mathrm{int}}_T(t)\le H^c_T(t)$. The two entropies coincide at every radius, and $\rho^{\mathrm{int}}_T=\rho_T$: the center convention is immaterial.

We also verify the comparison with the capped constraint of \citet{chatterjeelafferty2018}, stated in Section~1.1. Let
\[
\mathcal K_V(T)\ :=\ \bigl\{\mu\in\mathcal F(T):\ \mu(o)\le V\bigr\},
\qquad
R^{\mathrm{cap}}_T(V,\sigma)\ :=\ \inf_{\widehat\mu}\
\sup_{\mu\in\mathcal K_V(T)}\E_\mu\bigl\|\widehat\mu(Y)-\mu\bigr\|_2^2,
\]
the infimum over the same maps as in (1.3), and assume $n\ge2$. Since $\mathcal F_V(T)\subseteq\mathcal K_V(T)$, every estimator has at least as large a supremum risk on the larger set, so $R^{\mathrm{cap}}_T\ge R^*_T$.

For the reverse, split $\mu$ into its root value $\mu(o)$ and its remaining coordinates $\mu_-$, and write $C_a$ for the set of $\mu_-$ arising from $\mu\in\mathcal F_a(T)$. Adding $(V-a)p_o$ to such a $\mu$ raises its root value to $V$ and changes no other coordinate, so $C_a\subseteq C_V$ whenever $0<a\le V$, and $C_0:=\{0\}\subseteq C_V$ as well. On the slice the root value is known, so it is estimated at no loss, while $Y_o$ is independent of $Y_-$ with a law free of $\mu_-$ and therefore acts as pure randomization, which cannot lower a maximum risk under squared loss; hence $R^*_T(V,\sigma)$ is exactly the minimax risk of estimating $\mu_-\in C_V$ from $Y_-$. Fix $\varepsilon>0$, let $\widehat\theta$ attain that risk to within $\varepsilon$, and on the capped body take $\widehat\mu_-:=\widehat\theta(Y_-)$ together with
\[
\widehat\mu_o\ :=\
\begin{cases}
0, & V\le\sigma,\\[2pt]
\Pi_{[0,V]}(Y_o), & V>\sigma .
\end{cases}
\]
Every $\mu\in\mathcal K_V(T)$ has $\mu_-\in C_{\mu(o)}\subseteq C_V$, so its nonroot risk is at most $R^*_T(V,\sigma)+\varepsilon$. Its root risk is at most $\mu(o)^2\le V^2$ in the first case, and at most $\E(Y_o-\mu(o))^2=\sigma^2$ in the second, projection onto an interval containing $\mu(o)$ being nonexpansive. Letting $\varepsilon\downarrow0$,
\[
R^{\mathrm{cap}}_T(V,\sigma)\ \le\
R^*_T(V,\sigma)+\min\{V^2,\sigma^2\} .
\]
Finally $n\ge2$ forces $H_T\ge1$, so the two-point bound (3.8) gives
\[
\min\{V^2,\sigma^2\}\ \le\ \min\{V^2H_T,\ \sigma^2\}\ \le\
c^{-1}R^*_T(V,\sigma) :
\]
fixing the root value and capping it pose the same problem up to universal constants.

The comparison is constructive: the estimator of Theorem~2 transfers to the capped body. At budget $V$ its root coordinate is identically $V$, and its nonroot coordinates depend on $Y$ only through $Y_-$: every integer state has root coordinate $k$, so the root factor cancels from the ratio (4.2), and the branch selection in (4.3) depends on $(T,V,\sigma)$ alone. Write $\widehat\mu^V_-(Y_-)$ for this nonroot part. For $\mu\in\mathcal K_V(T)$ the raised signal $(V,\mu_-)$ lies in $\mathcal F_V(T)$, and $Y_-$ has the law of the nonroot data of the slice model there, so Theorem~2 bounds the risk of $\widehat\mu^V_-$ against $\mu_-$ by $C\,R^*_T(V,\sigma)$, uniformly over $\mathcal K_V(T)$. Paired with the root rule above, it yields a deterministic estimator, computed in $O(n^2\log(n+1))$ arithmetic operations, whose risk on the capped body is at most $C\bigl(R^*_T(V,\sigma)+\min\{V^2,\sigma^2\}\bigr)\le C'\,R^{\mathrm{cap}}_T(V,\sigma)$.

%%%%%%%%%%%%%%%%%%%%%%%%%%%%%%%%%%%%%%%%%%%%%%%%%%%%%%%%%%%%%%%%%%%%%%%%%%%%%%
\section{The Coded Cover}\label{ssec:cover}
%%%%%%%%%%%%%%%%%%%%%%%%%%%%%%%%%%%%%%%%%%%%%%%%%%%%%%%%%%%%%%%%%%%%%%%%%%%%%%

This section completes the proof of Proposition~3.1. The main text gives the mechanism behind each stage. To make this section self-contained, we restate the objects below and supply the estimates and accounting: the floor on the working scale and the sizes of the nets, the structure of the cells, the disjointness and charge accounting of the heavy hierarchy, the collapse and exit estimates, the edge count of the skeleton together with the explicit constants of the error assembly, and the counting bound for the coded cover.

Throughout, $k$ is an integer with $2\le k\le n/C_\star$ and, as in Section~3.1, $\bar\alpha:=k\overline\delta_k(T)$; the level weights are $m_j=2^j$ for $0\le j\le J$, with $m_J$ the largest power of two strictly below $k$; $S_j$ is the minimum ancestor $\lfloor\bar\alpha/m_j\rfloor$-net computed by the greedy of \Cref{slem:greedy}; and $R_j:=S_0\cup\dots\cup S_j$. The net hierarchy, the active support and the activation charges below therefore depend on $(T,k)$ alone, the signal entering only through the stopped refinement of the construction. The active support, the first-appearance weight, and the activation charge are
\[
A_k:=R_J,
\qquad
\omega(v):=\min\{m_j:\ v\in R_j\},
\qquad
\widetilde\omega(v):=\omega(v)\,\ind\{v\notin R_0\},
\]
and the code length of an integer vector $z$ supported on $A_k$ is
\[
\Gamma(z)\ :=\ \sum_{v\in A_k}\Bigl(|z_v|
+\widetilde\omega(v)\,\ind\{z_v\ne0\}\Bigr),
\]
the display (3.1). Throughout, $e_v$ denotes the coordinate vector at $v$.

The first lemma contains the size facts quoted when the nets are introduced: the working scale is never degenerate, and the profile itself caps the size of every net.

\begin{lemma}[Scale floor and net sizes]\label{slem:netsizes}
For every integer $2\le k\le n/C_\star$:
\begin{enumerate}
\item\label{sit:floor} $\bar\alpha\ge\alpha_k(T)\ge2$;
\item\label{sit:netsize}
$\Nup(\lfloor\bar\alpha/m_j\rfloor)\le ke^{m_j}$ for every
$0\le j\le J$;
\item\label{sit:union} $|R_j|\le2ke^{m_j}$ for every $0\le j\le J$.
\end{enumerate}
\end{lemma}

\begin{proof}
\emph{Part (1).} The surrogate dominates the profile, $\overline\delta_k\ge\delta_k$ by Lemma~2.3, so $\bar\alpha=k\overline\delta_k\ge k\delta_k=\alpha_k(T)$. For the floor, tree distances are integers, so at radii $0<r<1$ every vertex is its own only ancestor within distance $r$; every ancestor $r$-net then contains all of $\sfV$, and $\Nup(r)=n$. Hence
\[
\alpha_k(T)\ \ge\ \sup_{0<r<1}\ r\,\min\Bigl\{k,\
\Bigl[\log\frac nk\Bigr]_+\Bigr\}
\ =\ \min\Bigl\{k,\ \Bigl[\log\frac nk\Bigr]_+\Bigr\}
\ \ge\ \min\{2,\ \log C_\star\}\ =\ 2,
\]
using $k\ge2$ and $n/k\ge C_\star=e^2$.

\emph{Part (2).} Fix $0\le j\le J$ and suppose $\Nup(\lfloor\bar\alpha/m_j\rfloor)>ke^{m_j}$. The radius $r:=\bar\alpha/m_j$ is positive by part (\ref{sit:floor}), and the covering constraint depends on $r$ only through $\lfloor r\rfloor$, so $\Nup(r)=\Nup(\lfloor r\rfloor)>ke^{m_j}$, that is, $[\log(\Nup(r)/k)]_+>m_j$. Since also $m_j\le m_J<k$, the minimum of $k$ and this positive part exceeds $m_j$, and the term of (1.4) at $r$ satisfies
\[
r\,\min\Bigl\{k,\ \Bigl[\log\frac{\Nup(r)}k\Bigr]_+\Bigr\}
\ >\ \frac{\bar\alpha}{m_j}\cdot m_j\ =\ \bar\alpha\ \ge\ \alpha_k(T),
\]
contradicting that $\alpha_k(T)$ is the supremum of these terms.

\emph{Part (3).} The sets $S_i$ are minimum nets, so $|S_i|=\Nup(\lfloor\bar\alpha/m_i\rfloor)\le ke^{m_i}$ by part (\ref{sit:netsize}), and $|R_j|\le k\sum_{i\le j}e^{m_i}$. For $j\ge1$ the sum has $j$ terms before the last, each at most $e^{m_{j-1}}$, and $j\le e^{2^{j-1}}$ gives $je^{m_{j-1}}\le e^{m_{j-1}+2^{j-1}}=e^{m_j}$, so $\sum_{i\le j}e^{m_i}\le2e^{m_j}$; for $j=0$ this is immediate.
\end{proof}

The cells inherit their structure from the nesting of the nets alone. Parts (2)--(4) below are the rooted-connectedness, refinement, and radius properties quoted in the main body; part (3) also identifies the parent of each cell, the hierarchy used from here on.

\begin{lemma}[Cell structure]\label{slem:cells}
Fix $0\le j\le J$ and let $a_j(v)$ be the deepest ancestor of $v$ in $R_j$. Then:
\begin{enumerate}
\item\label{sit:partition} the cells $B_{a,j}=\{v:a_j(v)=a\}$,
$a\in R_j$, partition $\sfV$, and each contains its root $a$;
\item\label{sit:connected} if $v\in B_{a,j}$ and $a\preceq z\preceq v$,
then $z\in B_{a,j}$;
\item\label{sit:refine} if $j<J$ and $a_{j+1}(u)=b$, then
$a_j(u)=a_j(b)$; consequently every level-$(j{+}1)$ cell is contained
in a single level-$j$ cell, whose root is an ancestor of its root;
\item\label{sit:radius} $d_T(a,v)\le\lfloor\bar\alpha/m_j\rfloor$ for
every $v\in B_{a,j}$.
\end{enumerate}
\end{lemma}

\begin{proof}
All four parts are read off the chain structure of ancestors and the nesting $R_j\subseteq R_{j+1}$.

\emph{Part (1).} The root lies in every ancestor net, being its own only ancestor, so every vertex has an ancestor in $R_j$; the ancestors of $v$ form a chain, so among them the elements of $R_j$ have a unique deepest member, and $a_j(v)$ is well defined. The fibers of $v\mapsto a_j(v)$ partition $\sfV$, and $a_j(a)=a$ for $a\in R_j$.

\emph{Part (2).} Every ancestor of $z$ is an ancestor of $v$, so every $R_j$-ancestor of $z$ has depth at most $\depth(a_j(v))=\depth(a)$; and $a$ itself is an $R_j$-ancestor of $z$, since $a\preceq z$. Hence $a_j(z)=a$.

\emph{Part (3).} Let $s$ be an $R_j$-ancestor of $u$. Then $s\in R_{j+1}$, so $\depth(s)\le\depth(b)$; and $s$ and $b$ are comparable, both being ancestors of $u$, so $s\preceq b$ and $s$ is an $R_j$-ancestor of $b$. Conversely, every $R_j$-ancestor of $b$ is an ancestor of $u$, since $b\preceq u$. The $R_j$-ancestors of $u$ and of $b$ therefore coincide, and $a_j(u)=a_j(b)$. In particular $B_{b,j+1}\subseteq B_{a_j(b),j}$, and $a_j(b)\preceq b$.

\emph{Part (4).} The net property supplies $s\in R_j$ with $s\preceq v$ and $d_T(s,v)\le\lfloor\bar\alpha/m_j\rfloor$. The deepest $R_j$-ancestor $a=a_j(v)$ has $\depth(a)\ge\depth(s)$, so $a$ lies on the segment from $s$ to $v$ and $d_T(a,v)\le d_T(s,v)\le\lfloor\bar\alpha/m_j\rfloor$.
\end{proof}

The remaining facts concern the stopped refinement that the proof of Proposition~3.1 runs for a fixed signal. Fix $f=\sum_u\lambda_up_u$ with $\lambda_u\ge0$ and $\sum_u\lambda_u=1$, write $\lambda(B):=\sum_{u\in B}\lambda_u$, write $m(B):=m_j$ for the \emph{level weight} of a level-$j$ cell $B$, and call $B$ heavy when $\lambda(B)>m(B)/k$ and light otherwise. The refinement begins at level $0$, refines every heavy cell into its level-$(j{+}1)$ subcells, stops at every light cell, and stops everything at level $J$. The cells at which it stops are the \emph{terminal} cells. They partition $\sfV$ because the cells of one level do and each refinement step replaces a cell by a partition of it (\Cref{slem:cells}(\ref{sit:partition}),(\ref{sit:refine})), and $\mathcal L$ denotes their collection. For $L\in\mathcal L$ we write $a_L$ for its root and $\lambda_L:=\lambda(L)$. As in the main proof, $U$ consists of $R_0$ together with the roots of all heavy cells; a heavy cell is \emph{maximal} when no heavy cell lies strictly below it in the hierarchy of \Cref{slem:cells}(\ref{sit:refine}).

\begin{lemma}[The heavy structure]\label{slem:heavy}
In this setting:
\begin{enumerate}
\item\label{sit:disjoint} distinct maximal heavy cells have disjoint
underlying vertex sets;
\item\label{sit:edges} the minimal rooted subtree $S$ spanning $U$
has at most $5\bar\alpha k$ edges;
\item\label{sit:charges}
$\sum_{v\in U\setminus R_0}\widetilde\omega(v)<2k$;
\item\label{sit:maxweight}
$\sum_{C\ \text{maximal heavy}}m(C)<k$, which is the display (3.3) of
the main paper.
\end{enumerate}
\end{lemma}

\begin{proof}
Part (\ref{sit:disjoint}) is a chain argument in the hierarchy, and part (\ref{sit:maxweight}) follows from it; part (\ref{sit:edges}) charges the skeleton to vertical segments along the nets; part (\ref{sit:charges}) charges activation weights to maximal heavy cells.

\emph{Parts (1) and (4).} Suppose two distinct heavy cells $(B,i)$ and $(B',j)$ with $i\le j$ have intersecting underlying sets. Cells at one level partition $\sfV$, so $i<j$. Iterating \Cref{slem:cells}(\ref{sit:refine}), $B'$ is contained in a unique cell of each level below $j$; its level-$i$ cell meets $B$, hence equals $B$, and the chain of cells containing $B'$ at levels $i,\dots,j$ ascends from $(B',j)$ to $(B,i)$ in the hierarchy. Each cell of this chain, at level $\ell$ say, has mass at least $\lambda(B')>m_j/k\ge m_\ell/k$ and is therefore heavy. In particular the chain member at level $i{+}1$ is a heavy cell strictly below $(B,i)$, so $(B,i)$ is not maximal. Distinct maximal heavy cells therefore cannot intersect. Each carries mass exceeding $m(C)/k$, and their supports being disjoint, their masses total at most $\lambda(\sfV)=1$, so
\[
\sum_{C\ \text{maximal heavy}}m(C)\ <\ k ,
\]
which is part (\ref{sit:maxweight}); the empty family satisfies it trivially.

\emph{Part (2).} The subtree $S$ lies in the union of two families of vertical segments, whose total length we bound.

First, connect $R_0$ to the root. For $a\in R_0\setminus\{o\}$, the net property of $R_0$ at the parent of $a$ supplies $s\in R_0$ with $s\preceq\pa(a)$ and $d_T(s,\pa(a))\le\lfloor\bar\alpha\rfloor$; the deepest strict $R_0$-ancestor $\bar a$ of $a$ then lies on the segment from $s$ to $a$, so
\[
d_T(\bar a,a)\ \le\ d_T(s,a)\ \le\ \bar\alpha+1\ \le\
\tfrac32\,\bar\alpha,
\]
by \Cref{slem:netsizes}(\ref{sit:floor}). Iterating $a\mapsto\bar a$ from any element of $R_0$ reaches the root, which belongs to every ancestor net, so the segments $[\bar a,a]$, $a\in R_0\setminus\{o\}$, connect $R_0$ to $o$; by \Cref{slem:netsizes}(\ref{sit:netsize}) their total length is at most $\tfrac32\bar\alpha\,|R_0|\le\tfrac32e\,\bar\alpha k$.

Second, connect the heavy roots to $R_0$. The root $b$ of a heavy cell $B$ at level $j\ge1$ lies in its parent cell, whose root $a_{j-1}(b)$ is an ancestor of $b$ within distance $\lfloor\bar\alpha/m_{j-1}\rfloor\le2\bar\alpha/m_j$ (\Cref{slem:cells}(\ref{sit:refine})--(\ref{sit:radius})); and the parent cell is heavy as well, its mass being at least $\lambda(B)>m_j/k>m_{j-1}/k$. Iterating these segments joins every heavy root to the root of a level-$0$ cell, which lies in $R_0$. At level $j$ the heavy cells are disjoint and each has mass exceeding $m_j/k$, so there are fewer than $k/m_j$ of them, and the segments they contribute have total length less than
\[
\sum_{j\ge1}\frac k{m_j}\cdot\frac{2\bar\alpha}{m_j}
\ =\ 2\bar\alpha k\sum_{j\ge1}4^{-j}
\ =\ \tfrac23\,\bar\alpha k .
\]
Every member of $U$ is thus connected to the root inside the union of the two families, so $S$ is contained in that union and $|E(S)|\le\tfrac32e\,\bar\alpha k+\tfrac23\,\bar\alpha k\le 5\bar\alpha k$.

\emph{Part (3).} The set $U\setminus R_0$ consists of roots of heavy cells at levels $j\ge1$. Charge each $v\in U\setminus R_0$ to a heavy cell rooted at $v$ of minimal level $j(v)\ge1$; roots of level-$j$ cells lie in $R_j$, so $\widetilde\omega(v)=\omega(v)\le m_{j(v)}$, and the charging is injective into the set $\mathcal H$ of heavy cells. Hence
\[
\sum_{v\in U\setminus R_0}\widetilde\omega(v)
\ \le\ \sum_{(B,j)\in\mathcal H}m_j .
\]
Assign every heavy cell to one maximal heavy cell weakly below it in the hierarchy, obtained by descending to heavy children while one exists. The cells assigned to a fixed maximal cell $(C,\ell)$ all lie on its ancestor chain, so their level weights total at most $m_0+m_1+\dots+m_\ell<2m_\ell$, and
\[
\sum_{(B,j)\in\mathcal H}m_j
\ \le\ 2\sum_{C\ \text{maximal heavy}}m(C)
\ <\ 2k ,
\]
the last step by part (\ref{sit:maxweight}), which also covers the case of an empty family.
\end{proof}

The construction of the main proof samples the light exits at random. The next lemma supplies the collapse bound quoted there, the exit estimates behind the fixed realization, and the budgets that realization carries. Its objects are the following.

Collapsing every terminal cell to its root gives $g:=\sum_{L\in\mathcal L}\lambda_Lp_{a_L}$. A heavy terminal cell is one the refinement did not split, so it sits at level $J$ and no heavy cell lies below it; being maximal it contributes $m_J\ge k/2$ to the sum in \Cref{slem:heavy}(\ref{sit:maxweight}), and there is therefore at most one. Each terminal cell $L$ is attached to $U$ at $h(L):=a_L$ when $L$ is a level-$0$ cell or the heavy terminal one, and otherwise at the root of the parent cell of $L$, which is heavy because $L$ was reached by the refinement; either way $h(L)\in U$, and $h(L)\preceq a_L$ by \Cref{slem:cells}(\ref{sit:refine}). Put $q_L:=p_{a_L}-p_{h(L)}$. In the first two cases $q_L=0$; otherwise, $L$ sits at a level $j\ge1$, and its parent cell has level weight $m_j/2=m(L)/2$ and contains $a_L$, so \Cref{slem:cells}(\ref{sit:radius}) and Lemma~2.2(1) give
\begin{equation}\label{eq:exitlength}
\|q_L\|_2^2\ =\ d_T\bigl(h(L),a_L\bigr)\ \le\
\Bigl\lfloor\frac{2\bar\alpha}{m(L)}\Bigr\rfloor\ \le\
\frac{2\bar\alpha}{m(L)} .
\end{equation}
With $b_h:=\sum_{L:\,h(L)=h}\lambda_L$ this splits the collapsed signal as $g=g_0+\sum_L\lambda_Lq_L$, where $g_0:=\sum_{h\in U}b_hp_h$ is the \emph{skeleton} and the nonzero $q_L$ are the \emph{light exits}.

The skeleton is rounded through one scalar. List $U$ as $h_1,\dots,h_s$ in a fixed depth-first preorder of $T$, write $b_i:=b_{h_i}$ for the masses in that order, set $B_i:=\sum_{\ell\le i}b_\ell$ with $B_0:=0$, and put $Q_i:=\lfloor kB_i\rfloor-\lfloor kB_{i-1}\rfloor\in\mathbb Z_{\ge0}$. Let $Q$ be the integer vector with $Q_{h_i}:=Q_i$ and $Q_v:=0$ for $v\notin U$, and set $\bar g_0:=\frac1k\sum_{h\in U}Q_hp_h$. Over an interval of the list the $Q_i$ telescope, so their sum differs from $k$ times the corresponding mass by a difference of two floor errors, hence by less than $1$. The terminal cells partition $\sfV$, so $B_s=\sum_L\lambda_L=1$ and the whole list gives $\sum_iQ_i=\lfloor kB_s\rfloor=k$; as the $B_i$ increase, $Q$ is a nonnegative integer leak vector of total mass $k$ supported on $U$. Every subtree $T_v$ is a contiguous block of a depth-first preorder, so $U\cap T_v$ is such an interval, and coordinates are subtree sums by Lemma~2.1; therefore $g_0$ and $\bar g_0$ differ by less than $1/k$ at every vertex simultaneously.

Finally, write $\mathcal L_{\mathrm{ex}}:=\{L\in\mathcal L:q_L\ne0\}$ for the light exits, every member of which is a light cell and therefore has $\lambda_L\le m(L)/k$. They are quantized by the independent variables $Z_L=\frac{m(L)}k\xi_L$, $L\in\mathcal L_{\mathrm{ex}}$, with $\xi_L\sim\mathrm{Bernoulli}(k\lambda_L/m(L))$, which is a legitimate parameter by that bound. With the selection weight $W:=\sum_{L\in\mathcal L_{\mathrm{ex}}}m(L)\ind\{Z_L\ne0\}$, the integer vector is
\[
z\ =\ Q+\sum_{L\in\mathcal L_{\mathrm{ex}}:\,Z_L\ne0}
m(L)\bigl(e_{a_L}-e_{h(L)}\bigr),
\]
and the \emph{flow states} of an integer vector $z$ are the subtree sums $x_v:=\sum_{u\succeq v}z_u$, that is, the coordinates of $\sum_uz_up_u$.

\begin{lemma}[Collapse and exit estimates]\label{slem:collapse}
In this setting:
\begin{enumerate}
\item\label{sit:collapse} $\|f-g\|_2^2\le3\bar\alpha/k$, the display
(3.4);
\item\label{sit:exitvar} writing $\Sigma$ for the sum
$\sum_{L\in\mathcal L_{\mathrm{ex}}}(Z_L-\lambda_L)q_L$, one has
$\E\|\Sigma\|_2^2\le2\bar\alpha/k$ and
$\E W\le k$, so with probability at least $\tfrac14$ both $W\le2k$
and $\|\Sigma\|_2^2\le8\bar\alpha/k$ hold;
\item\label{sit:zbudget} on such a realization $z$ is supported on
$A_k$ and satisfies $\|z\|_1\le5k$, and its selected exit roots carry
total activation charge at most $W$.
\end{enumerate}
\end{lemma}

\begin{proof}
\emph{Part (1).} The blocks $e_L=\sum_{u\in L}\lambda_u(p_u-p_{a_L})$ have pairwise disjoint supports, as recorded in the main proof, so $\|f-g\|_2^2=\sum_{L\in\mathcal L}\|e_L\|_2^2$. Every $u\in L$ has $d_T(a_L,u)\le\bar\alpha/m(L)$ by \Cref{slem:cells}(\ref{sit:radius}), so the triangle inequality and Lemma~2.2(1) give $\|e_L\|_2\le\sum_{u\in L}\lambda_u\sqrt{d_T(a_L,u)}\le \lambda_L\sqrt{\bar\alpha/m(L)}$. A light terminal cell has $\lambda_L\le m(L)/k$, whence $\|e_L\|_2^2\le\lambda_L^2\bar\alpha/m(L)\le\lambda_L\bar\alpha/k$, and summing over the light cells with $\sum_L\lambda_L\le1$ bounds their total contribution by $\bar\alpha/k$. At most one terminal cell is heavy; it sits at level $J$, where $m_J\ge k/2$, and contributes $\|e_L\|_2^2\le\lambda_L^2\bar\alpha/m_J\le2\bar\alpha/k$. Adding the two bounds proves (1).

\emph{Part (2).} Each $Z_L$ has mean $\lambda_L$ and
\[
\operatorname{Var}(Z_L)\ \le\ \Bigl(\frac{m(L)}k\Bigr)^2
\frac{k\lambda_L}{m(L)}\ =\ \frac{\lambda_Lm(L)}k ,
\]
while $\|q_L\|_2^2\le2\bar\alpha/m(L)$ by \Cref{eq:exitlength}. The $Z_L$ are independent, so the cross terms vanish in expectation. With every sum below running over $L\in\mathcal L_{\mathrm{ex}}$,
\[
\E\Bigl\|\sum_L(Z_L-\lambda_L)q_L\Bigr\|_2^2
\ =\ \sum_L\operatorname{Var}(Z_L)\,\|q_L\|_2^2
\ \le\ \frac{2\bar\alpha}k\sum_L\lambda_L\ \le\ \frac{2\bar\alpha}k .
\]
Also $\E W=\sum_Lm(L)\,k\lambda_L/m(L)=k\sum_L\lambda_L\le k$. By Markov's inequality $\mathbb P(W>2k)\le\tfrac12$, and the squared norm exceeds $8\bar\alpha/k$ with probability at most $\tfrac14$, so both bounds hold with probability at least $\tfrac14$.

\emph{Part (3).} The skeleton $Q$ is supported on $U\subseteq A_k$ and each selected transfer on $\{a_L,h(L)\}$; the root of a level-$j$ cell lies in $R_j\subseteq A_k$, so $\operatorname{supp}(z)\subseteq A_k$. Since $Q$ is nonnegative with total mass $k$ and each selected transfer adds $2m(L)$ to the $\ell_1$ norm, $\|z\|_1\le k+2W\le5k$. The selected exit roots are distinct, being roots of distinct terminal cells, and $\widetilde\omega(a_L)\le\omega(a_L)\le m(L)$ because $a_L\in R_j$ with $m_j=m(L)$; their charges therefore total at most $\sum_{L:\,Z_L\ne0}m(L)=W$.
\end{proof}

Two explicit constants follow. The subtree $S$ has at most $5\bar\alpha k+1\le6\bar\alpha k$ vertices, by \Cref{slem:heavy}(\ref{sit:edges}) and $\bar\alpha\ge2$; since $g_0$ and $\bar g_0$ differ by less than $1/k$ at every vertex and vanish outside $S$, $\|g_0-\bar g_0\|_2^2<6\bar\alpha k/k^2=6\bar\alpha/k$: the display (3.5) holds with $C=6$. Together with the collapse and exit bounds of \Cref{slem:collapse}, the error assembly of the main proof reads
\[
\Bigl\|f-\frac1k\sum_vz_vp_v\Bigr\|_2^2
\ \le\ 3\Bigl(\frac{3\bar\alpha}k+\frac{6\bar\alpha}k
+\frac{8\bar\alpha}k\Bigr)
\ =\ \frac{51\,\bar\alpha}k
\ =\ 51\,\overline\delta_k(T)\ \le\ 102\,\delta_k(T),
\]
the last step by Lemma~2.3, so part (1) of Proposition~3.1 holds with $C_0=102$.

It remains to prove the counting bound.

\begin{lemma}[Counting bound]\label{slem:count}
The integer vectors $z$ supported on $A_k$ with $\sum_vz_v=k$, $\Gamma(z)\le9k$, and all flow states in $[0,3k]$ number at most $e^{45k}$. In particular part (2) of Proposition~3.1 holds with $C_1=45$.
\end{lemma}

\begin{proof}
We bound the larger class defined by the support and code-length constraints alone. Each member $z$ determines its support, split into the free part $\operatorname{supp}(z)\cap R_0$ and the charged part $\operatorname{supp}(z)\setminus R_0$, and its coefficient vector; we count the three in turn and multiply.

\emph{Free parts.} Any subset of $R_0$ can occur, and $|R_0|\le ke^{m_0}=ek$ by \Cref{slem:netsizes}(\ref{sit:netsize}), so the free parts number at most $2^{ek}\le e^{ek}$.

\emph{Charged parts.} Group the charged vertices by first appearance: $G_j:=\{v\in A_k\setminus R_0:\ \omega(v)=m_j\}$ satisfies $G_j\subseteq R_j$, hence $|G_j|\le2ke^{m_j}$ by \Cref{slem:netsizes}(\ref{sit:union}), for $1\le j\le J$. The groups partition $A_k\setminus R_0$, so a charged part is the disjoint union of its intersections with them, and one with $q_j$ elements of $G_j$ carries activation charge $\sum_jm_jq_j\le\Gamma(z)\le9k$. Inserting the factor $e^{2(9k-\sum_jm_jq_j)}\ge1$ and then dropping the charge constraint,
\begin{align*}
\#\{\text{charged parts}\}
\ \le\ e^{18k}\prod_{j=1}^{J}\ \sum_{q\ge0}
\binom{|G_j|}q\,e^{-2m_jq}
\ &=\ e^{18k}\prod_{j=1}^{J}\bigl(1+e^{-2m_j}\bigr)^{|G_j|}
\\
&\le\ e^{18k}\exp\Bigl(2k\sum_{j\ge1}e^{-m_j}\Bigr)
\ \le\ e^{19k},
\end{align*}
using $\log(1+x)\le x$, $|G_j|e^{-2m_j}\le2ke^{-m_j}$, and $\sum_{j\ge1}e^{-2^j}<\tfrac16$.

\emph{Coefficients.} Every charged support vertex carries charge at least $1$, so a support has at most $ek+9k\le12k$ vertices. On a fixed support of size $d$, the map $z\mapsto(z^+,z^-)$ is injective into the nonnegative integer vectors on $2d$ coordinates with total at most $9k$, since $\|z\|_1\le\Gamma(z)\le9k$; these number $\binom{9k+2d}{2d}\le2^{9k+2d}\le2^{33k}\le e^{23k}$.

Multiplying the three counts bounds the class by $e^{(e+19+23)k}\le e^{45k}$.
\end{proof}

The facts quoted in the main proof are now all in place, and Proposition~3.1 is proved.

%%%%%%%%%%%%%%%%%%%%%%%%%%%%%%%%%%%%%%%%%%%%%%%%%%%%%%%%%%%%%%%%%%%%%%%%%%%%%%
\section{The Packing Lower Bound}\label{ssec:packing}
%%%%%%%%%%%%%%%%%%%%%%%%%%%%%%%%%%%%%%%%%%%%%%%%%%%%%%%%%%%%%%%%%%%%%%%%%%%%%%

This section completes the lower-bound half of Section~3.2. Lemmas~3.2 and~3.3, whose proofs were sketched there, are proved in full first, with the constants of the latter fixed. Two self-contained lemmas follow: the additivity of Bayes risks across coordinate-disjoint blocks, and the edge-disjoint grouping of a separated vertex set. The section then completes the proofs of Propositions~3.4 and~3.5, and finally the deferred cases of Theorem~1 and Proposition~1.

\begin{proof}[Proof of Lemma 3.2]
Write $\mathcal P:=\{p_u:u\in\sfV\}$ and $\varepsilon:=\sqrt r/3$, and fix a maximal $\varepsilon$-separated subset of $\mathcal P$; by maximality, every point of $\mathcal P$ lies within Euclidean distance $\varepsilon$ of the subset. By Lemma~2.2(1), the vertices indexing the subset have pairwise tree distances at least $\varepsilon^2=r/9$, so it suffices to bound their number below by $\Nup(r)$.

Consider the ball of radius $\varepsilon$ around one point of the subset, let $C$ be the set of vertices whose indicators it contains, and let $a$ be the deepest common ancestor of $C$. Fix $u\in C$ with $u\ne a$ and let $c$ be the child of $a$ on the path toward $u$. Not all of $C$ lies in the subtree $T_c$, since $c$ would then be a common ancestor of $C$ deeper than $a$; fix $v\in C\setminus T_c$. The meet $u\wedge v$ is at least as deep as $a$, the latter being a common ancestor of $u$ and $v$; and it is not deeper, since an ancestor of $u$ strictly below $a$ lies in $T_c$, and $u\wedge v\in T_c$ would force $v\in T_c$. Hence $u\wedge v=a$, so the path from $u$ to $v$ passes through $a$, and by Lemma~2.2(1),
\[
d_T(a,u)\ \le\ d_T(u,v)\ =\ \|p_u-p_v\|_2^2\ \le\ (2\varepsilon)^2
\ =\ \tfrac49\,r\ <\ r ,
\]
the indicators $p_u$ and $p_v$ lying in one ball of radius $\varepsilon$. The bound holds trivially at $u=a$, so $a$ ancestor-covers every vertex of $C$ within distance $r$. Every vertex of $T$ lies in the captured set of some ball, so collecting one such ancestor per ball yields an ancestor $r$-net, of size at most the number of balls; the number of balls is the size of the separated subset, and every ancestor $r$-net has at least $\Nup(r)$ elements, which proves the bound.
\end{proof}

Two facts about Gaussian shifts are used without further comment \citep{tsybakov2009}. The Kullback--Leibler divergence between $N(\theta,\sigma^2I_N)$ and $N(\theta',\sigma^2I_N)$ is $\|\theta-\theta'\|_2^2/(2\sigma^2)$. Fano's inequality in its pairwise form states that if a label is uniform on $M\ge2$ hypotheses whose data laws satisfy $\max_{i,j}\operatorname{KL}(P_i,P_j)\le\kappa$, then every decoder based on the data errs with probability at least $1-(\kappa+\log2)/\log M$.

\begin{proof}[Proof of Lemma 3.3]
Fix $\eta:=\tfrac1{10}$, the constant of the main-text sketch, and shrink each block into its information budget:
\[
\tau_j^2\ :=\ \min\Bigl\{1,\ \frac{\eta\,\sigma^2h_j}{D_j^2}\Bigr\},
\qquad
\widetilde\mu_{\mathbf a}\ :=\ \mu_0+\sum_j\tau_jg_j(a_j),
\]
so that $\tau_j\in[0,1]$ and, by hypothesis (1), the shrunken family lies in $K$. Place the uniform product prior on $\mathbf a$. Since the family lies in $K$, the minimax risk over $K$ dominates the Bayes risk of this prior, and it suffices to bound the latter for an arbitrary estimator $\widehat\mu=\widehat\mu(Y)$.

\emph{Separation.} Let $\widehat{\mathbf a}$ index a member of the shrunken family nearest to $\widehat\mu$, ties broken by a fixed order, so that $\|\widehat\mu-\widetilde\mu_{\mathbf a}\|_2\ge \tfrac12\|\widetilde\mu_{\widehat{\mathbf a}} -\widetilde\mu_{\mathbf a}\|_2$ by the triangle inequality, and write $\Delta_j:=g_j(\widehat a_j)-g_j(a_j)$, blocks with $\Delta_j=0$ dropped. With $\rho_{ij}:=|\langle\Delta_i,\Delta_j\rangle| /(\|\Delta_i\|_2\|\Delta_j\|_2)$, hypothesis (3) gives $\sum_{j\ne i}\rho_{ij}\le\gamma$ for every $i$, and
\[
\Bigl\|\sum_j\tau_j\Delta_j\Bigr\|_2^2
\ \ge\ \sum_j\tau_j^2\|\Delta_j\|_2^2
-\sum_{i\ne j}\rho_{ij}\,\tau_i\|\Delta_i\|_2\,\tau_j\|\Delta_j\|_2
\ \ge\ (1-\gamma)\sum_j\tau_j^2\|\Delta_j\|_2^2 ,
\]
the last step by $xy\le\tfrac12(x^2+y^2)$ and the row sums. Since $\|\Delta_j\|_2\ge d_j$ whenever $\widehat a_j\ne a_j$, pathwise
\[
\|\widehat\mu-\widetilde\mu_{\mathbf a}\|_2^2
\ \ge\ \frac{1-\gamma}4\sum_j\tau_j^2d_j^2\,
\ind\{\widehat a_j\ne a_j\}.
\]

\emph{Per-block error.} Fix $j$ and condition on the other labels $\mathbf a_{-j}$. Under the product prior, $a_j$ remains uniform on $\mathcal A_j$, and subtracting the known vector $\mu_0+\sum_{i\ne j}\tau_ig_i(a_i)$ from $Y$ reduces the conditional problem to testing the finite Gaussian family $\{\tau_jg_j(a):a\in\mathcal A_j\}$, whose pairwise Kullback--Leibler divergences are at most $\tau_j^2D_j^2/(2\sigma^2)\le\eta h_j/2$. If $|\mathcal A_j|\ge3$, Fano's inequality gives every decoder error probability at least $1-\eta/2-\log2/\log|\mathcal A_j| \ge1-\tfrac1{20}-\log2/\log3>\tfrac14$. If $|\mathcal A_j|=2$, then $D_j=d_j$ and the two means are at distance $\tau_jD_j\le\sigma\sqrt{\eta\log2}$; the likelihood-ratio test is optimal and projects the data onto the difference of the means; its two conditional errors are equal, and the Bayes error is
\[
\Phi_{\mathcal N}\Bigl(-\frac{\tau_jD_j}{2\sigma}\Bigr)
\ \ge\ \frac12-\frac{\sqrt{\eta\log2}}{2\sqrt{2\pi}}\ >\ \frac14,
\]
where $\Phi_{\mathcal N}$ is the standard normal distribution function, whose density is at most $1/\sqrt{2\pi}$. In either case, every decoder of $a_j$ that is measurable in $Y$ is, conditionally on $\mathbf a_{-j}$, a decoder of the reduced problem, so $\mathbb P(\widehat a_j\ne a_j)\ge\tfrac14$.

\emph{Assembly.} Taking expectations in the separation bound,
\[
\E\|\widehat\mu-\widetilde\mu_{\mathbf a}\|_2^2
\ \ge\ \frac{1-\gamma}4\sum_j\tau_j^2d_j^2\,
\mathbb P(\widehat a_j\ne a_j)
\ \ge\ \frac{1-\gamma}{16}\sum_j\tau_j^2d_j^2 ,
\]
and by the definition of $\tau_j$ and hypothesis (2),
\[
\tau_j^2d_j^2
\ =\ \min\Bigl\{d_j^2,\ \eta\,\sigma^2h_j\,\frac{d_j^2}{D_j^2}\Bigr\}
\ \ge\ \min\Bigl\{d_j^2,\ \frac\eta A\,\sigma^2h_j\Bigr\}
\ \ge\ \frac\eta A\,\min\bigl\{d_j^2,\ \sigma^2h_j\bigr\},
\]
so the Bayes risk of the prior is at least $c_A(1-\gamma)\sum_j\min\{d_j^2,\sigma^2h_j\}$ with $c_A:=\eta/(16A)=1/(160A)$.

\emph{The restricted loss.} When every contrast is supported in the coordinate set $S$, choose $\widehat{\mathbf a}$ nearest to $\widehat\mu$ in the norm restricted to $S$. The differences $(\widetilde\mu_{\widehat{\mathbf a}}-\widetilde\mu_{\mathbf a})|_S =\sum_j\tau_j\Delta_j$ are unchanged, since the contrasts live in $S$, so the separation bound holds for $\|(\widehat\mu-\widetilde\mu_{\mathbf a})|_S\|_2^2$; and the per-block bound applies to any decoder measurable in the full data $Y$, in particular to this one. The assembly is identical.
\end{proof}

The assembly across groups in Proposition~3.4 rests on the fact that coordinate-disjoint blocks contribute additively to the Bayes risk.

\begin{lemma}[Bayes risks add across disjoint blocks]\label{slem:blocks}
Let
\[
\mu_{\boldsymbol\theta}\ =\ \mu_0+\sum_{i=1}^{m}h_i(\theta_i),
\qquad
\boldsymbol\theta=(\theta_1,\dots,\theta_m),
\]
where the $\theta_i$ are independent under a product prior and, for each $i$, every contrast $h_i(\theta_i)-h_i(\theta_i')$ is supported in a coordinate set $S_i$, the sets $S_1,\dots,S_m$ pairwise disjoint. In the model $Y=\mu_{\boldsymbol\theta}+\sigma Z$, $Z\sim N(0,I_N)$, the Bayes risk of estimating $\mu_{\boldsymbol\theta}$ in squared loss is at least the sum over $i$ of the Bayes risks of the isolated problems: observe $Y_{S_i}$, which up to a fixed additive vector is $h_i(\theta_i)|_{S_i}+\sigma Z_{S_i}$, and estimate $\mu_{\boldsymbol\theta}|_{S_i}$ under the $i$-th marginal prior.
\end{lemma}

\begin{proof}
The squared loss dominates the sum of its restrictions to the disjoint sets: $\|\widehat\mu-\mu_{\boldsymbol\theta}\|_2^2\ge \sum_i\|(\widehat\mu-\mu_{\boldsymbol\theta})|_{S_i}\|_2^2$. Fix $i$. For $\ell\ne i$ the restriction $h_\ell(\theta_\ell)|_{S_i}$ does not depend on $\theta_\ell$: two values of $\theta_\ell$ differ on $S_i$ by a contrast of block $\ell$ restricted to $S_i$, which vanishes because those contrasts are supported in $S_\ell$, disjoint from $S_i$. Hence $\mu_{\boldsymbol\theta}|_{S_i}=c_i+h_i(\theta_i)|_{S_i}$ for a fixed vector $c_i$, and $Y_{S_i}$ is distributed as in the isolated problem. Moreover $Y_{S_i^c}$ is a function of $(\boldsymbol\theta_{-i},Z_{S_i^c})$, which is independent of $(\theta_i,Z_{S_i})$ under the product prior, so $Y_{S_i^c}$ is independent of the pair $(\theta_i,Y_{S_i})$ and the conditional law of $\mu_{\boldsymbol\theta}|_{S_i}$ given $Y$ equals its conditional law given $Y_{S_i}$. The Bayes-optimal estimate of $\mu_{\boldsymbol\theta}|_{S_i}$, its posterior mean, is therefore a function of $Y_{S_i}$ alone, and $\E\|(\widehat\mu-\mu_{\boldsymbol\theta})|_{S_i}\|_2^2$ is at least the isolated Bayes risk for every estimator $\widehat\mu$. Summing over $i$ completes the proof.
\end{proof}

The next lemma produces the edge-disjoint groups of the main-text sketch, with explicit comparison constants.

\begin{lemma}[Edge-disjoint grouping]\label{slem:group}
Let $U\subseteq\sfV$ consist of $M\ge2$ vertices and let $q$ be an integer with $1\le q\le M/2$. There are connected subtrees $\mathcal C_1,\dots,\mathcal C_{q'}$ of $T$ with pairwise disjoint edge sets, and pairwise disjoint sets $U_i\subseteq U\cap \mathcal C_i$, such that
\[
\frac q4\ \le\ q'\ \le\ 6q,
\qquad
\frac M{6q}\ \le\ |U_i|\ \le\ \frac Mq
\quad\text{and}\quad
|U_i|\ \ge\ 2
\qquad\text{for every }i .
\]
\end{lemma}

\begin{proof}
Call the members of $U$ \emph{marks}. Work in the minimal subtree of $T$ spanning $U$, rooted by the orientation inherited from $T$, and set $L:=\max\{2,\lfloor M/(4q)\rfloor\}$.

\emph{The procedure.} Process the vertices in postorder, distinguishing each mark's \emph{token}, which will be assigned to at most one group, from the vertices and edges that only provide connectivity. Each vertex $v$ receives a collection of \emph{items}, each a connected subgraph containing $v$ together with the unassigned tokens it carries: for every child $c$ of $v$, the residual passed up from $c$, extended by the edge $(v,c)$; and, if $v$ is a mark, the one-point item $\{v\}$ carrying the token of $v$. Insert the items into an open bin one at a time; whenever the bin first carries at least $L$ tokens, seal it: its subgraphs unite into a component, connected through $v$, and its tokens, now permanently assigned, form the component's mark set. After the last item, pass the union of the open bin and the bare vertex $\{v\}$, carrying the bin's tokens, upward as the residual of $v$. At the root, discard the final residual.

\emph{The invariant.} Sealed components and residuals are connected, being unions of items that share the vertex at which they formed. A residual carries fewer than $L$ tokens, its bin having never reached $L$; consequently every item carries fewer than $L$ tokens, the one-point items because $L\ge2$, and a component seals with at least $L$ but fewer than $2L$ tokens. Each token is assigned at most once, so the mark sets $U_i$ are pairwise disjoint; and a token's vertex belongs to every subgraph that carries it, so $U_i\subseteq U\cap \mathcal C_i$. Finally, each edge of the spanning subtree enters the process in exactly one item and thereafter stays inside whichever bin, component, or residual that item joined; the components are therefore pairwise edge-disjoint.

\emph{Counting.} Let $q'$ be the number of sealed components; by the invariant, $|U_i|\in[L,2L)$ for each. Every token is either assigned at some sealing or sits in the final, discarded residual, which carries fewer than $L$; hence
\[
q'L\ \le\ M\ <\ 2Lq'+L .
\]
If $M\ge12q$, then $L=\lfloor M/(4q)\rfloor\ge3$ and $M/(6q)\le L\le M/(4q)$, since $\lfloor x\rfloor\ge x-1\ge\tfrac23x$ for $x\ge3$; hence $q'\le M/L\le6q$ and $q'>M/(2L)-\tfrac12\ge2q-\tfrac12$, so $q'\ge2q\ge q$, while $|U_i|\in[L,2L)\subseteq[M/(6q),\,M/(2q)]$. If $M<12q$, then $L=2$: every item then carries at most one token, so components seal with exactly two, and $|U_i|=2$ lies in $[M/(6q),M/q]$ because $2q\le M<12q$. For $M\ge4$ the count gives $q'\le M/2\le6q$ and $q'>(M-2)/4\ge M/8\ge q/4$; for $M\in\{2,3\}$, necessarily $q=1$, and the procedure seals exactly one component: fewer than $4$ tokens admit at most one, and at least one seals because all unassigned tokens funnel into the bins at the root.
\end{proof}

\begin{proof}[Completion of the proof of Proposition 3.4]
\Cref{slem:group}, applied with the parameter $q$ of the proposition, supplies the groups of the main-text sketch, with $q/4\le q'\le6q$ and $M/(6q)\le M_i\le M/q$; every group receives the budget $a=V/q'$. Fix a group, with marked set $U_i$ of size $M_i\ge2$, and adopt the objects of the sketch: the radii $r_\ell=16^\ell r$, the nested maximal separated subsets $N_\ell$, the descent chain with its node weights and child counts $m_\ell$, the retained levels $I$, the alphabets $\mathcal A_\ell$, and the masses $a_\ell$ with their normalizer $S$. We verify the claims quoted there and assemble.

\emph{The hierarchy and its entropy.} Maximality of $N_\ell$ inside $N_{\ell-1}$ means that every point of $N_{\ell-1}$ lies within $r_\ell$ of a point of $N_\ell$, so parents can be assigned, each point of $N_\ell$ to itself; and once $r_\ell$ exceeds the diameter of $U_i$ an $r_\ell$-separated subset is a single point: the hierarchy has a top. The weight of a node is at most its number of children times the largest weight of a child, so descending along children of maximal weight from the top, whose weight is $M_i$, to a leaf, whose weight is one, gives $M_i\le\prod_\ell m_\ell$, that is, $\sum_\ell\log m_\ell\ge\log M_i$. Levels with $m_\ell=1$ contribute zero to this sum, and among the four residue classes modulo four of the remaining levels one carries at least a quarter of it; the retained class $I$ therefore has $H=\sum_{\ell\in I}h_\ell\ge\tfrac14\log M_i$.

\emph{Separations and aspect.} Distinct children $u\ne v$ of the chain node at level $\ell$ lie in $N_{\ell-1}$, so $d_T(u,v)\ge r_{\ell-1}$; both lie within $r_\ell$ of their parent, so $d_T(u,v)\le2r_\ell=32\,r_{\ell-1}$. By Lemma~2.2(1), the contrasts $\Delta=a_\ell(p_u-p_v)$ then satisfy $a_\ell^2r_{\ell-1}\le\|\Delta\|_2^2\le32\,a_\ell^2r_{\ell-1}$, so $d_\ell^2\ge a_\ell^2r_{\ell-1}$ and $D_\ell^2\le32\,a_\ell^2r_{\ell-1}\le32\,d_\ell^2$: hypothesis (2) of Lemma~3.3 holds with $A=32$.

\emph{Overlap.} For $\ell<\ell'$ in $I$, a contrast $\Delta_\ell$ has entries in $[-a_\ell,a_\ell]$ and support of size $d_T(u,v)\le32\,r_{\ell-1}$, by Lemma~2.2(2), so $|\langle\Delta_\ell,\Delta_{\ell'}\rangle| \le32\,r_{\ell-1}a_\ell a_{\ell'}$, and with the norm lower bounds just proved the normalized entry is at most $32\sqrt{r_{\ell-1}/r_{\ell'-1}}=32\cdot4^{-(\ell'-\ell)}$. Retained levels differ by at least four, so every row of the normalized Gram matrix sums to at most
\[
2\sum_{t\ge1}32\cdot4^{-4t}\ =\ \frac{64}{255}\ =:\ \gamma\ <\ 1 :
\]
hypothesis (3) of Lemma~3.3 holds with this $\gamma$.

\emph{Legality and scalability.} With the reference points $u_\ell^0\in\mathcal A_\ell$ and $\mu_0:=\sum_{\ell\in I}a_\ell p_{u_\ell^0}$, a blockwise-scaled member of the group family is
\[
\mu_0+\sum_{\ell\in I}\tau_\ell\,a_\ell\bigl(p_{u_\ell}-p_{u_\ell^0}\bigr)
\ =\ \sum_{\ell\in I}a_\ell\bigl[(1-\tau_\ell)\,p_{u_\ell^0}
+\tau_\ell\,p_{u_\ell}\bigr],
\]
a leak combination with nonnegative coefficients of total mass $\sum_{\ell\in I}a_\ell=a$; hypothesis (1) of Lemma~3.3 holds with $K$ the convex set of all such mass-$a$ combinations.

\emph{The allocation.} Since $a_\ell^2r_{\ell-1}=a^2h_\ell/S^2$,
\[
\sum_{\ell\in I}\min\bigl\{a_\ell^2r_{\ell-1},\ \sigma^2h_\ell\bigr\}
\ =\ \sum_{\ell\in I}h_\ell\,\min\Bigl\{\frac{a^2}{S^2},\
\sigma^2\Bigr\}
\ =\ H\,\min\Bigl\{\frac{a^2}{S^2},\ \sigma^2\Bigr\}.
\]
By the Cauchy--Schwarz inequality, $S^2\le H\sum_{\ell\in I}r_{\ell-1}^{-1}$, and the radii grow geometrically, so $\sum_{\ell\in I}r_{\ell-1}^{-1}\le r^{-1}\sum_{i\ge0}16^{-i} =16/(15r)\le2/r$; hence $a^2/S^2\ge a^2r/(2H)$ and
\[
H\,\min\Bigl\{\frac{a^2}{S^2},\ \sigma^2\Bigr\}
\ \ge\ \min\Bigl\{\frac{a^2r}2,\ \sigma^2H\Bigr\}
\ \ge\ \frac14\,\min\bigl\{a^2r,\ \sigma^2\log M_i\bigr\},
\]
using $H\ge\tfrac14\log M_i$.

\emph{Assembly.} The contrasts of the group join members of $U_i\subseteq \mathcal C_i$, and the connecting paths run inside the connected subtree $\mathcal C_i$, so all group contrasts are supported in the set $S_i$ of lower endpoints of the edges of $\mathcal C_i$ (Lemma~2.2(2)). Lemma~3.3, in its restricted-loss form with $A=32$ and $\gamma=64/255$, therefore bounds the Bayes risk of the shrunken group family, for every estimator measurable in the full data, below by
\[
c_A(1-\gamma)\sum_{\ell\in I}\min\bigl\{d_\ell^2,\sigma^2h_\ell\bigr\}
\ \ge\ c'\,\min\bigl\{a^2r,\ \sigma^2\log M_i\bigr\},
\qquad
c':=\frac{c_A(1-\gamma)}4 ,
\]
by the aspect and allocation bounds above. Sums of one shrunken member per group are legal points of $\mathcal F_V(T)$, as verified in the main body; the sets $S_i$ are pairwise disjoint, the $E(\mathcal C_i)$ being disjoint and distinct edges having distinct lower endpoints; and the prior is a product across groups. \Cref{slem:blocks} therefore bounds the joint Bayes risk below by the sum of the isolated group risks, and each isolated problem is, after subtracting its fixed vector, the group problem observed through $Y_{S_i}$ alone, so its Bayes risk also dominates the displayed bound. Hence
\[
R^*_T(V,\sigma)\ \ge\ c'\sum_{i=1}^{q'}
\min\bigl\{a^2r,\ \sigma^2\log M_i\bigr\}.
\]
Finally, $a=V/q'\ge V/(6q)$, and $\log M_i\ge\tfrac17\log(eM/q)$: if $M/q\ge36$, then with $x:=\log(M/q)$, $\log M_i\ge\log(M/(6q))=x-\log6$ and $5(x-\log6)-(1+x)=4x-5\log6-1\ge4\log36-5\log6-1=3\log6-1>0$, so the ratio to $\log(eM/q)=1+x$ is at least $\tfrac15$; if $M/q<36$, then $\log M_i\ge\log2$ and $\log(eM/q)<1+\log36$, and $7\log2>4.8>1+\log36$. Therefore
\[
R^*_T(V,\sigma)\ \ge\ \frac{q'c'}{36}\,
\min\Bigl\{\frac{V^2r}{q^2},\ \sigma^2\log\frac{eM}q\Bigr\}
\ \ge\ \frac{c'}{144}\,
\min\Bigl\{\frac{V^2r}q,\ \sigma^2q\log\frac{eM}q\Bigr\},
\]
which is the claim.
\end{proof}

The reduction of Proposition~3.5 rests on a scalar matching of the free parameter $q$ to the profile; the constant $8$ below is the one priced into $A_0=18\cdot8=144$.

\begin{lemma}[Scalar matching]\label{slem:match}
Let $k\ge1$ and $M>k$ be integers, let $r>0$, and set $L:=\log(M/k)>0$. If
\[
V^2r\,\min\Bigl\{1,\ \frac Lk\Bigr\}\ \ge\ 8\,\sigma^2k ,
\]
then some integer $1\le q\le M/2$, comparable to $\max\{1,\ k/(1+L)\}$ within universal factors, satisfies
\[
\min\Bigl\{\frac{V^2r}q,\ \sigma^2q\log\frac{eM}q\Bigr\}
\ \ge\ \frac{\sigma^2k}7 .
\]
\end{lemma}

\begin{proof}
Write $B:=1+L$ and note $M\ge2$. There are three cases.

\emph{Case $L\ge k$.} Take $q:=1$, which equals $\max\{1,k/B\}$ since $k/B<1$. The hypothesis reads $V^2r\ge8\sigma^2k$, and $\sigma^2\log(eM)=\sigma^2(1+\log M)\ge\sigma^2(1+L)\ge\sigma^2k$, so both branches are at least $\sigma^2k$.

\emph{Case $L<k$ and $M\ge8k/B$.} Take $q:=\max\{1,\lfloor k/B\rfloor\}$. Then $k/(2B)\le q\le2k/B$: if $k\ge B$ this is $\lfloor x\rfloor\in[x/2,x]$ at $x=k/B\ge1$; if $k<B$, then $q=1$, and $k/(2B)<\tfrac12<1\le2k/B$ because $B=1+L<1+k\le2k$. Also $q\le2k/B\le M/4<M/2$ by the case hypothesis. The hypothesis of the lemma reads $V^2rL/k\ge8\sigma^2k$, so
\[
\frac{V^2r}q\ \ge\ \frac B{2k}\,V^2r\ \ge\ \frac L{2k}\,V^2r
\ \ge\ 4\,\sigma^2k .
\]
For the entropy branch, $M=ke^L$ and $q\le2k/B$ give $eM/q\ge\tfrac e2\,Be^L$, so
\[
\log\frac{eM}q\ \ge\ 1+L+\log B-\log2\ \ge\ 0.3\,B ,
\]
since $0.7(1+L)+\log(1+L)\ge0.7>\log2$; hence $\sigma^2q\log(eM/q)\ge\sigma^2\,\tfrac k{2B}\cdot0.3B =0.15\,\sigma^2k$.

\emph{Case $L<k$ and $M<8k/B$.} Take $q:=\lfloor M/2\rfloor$, so $M/3\le q\le M/2$; and $q$ is comparable to $\max\{1,k/B\}$, since $k<M<8k/B$ forces $B<8$, whence $q\ge(M-1)/2\ge k/2$ and $q<4k/B\le4k$. The entropy branch has $\log(eM/q)\ge\log(2e)>1$, so $\sigma^2q\log(eM/q)\ge\sigma^2M/3>\sigma^2k/3$; and the mass branch, using $k/M>B/8$ and $B>L$,
\[
\frac{V^2r}q\ \ge\ \frac{2V^2r}M
\ =\ 2\,\frac{V^2rL}k\cdot\frac k{LM}
\ \ge\ 16\,\sigma^2k\cdot\frac B{8L}\ \ge\ 2\,\sigma^2k .
\]
In every case the minimum is at least $0.15\,\sigma^2k\ge\sigma^2k/7$.
\end{proof}

\begin{proof}[Completion of the proof of Proposition 3.5]
The main text reduces the proposition to the following situation: an integer $k\ge1$, and $M\ge\Nup(r)>k$ vertices with pairwise tree distances at least $r':=r/9$, where
\[
V^2r'\,\min\Bigl\{1,\ \frac1k\Bigl[\log\frac Mk\Bigr]_+\Bigr\}
\ \ge\ \frac1{18}\,V^2\delta_k(T)\ \ge\ \frac{A_0}{18}\,\sigma^2k
\ =\ 8\,\sigma^2k :
\]
the term at $(r',M)$ is at least one ninth of the term at $(r,\Nup(r))$, since $r'=r/9$ and enlarging the count only increases the positive part. With $L:=\log(M/k)>0$ this is the hypothesis of \Cref{slem:match}. Its conclusion returns an integer $1\le q\le M/2$, comparable to $\max\{1,k/(1+\log(M/k))\}$ as asserted in the main body, with $\min\{V^2r'/q,\ \sigma^2q\log(eM/q)\}\ge\sigma^2k/7$. Proposition~3.4, applied to the $M$ separated vertices with this $q$, gives $R^*_T(V,\sigma)\ge(c/7)\,\sigma^2k$ with $c$ the constant of that proposition, which proves Proposition~3.5.
\end{proof}

Section~3.3 assembles the two halves. We next record its two-point bound, the deferred comparisons from three of the four cases, and the fixed-point conclusion. First, the generic bound (3.6) rests on the oracle inequality for least squares over a finite set of candidates. The inequality is classical, and belongs to the Gaussian model selection theory of \citet{birgemassart2001} and \citet{massart2007}; we give the short proof so that its constants are explicit.

\begin{lemma}[Finite-class least squares]\label{slem:finiteclass}
Let $Y=\mu+\sigma Z$ with $Z\sim N(0,I_n)$ and $\mu\in\bbR^n$, let $\mathcal N\subset\bbR^n$ be finite and nonempty, and let $\widehat f$ minimize $\|Y-f\|_2^2$ over $f\in\mathcal N$, ties broken by a fixed rule. Then
\[
\E_\mu\bigl\|\widehat f-\mu\bigr\|_2^2\ \le\
4\min_{f\in\mathcal N}\|f-\mu\|_2^2
\ +\ 12\,\sigma^2\bigl(1+\log|\mathcal N|\bigr).
\]
\end{lemma}

\begin{proof}
Let $f^\star$ attain the minimum and $a:=\|f^\star-\mu\|_2$. If $|\mathcal N|=1$ then $\widehat f=f^\star$ and the claim is immediate, so assume $N:=|\mathcal N|\ge2$. Optimality of $\widehat f$ and $Y-f=(\mu-f)+\sigma Z$ give
\[
\|\mu-\widehat f\|_2^2+2\sigma\langle Z,\mu-\widehat f\rangle
\ \le\ a^2+2\sigma\langle Z,\mu-f^\star\rangle ,
\]
that is, $\|\mu-\widehat f\|_2^2\le a^2+2\sigma\langle Z,\widehat f-f^\star\rangle$. Writing $u_f:=(f-f^\star)/\|f-f^\star\|_2$ for $f\ne f^\star$ and $G:=\bigl[\max_{f\ne f^\star}\langle Z,u_f\rangle\bigr]_+$, we get $\langle Z,\widehat f-f^\star\rangle\le\|\widehat f-f^\star\|_2\,G$, whence, with $x:=\|\widehat f-\mu\|_2$ and the triangle inequality $\|\widehat f-f^\star\|_2\le x+a$,
\[
x^2\ \le\ a^2+2\sigma Gx+2\sigma Ga\ \le\
a^2+\Bigl(\tfrac{x^2}2+2\sigma^2G^2\Bigr)
+\bigl(a^2+\sigma^2G^2\bigr),
\]
so $x^2\le4a^2+6\sigma^2G^2$. Each $\langle Z,u_f\rangle$ is standard normal, so $\mathbb P(G>s)\le\min\{1,Ne^{-s^2/2}\}$ and, splitting the integral $\E G^2=\int_0^\infty2s\,\mathbb P(G>s)\,ds$ at $s_0:=\sqrt{2\log N}$,
\[
\E G^2\ \le\ s_0^2+N\int_{s_0}^\infty 2s\,e^{-s^2/2}\,ds
\ =\ 2\log N+2 .
\]
Combining the two displays proves the lemma.
\end{proof}

\begin{proof}[Completion of Theorem 1 and Proposition 1]
For the two-point bound, pick $u,w$ at tree distance $H_T$. The segment from $Vp_u$ to $Vp_w$ lies in $\mathcal F_V(T)$ and has length $V\sqrt{H_T}$ by Lemma~2.2(1), so it contains two points at distance
\[
\varrho\ :=\ \min\bigl\{V\sqrt{H_T},\ \sigma\bigr\}.
\]
Testing one against the other is a one-dimensional Gaussian shift of size $\varrho/\sigma\le1$, so under the uniform prior on the pair every test errs with probability at least a universal constant. For any estimator, decoding the nearer of the two points errs only when the squared loss is at least $\varrho^2/4$; hence the maximum of the two risks is at least $c\varrho^2$, which is the display (3.8).

\emph{Case $n<C_\star$.} By (3.8) and $\sigma^2\ge\sigma^2n/C_\star$,
\[
R^*_T\ \ge\ c\min\{V^2H_T,\sigma^2\}\ \ge\
\frac c{C_\star}\min\{V^2H_T,\ \sigma^2n\}\ \ge\
\frac c{C_\star}\bigl(\rho_T\wedge V^2H_T\wedge\sigma^2n\bigr),
\]
which is (I) and the lower half of (II); the upper half follows from (3.6), since $\rho_T\wedge V^2H_T\wedge\sigma^2n\le C_\star\min\{V^2H_T,\sigma^2\}$.

\emph{Case $k_0=1$ with $n\ge C_\star$.} The main text bounds $\rho_T\le(A_0/\log2)\,\sigma^2$ there, so
\[
\rho_T\wedge V^2H_T\wedge\sigma^2n\ \le\ \rho_T\ \le\
C\min\{\sigma^2,\ V^2H_T\},
\]
using $\rho_T\le V^2h_T\le V^2H_T$ for the second branch; with (3.8) this gives (I). For (II), $\min\{V^2H_T,\sigma^2k_0\}=\min\{V^2H_T,\sigma^2\}$, whose lower half is (3.8) and whose upper half is $R^*_T\le C\rho_T\wedge CV^2H_T\le C'\min\{\sigma^2,V^2H_T\}$ by (3.6).

\emph{Case $k_0=K+1$ with $n\ge C_\star$.} Here $K=\lfloor n/C_\star\rfloor\ge n/(2C_\star)$, so $\sigma^2k_0=\sigma^2(K+1)$ lies between $\sigma^2n/(2C_\star)$ and $2\sigma^2n/C_\star$, while $V^2H_T\ge V^2\delta_K>A_0\sigma^2K\ge \sigma^2k_0/2$. Hence $\min\{V^2H_T,\sigma^2k_0\}\asymp\sigma^2n$, bounded below by the profile-to-risk lower bound proved in the main body and above by (3.6) through the branch $\sigma^2n$.

\emph{The fixed point.} In the interior case $2\le k_0\le K$, the bound (3.6) gives $\rho_T\ge\rho_T\wedge V^2H_T\wedge\sigma^2n\ge R^*_T/C$, so (3.9) and (3.10) pin the fixed point itself: $c\,\sigma^2k_0\le\rho_T\le C\,\sigma^2k_0$.
\end{proof}

%%%%%%%%%%%%%%%%%%%%%%%%%%%%%%%%%%%%%%%%%%%%%%%%%%%%%%%%%%%%%%%%%%%%%%%%%%%%%%
\section{Benchmark and Broom Profiles}\label{ssec:benchmarks}
%%%%%%%%%%%%%%%%%%%%%%%%%%%%%%%%%%%%%%%%%%%%%%%%%%%%%%%%%%%%%%%%%%%%%%%%%%%%%%

This section proves Corollaries~1, 2, and~3, together with the broom profile bounds (5.8) and the remaining computations deferred from the broom lower bound of Section~5: the objective cap, the probability of the forcing event, the broom's covering counts, and the crossing that pins its minimax rate. By Theorem~1 the minimax risk is $\min\{V^2H_T,\sigma^2k_0\}$ up to universal constants, so each corollary reduces to pinning the covering counts and the profile with explicit constants, localizing $k_0$, and identifying the active branch in every parameter regime. Throughout, $K:=\lfloor n/C_\star\rfloor$ with $C_\star=e^2$; the dimension clause of (1.5) holds at $K+1$, so $1\le k_0\le K+1$ on every tree. In the corollary proofs, $W:=V/\sigma$.

On the path, ancestor covering is interval covering, and the exact counts it yields reduce the profile to one scalar optimization.

\begin{lemma}[Path profile]\label{slem:path}
Let $P_L$ be the path with $L\ge1$ edges, rooted at one endpoint, so that $n=L+1$ and $h_{P_L}=H_{P_L}=L$. For integers $0\le q<L$,
\[
\Nup(q)\ =\ \Bigl\lceil\frac{L+1}{q+1}\Bigr\rceil .
\]
Moreover
\[
\alpha_k(P_L)\ \le\ \frac{2(L+1)}{ek}
\quad\text{for every }k\ge1,
\qquad
\alpha_k(P_L)\ \ge\ \frac{L+1}{2ek}
\quad\text{for }1\le k\le\frac{L+1}{e^2}.
\]
\end{lemma}

\begin{proof}
Index the vertices by depth $0,\dots,L$. The ancestors of $u$ are $0,\dots,u$, so a center $a$ covers exactly the vertices $a,\dots,\min\{a+q,L\}$: ancestor $q$-nets are coverings of $L+1$ points by blocks of $q+1$ consecutive integers anchored at their left endpoints. The centers $j(q+1)$, $0\le j\le\lfloor L/(q+1)\rfloor$, form such a covering, and there are $\lfloor L/(q+1)\rfloor+1=\lceil(L+1)/(q+1)\rceil$ of them; no net is smaller, because each center covers at most $q+1$ of the $L+1$ vertices and net sizes are integers.

\emph{Upper bound.} Consider a radius $q$ whose term in the discrete formula (2.3) is nonzero, so that $\Nup(q)>k\ge1$; then $\lceil(L+1)/(q+1)\rceil\ge2$ forces $(L+1)/(q+1)>1$, whence $\Nup(q)<(L+1)/(q+1)+1<2(L+1)/(q+1)$. With $y:=q+1$ and $A:=2(L+1)/k$, the term is therefore at most
\[
y\,\log\frac Ay\ \le\ \sup_{y>0}\ y\log\frac Ay\ =\ \frac Ae\ =\
\frac{2(L+1)}{ek},
\]
the supremum being attained at $y=A/e$.

\emph{Lower bound.} Set $x:=(L+1)/(ek)$ and $y:=\lfloor x\rfloor$. The hypothesis $k\le(L+1)/e^2$ gives $x\ge e>2$, so $y\ge x-1\ge x/2$; and $y\le x\le(L+1)/e\le L$, so $q:=y-1$ is an admissible radius. Then $\Nup(q)=\lceil(L+1)/y\rceil\ge(L+1)/y\ge ek$, so $[\log(\Nup(q)/k)]_+\ge1$ and the term at $q$ is at least $y\min\{k,1\}=y\ge(L+1)/(2ek)$.
\end{proof}

\begin{proof}[Proof of Corollary 1]
By Theorem~1 and $H_{P_L}=L$, $R^*_{P_L}(V,\sigma)\asymp\min\{V^2L,\sigma^2k_0\}$; write
\[
F\ :=\ \min\bigl\{V^2L,\ \ V^{2/3}\sigma^{4/3}L^{1/3},\ \
\sigma^2(L+1)\bigr\}
\]
for the right side of Corollary~1. Since its middle branch is the geometric interpolation $(V^2L)^{1/3}(\sigma^2)^{2/3}$, every branch of $F$, and hence $F$ itself, dominates $\min\{V^2L,\sigma^2\}$. Here $K=\lfloor(L+1)/e^2\rfloor$, and there are three cases.

\emph{Case $k_0=1$.} Here $\min\{V^2L,\sigma^2k_0\}=\min\{V^2L,\sigma^2\}$, which $F$ dominates; and since $F\le V^2L$, the matching upper bound needs an argument only when $\sigma^2<V^2L$. By (1.5), $k_0=1$ means $L+1<e^2$ or $V^2\delta_1\le A_0\sigma^2$. In the first case $F\le\sigma^2(L+1)<e^2\sigma^2$. Otherwise $L+1\ge e^2$, so \Cref{slem:path} applies at $k=1$ and gives $\delta_1=\alpha_1\ge(L+1)/(2e)$, whence $F\le V^2L\le2e\,V^2\delta_1\le2eA_0\sigma^2$.

\emph{Case $2\le k_0\le K$.} By minimality no clause of (1.5) holds below $k_0$, while at $k_0$ the crossing clause holds, the dimension clause being silent there. \Cref{slem:path} is available on all of $[1,K]$, since $K\le(L+1)/e^2$, and with
\[
\kappa\ :=\ \Bigl(\frac{V^2(L+1)}{eA_0\sigma^2}\Bigr)^{1/3}
\]
the crossing inequalities at $k_0$ and $k_0-1$ read
\[
\frac{V^2(L+1)}{2e\,k_0^2}\ \le\ V^2\delta_{k_0}\ \le\ A_0\sigma^2k_0,
\qquad
A_0\sigma^2(k_0-1)\ <\ V^2\delta_{k_0-1}\ \le\
\frac{2V^2(L+1)}{e\,(k_0-1)^2},
\]
that is, $k_0^3\ge\kappa^3/2$ and $(k_0-1)^3<2\kappa^3$; since $k_0\ge2$ gives $k_0\le2(k_0-1)$,
\[
2^{-1/3}\,\kappa\ \le\ k_0\ \le\ 2^{4/3}\,\kappa .
\]
Hence $\sigma^2k_0\asymp\sigma^2\kappa\asymp V^{2/3}\sigma^{4/3}L^{1/3}$, the middle branch of $F$, using $L\le L+1\le2L$. That branch is the least of the three up to constants: failure of the crossing clause at $k=1$ gives $A_0\sigma^2<V^2\delta_1\le4V^2L/e$, so $\sigma^2\le CV^2L$ and the middle branch is at most $(V^2L)^{1/3}(CV^2L)^{2/3}=C^{2/3}V^2L$; and it is $\asymp\sigma^2k_0\le\sigma^2K\le\sigma^2(L+1)$. Therefore $F\asymp\sigma^2k_0$, and the same middle-branch bound gives $\sigma^2k_0\le CV^2L$, so $\min\{V^2L,\sigma^2k_0\}\asymp\sigma^2k_0$ as well.

\emph{Case $k_0=K+1\ge2$.} Then $K\ge1$, so $L+1\ge e^2$, and no clause fired on $[1,K]$; at $k=K$,
\[
A_0\sigma^2K\ <\ V^2\delta_K\ \le\ \frac{2V^2(L+1)}{eK^2},
\quad\text{so}\quad
V^2(L+1)\ >\ \frac{eA_0}2\,\sigma^2K^3\ \ge\
\frac{A_0}{16e^5}\,\sigma^2(L+1)^3,
\]
using $K\ge(L+1)/(2e^2)$, which holds because $\lfloor x\rfloor\ge x/2$ for $x\ge1$. Hence $V^2\ge c\,\sigma^2(L+1)^2$ with $c:=A_0/(16e^5)$, and both remaining branches of $F$ dominate the third: $V^2L\ge\tfrac12V^2(L+1)\ge\tfrac c2\,\sigma^2(L+1)^3\ge \tfrac c2\,\sigma^2(L+1)$, and $(V^2L)^{1/3}\sigma^{4/3}\ge(\tfrac c2)^{1/3}\sigma^2(L+1)$. Therefore $F\asymp\sigma^2(L+1)$. On the other side, $(L+1)/e^2<K+1\le2(L+1)/e^2$ gives $\sigma^2k_0\asymp\sigma^2(L+1)$, and $\min\{V^2L,\sigma^2k_0\}\asymp\sigma^2(L+1)$ as well, by $V^2L\ge\tfrac c2\,\sigma^2(L+1)$. The three cases are exhaustive.
\end{proof}

On the star, $h_{S_m}=1$, so only the count $\Nup(0)=m+1$ remains; the profile is exact by Remark~2.1, splitting into a capped and a logarithmic phase, and the case analysis tracks which phase meets the crossing.

\begin{proof}[Proof of Corollary 2]
By Theorem~1 and $H_{S_m}=2$, $R^*_{S_m}(V,\sigma)\asymp\min\{V^2,\sigma^2k_0\}$ after absorbing the factor two; here $n=m+1$. By Remark~2.1,
\[
\alpha_k\ =\ \min\{k,\ \ell(k)\},
\qquad
\ell(x)\ :=\ \Bigl[\log\frac{m+1}x\Bigr]_+\quad(x>0),
\]
and since $k$ grows while $\ell(k)$ does not, the set $\{k\ge1:k\le\ell(k)\}$ is an initial segment of the integers; it contains $1$, because $\ell(1)=\log(m+1)>1$, and is finite, because $\ell$ vanishes beyond $m+1$. With $\bar k$ its largest element, the profile has two phases:
\[
V^2\delta_k\ =\ \sigma^2W^2\ \ (k\le\bar k),
\qquad
V^2\delta_k\ =\ \sigma^2W^2\,\frac{\ell(k)}k\ \ (k>\bar k).
\]
Write
\[
\Lambda\ :=\ 1+\bigl[\log(m/W)\bigr]_+,
\qquad
F\ :=\ \min\bigl\{V^2,\ \sigma^2W\sqrt\Lambda,\ \sigma^2m\bigr\},
\]
so that $F$ is the right side of Corollary~2.

\emph{Bounded stars.} For any universal $m_0$ the corollary holds for $m\le m_0$ with constants depending only on $m_0$: from $1\le k_0\le K+1\le m_0$, $\min\{V^2,\sigma^2k_0\}\asymp\min\{V^2,\sigma^2\}$, while $F\le\min\{V^2,\sigma^2m\}\le m_0\min\{V^2,\sigma^2\}$ and every branch of $F$ dominates $\min\{V^2,\sigma^2\}$, the middle one because $\sigma^2W\sqrt\Lambda\ge\sigma V\ge\min\{V^2,\sigma^2\}$. Now fix a universal $m_0$ with the following four properties for every $m\ge m_0$ (one may take $m_0=\lceil e^{10}\rceil$): (a) $\bar k\ge\tfrac12\log(m+1)$; (b) $\bar k+1\le K$; (c) $\log(m/w)\ge\tfrac12(1+\log m)$ whenever $1\le w\le\sqrt{A_0\log(m+1)}$; (d) $m\ge4e^4(1+\log m)$. Property (a) holds because $k':=\lceil\tfrac12\log(m+1)\rceil$ satisfies $k'+\log k'\le\log(m+1)$ once $m$ is large, so that $k'\le\ell(k')$; (b) and (d) compare linear with logarithmic growth; (c) is worst at $w=\sqrt{A_0\log(m+1)}$, where it reads $\tfrac12\log m\ge\tfrac12+\tfrac12\log(A_0\log(m+1))$. We also use, unconditionally, $\bar k\le\ell(\bar k)\le\log(m+1)$. Assume $m\ge m_0$; there are four cases in $W$.

\emph{Case $W^2\le A_0$.} Since $\bar k\ge1$, the capped phase gives $V^2\delta_1=\sigma^2W^2\le A_0\sigma^2$: the crossing clause fires at once and $k_0=1$. Then $\min\{V^2,\sigma^2k_0\}\in[V^2/A_0,\,V^2]$. And $F\asymp V^2$: it is at most $V^2$, while $\sigma^2W\sqrt\Lambda\ge\sigma^2W\ge V^2/\sqrt{A_0}$ and $\sigma^2m\ge2\sigma^2\ge2V^2/A_0$.

\emph{Case $A_0<W^2\le A_0\bar k$.} Every integer $k<W^2/A_0$ has $k<\bar k$, so the capped phase gives $V^2\delta_k=\sigma^2W^2>A_0\sigma^2k$, and $k<\bar k\le K$ keeps the dimension clause silent: no clause fires below $W^2/A_0$. At $k_1:=\lceil W^2/A_0\rceil$, which is at most $\bar k\le K$ because $W^2/A_0\le\bar k$ and $\bar k$ is an integer, the bound $\delta_{k_1}\le1$ gives $V^2\delta_{k_1}\le\sigma^2W^2\le A_0\sigma^2k_1$. Hence
\[
k_0\ =\ \Bigl\lceil\frac{W^2}{A_0}\Bigr\rceil\ \in\
\Bigl[\frac{W^2}{A_0},\ \frac{2W^2}{A_0}\Bigr],
\qquad
\sigma^2k_0\ \asymp\ V^2 ,
\]
so $\min\{V^2,\sigma^2k_0\}\asymp V^2$. For $F\asymp V^2$: the case hypothesis and $\bar k\le\log(m+1)$ give $W\le\sqrt{A_0\log(m+1)}$, so property (c) yields $\Lambda\ge\tfrac12(1+\log m)$ and
\[
\sigma^2W\sqrt\Lambda\ \ge\ V^2\,\frac{\sqrt\Lambda}W\ \ge\
V^2\,\frac{\sqrt{(1+\log m)/2}}{\sqrt{A_0(1+\log m)}}\ =\
\frac{V^2}{\sqrt{2A_0}},
\]
the denominator by $\log(m+1)\le1+\log m$. Also $\sigma^2m\ge\sigma^2W^2=V^2$, since property (d) gives $m\ge A_0\log(m+1)\ge W^2$.

\emph{Case $W^2>A_0\bar k$ and $W<m$.} Now $W>\sqrt{A_0}$ and $\log(m/W)>0$, so $\Lambda=1+\log(m/W)$. Set
\[
\kappa\ :=\ W\sqrt\Lambda,
\qquad
c_1\ :=\ \frac1{2\sqrt{A_0}} .
\]
\emph{No crossing below $c_1\kappa$:} fix an integer $k\le\min\{c_1\kappa,K\}$, so that the dimension clause is silent at $k$. If $k\le\bar k$, the capped phase and the case hypothesis give $V^2\delta_k=\sigma^2W^2>A_0\sigma^2\bar k\ge A_0\sigma^2k$. If $\bar k<k\le c_1\kappa$, monotonicity of $\ell$ gives
\[
\ell(k)\ \ge\ \ell(c_1\kappa)\ \ge\ \log\frac m{c_1W\sqrt\Lambda}
\ =\ (\Lambda-1)+\log(2\sqrt{A_0})-\tfrac12\log\Lambda
\ \ge\ \frac\Lambda4+1,
\]
the last step because $\log(2\sqrt{A_0})=\log24>3$ and $\tfrac34\Lambda+1\ge\tfrac12\log\Lambda$ for every $\Lambda\ge1$; hence, using $A_0c_1^2=\tfrac14$,
\[
V^2\delta_k\ =\ \sigma^2\,\frac{W^2\ell(k)}k\ >\
\sigma^2\,\frac{W^2\Lambda}{4k}\ =\
A_0\sigma^2\,\frac{(c_1\kappa)^2}k\ \ge\ A_0\sigma^2k .
\]
So no clause fires at $k$, and $k_0>\min\{c_1\kappa,K\}$. \emph{Crossing by $\lceil\kappa\rceil$:} if $\lceil\kappa\rceil\le K$, then $k_1:=\lceil\kappa\rceil$ has
\begin{gather*}
\ell(k_1)\ \le\ \Bigl[\log\frac{2m}\kappa\Bigr]_+\ \le\
\log2+(\Lambda-1)-\tfrac12\log\Lambda\ <\ \Lambda,
\\
V^2\delta_{k_1}\ \le\ \sigma^2\,\frac{W^2\Lambda}{k_1}\ \le\
\sigma^2\kappa\ \le\ A_0\sigma^2k_1 ,
\end{gather*}
so the crossing clause fires by $k_1$ and $k_0\le\lceil\kappa\rceil$.

If $\kappa\le K$: then $c_1\kappa<k_0\le\lceil\kappa\rceil\le2\kappa$, using $\kappa\ge W>1$, so $\sigma^2k_0\asymp\sigma^2\kappa =\sigma^2W\sqrt\Lambda$, the middle branch of $F$. Moreover, by property (a) and the case hypothesis,
\[
\Lambda\ \le\ 1+\log m\ \le\ 1+2\bar k\ <\ 1+\frac{2W^2}{A_0}\ \le\
\frac{3W^2}{A_0},
\qquad\text{so}\qquad
\sigma^2\kappa\ \le\ \sqrt{\tfrac3{A_0}}\,V^2\ \le\ \tfrac16\,V^2 :
\]
hence $\min\{V^2,\sigma^2k_0\}\asymp\sigma^2\kappa$, and $F\asymp\sigma^2\kappa$ as well: the first branch is at least $6\sigma^2\kappa$, and the third is at least the second because $\kappa\le K\le m$. If $\kappa>K$: then $k_0>\min\{c_1\kappa,K\}\ge c_1K$ and $k_0\le K+1\le2K$, so $\sigma^2k_0\asymp\sigma^2K\asymp\sigma^2m$, using $K\ge(m+1)/(2e^2)$. Also $V^2>\sigma^2m$: squaring $W\sqrt\Lambda=\kappa>K\ge m/(2e^2)$ gives
\[
V^2\ =\ \sigma^2W^2\ >\ \frac{\sigma^2m^2}{4e^4\Lambda}\ \ge\
\sigma^2m ,
\]
by property (d) and $\Lambda\le1+\log m$. Hence $\min\{V^2,\sigma^2k_0\}\asymp\sigma^2m$, and $F\asymp\sigma^2m$: its first branch exceeds $\sigma^2m$, and its second is $\sigma^2\kappa>\sigma^2K\ge\sigma^2m/(2e^2)$.

\emph{Case $W\ge m$.} Then $\Lambda=1$ and $F=\sigma^2\min\{W^2,W,m\}=\sigma^2m$. For every integer $k\le K$,
\[
\ell(k)\ \ge\ \ell(K)\ =\ \log\frac{m+1}K\ \ge\ \log C_\star\ =\ 2,
\]
so $\min\{k,\ell(k)\}\ge1$ and $V^2\delta_k\ge\sigma^2W^2/k>A_0\sigma^2k$ whenever $k<W/\sqrt{A_0}$: no clause fires below $\min\{W/\sqrt{A_0},\,K+1\}\ge m/12$, the first entry because $W\ge m$ and $\sqrt{A_0}=12$, the second because $K+1>(m+1)/e^2$ and $e^2<12$. Hence $m/12\le k_0\le K+1\le m$ and $\sigma^2k_0\asymp\sigma^2m$. Since $V^2=\sigma^2W^2\ge\sigma^2m^2\ge\sigma^2m$, also $\min\{V^2,\sigma^2k_0\}\asymp\sigma^2m\asymp F$.

The four cases cover $m\ge m_0$: if $W\ge m$ the fourth applies, and otherwise $W^2\le A_0$, $A_0<W^2\le A_0\bar k$, and $W^2>A_0\bar k$ partition the range.
\end{proof}

On the complete binary tree the covering counts decay exponentially in the radius. Radius therefore trades against entropy linearly, and the profile is governed by $\log(n/k)$: quadratic in that logarithm when the budget exceeds it, and truncated at the budget otherwise.

\begin{lemma}[Binary profile]\label{slem:binary}
Let $B_h$ be the complete binary tree of height $h\ge1$, with $n=2^{h+1}-1$ vertices, $h_{B_h}=h$, and $H_{B_h}=2h$. For integers $0\le q<h$,
\[
2^{h-q}\ \le\ \Nup(q)\ \le\ 3\cdot2^{h-q},
\]
and for every integer $1\le k\le n/C_\star$, with $\beta:=\log(n/k)$, which is at least $2$ in this range,
\[
c_B\,\beta\min\{k,\beta\}\ \le\ \alpha_k(B_h)\ \le\
C_B\,\beta\min\{k,\beta\},
\qquad
c_B:=\frac1{8\log2},
\quad
C_B:=3 .
\]
\end{lemma}

\begin{proof}
\emph{Covering counts.} An ancestor of a depth-$h$ leaf within distance $q$ has depth at least $h-q$, and a vertex at depth $d\ge h-q$ has $2^{h-d}\le2^q$ leaf descendants; covering all $2^h$ leaves therefore needs at least $2^{h-q}$ centers. Conversely, let $R$ consist of the root together with every vertex whose depth is of the form $D_j:=h-q-j(q+1)\ge0$, $j\ge0$. A selected vertex covers exactly the depths $D_j,\dots,D_j+q$ of its subtree; consecutive selected depths differ by $q+1$, so the covered bands tile $\{D_{j^*},\dots,h\}$, where $D_{j^*}$ is the smallest selected depth, and $D_{j^*}\le q$, since otherwise $D_{j^*}-(q+1)\ge0$ would be selected as well; the root covers the depths below $D_{j^*}$. Hence $R$ is an ancestor $q$-net, of size
\[
1+\sum_{j:\,D_j\ge0}2^{D_j}\ \le\
1+2^{h-q}\sum_{j\ge0}2^{-j(q+1)}\ \le\ 1+2\cdot2^{h-q}\ \le\
3\cdot2^{h-q} .
\]

\emph{Profile, upper bound.} Since $2^{h+1}=n+1$ and $\log(n+1)\le\log n+\tfrac13$ for $n\ge3$, every integer $0\le q<h$ has, with $y:=q+1$,
\begin{align*}
\log\frac{\Nup(q)}k\ &\le\ (h-q)\log2+\log3-\log k
\\
&=\ \log(n+1)+\log3-\log k-y\log2
\ \le\ \beta+\tfrac32-y\log2 ,
\end{align*}
and $\beta+\tfrac32\le\tfrac74\beta$ because $\beta\ge2$. Substituting $u:=y\log2$ in the discrete formula (2.3) and writing $B:=\tfrac74\beta$,
\begin{align*}
\alpha_k(B_h)\ \le\ \frac1{\log2}\,\sup_{u>0}\
u\,\min\bigl\{k,\ (B-u)_+\bigr\}
\ &\le\ \frac1{\log2}\,B\,\min\Bigl\{\frac B4,\ k\Bigr\}
\\
&\le\ \frac7{4\log2}\,\beta\min\{k,\beta\}\ \le\
C_B\,\beta\min\{k,\beta\}:
\end{align*}
if $k\ge B/4$ the supremum is at most $\sup_u u(B-u)_+=B^2/4$, while if $k<B/4$ the product is at most $uk\le(B-k)k$ for $u\le B-k$, and at most $u(B-u)_+\le(B-k)k$ beyond, since $u(B-u)_+$ is nonincreasing past $B/2<B-k$.

\emph{Profile, lower bound.} Set $y:=\lfloor\beta/(2\log2)\rfloor$. From $\beta\ge2$: $\beta/(2\log2)\ge1/\log2>1$, so $y\ge1$ and, halving, $y\ge\beta/(4\log2)$; and $y\le\beta/(2\log2)\le\log n/(2\log2)\le(h+1)/2\le h$, so $q:=y-1$ is an admissible radius. At this radius,
\[
\log\frac{\Nup(q)}k\ \ge\ (h-q)\log2-\log k
\ =\ \log(n+1)-\log k-y\log2\ \ge\ \beta-\frac\beta2\ =\ \frac\beta2,
\]
so, using $\min\{k,\beta/2\}\ge\tfrac12\min\{k,\beta\}$, the term at $q$ is at least
\[
y\,\min\Bigl\{k,\ \frac\beta2\Bigr\}\ \ge\
\frac\beta{4\log2}\cdot\frac{\min\{k,\beta\}}2\ =\
c_B\,\beta\min\{k,\beta\} . \qedhere
\]
\end{proof}

\begin{proof}[Proof of Corollary 3]
By Theorem~1 and $H_{B_h}=2h$, $R^*_{B_h}(V,\sigma)\asymp\min\{V^2h,\sigma^2k_0\}$ after absorbing the factor two. Write $\beta(k):=\log(n/k)$, as in \Cref{slem:binary}, and
\[
\Lambda\ :=\ 1+\bigl[\log(n/W)\bigr]_+,
\qquad
F\ :=\ \sigma^2\min\bigl\{W^2h,\ W\Lambda,\ n\bigr\},
\]
the right side of Corollary~3.

\emph{Bounded trees.} For any universal $n_0$ the corollary holds for $n\le n_0$ with constants depending only on $n_0$: from $1\le k_0\le K+1\le n_0$, $\min\{V^2h,\sigma^2k_0\}\asymp\min\{V^2h,\sigma^2\}$, while $F\le\min\{V^2h,\sigma^2n\}\le n_0\min\{V^2h,\sigma^2\}$ and every branch of $F$ dominates $\min\{V^2h,\sigma^2\}$; for the middle branch, when $W\le1$ this is $\sigma^2W\Lambda\ge\log2\cdot\sigma^2W^2h$, using $W\ge W^2$ and $\Lambda\ge1+\log n\ge h\log2$, and when $W>1$ it is $\sigma^2W\Lambda\ge\sigma^2$. Now fix a universal $n_0$ (one may take $n_0=2^7$) such that for every $n\ge n_0$: $h\le K$, $n\le2e^2K$, and $\log h\le\tfrac{\log2}2\,h$. Assume $n\ge n_0$.

\emph{The diameter regime $W\le1$.} We claim $k_0\ge c\,W^2h$ with $c:=c_B\log2/(4A_0)$. This is trivial when $cW^2h<1$; otherwise fix an integer $k\le cW^2h\le ch$. Then
\[
\beta(k)\ \ge\ \log\frac{2^h}h\ =\ h\log2-\log h\ \ge\
\frac{\log2}2\,h\ \ge\ k ,
\]
the last step because $c\le\log2/2$; so \Cref{slem:binary} gives
\[
V^2\delta_k\ \ge\ c_B\,\sigma^2W^2\beta(k)\ \ge\
\frac{c_B\log2}2\,\sigma^2W^2h\ >\ A_0\sigma^2\,cW^2h\ \ge\
A_0\sigma^2k ,
\]
and the dimension clause is silent at $k\le h\le K$: no clause fires at any $k\le cW^2h$, proving the claim. In either case $\sigma^2k_0\ge c\,V^2h$, so
\[
\min\{V^2h,\ \sigma^2k_0\}\ \asymp\ V^2h .
\]
And $F\asymp V^2h$: its first branch is $V^2h$; its middle branch dominates $\log2\cdot V^2h$ as above; and $\sigma^2n\ge\sigma^2h\ge V^2h$. This proves the corollary for $W\le1$.

\emph{The entropy and dimension regimes $W>1$.} Put $\kappa:=W\Lambda\ge1$. First, if $\lceil\kappa\rceil\le K$, the crossing clause fires by $k_1:=\lceil\kappa\rceil$: there $W\le k_1\le K\le n$, so $\beta(k_1)\le\log(n/W)\le\Lambda$, and \Cref{slem:binary} gives
\[
V^2\delta_{k_1}\ \le\ C_B\,\sigma^2\,\frac{W^2\beta(k_1)^2}{k_1}
\ \le\ C_B\,\sigma^2\,\frac{W^2\Lambda^2}\kappa\ =\
C_B\,\sigma^2\kappa\ \le\ A_0\sigma^2k_1 ;
\]
hence $k_0\le\lceil\kappa\rceil$ whenever $\lceil\kappa\rceil\le K$. Second, with $c:=c_B/(8A_0)$, no clause fires at any integer $k\le\min\{c\kappa,K\}$. The dimension clause is silent since $k\le K$. For the crossing clause, note first that $\beta(k)\ge\Lambda/4$: if $W\ge n$ this reads $\beta(k)\ge2>\tfrac14=\Lambda/4$, and if $W<n$, then $k\le cW\Lambda$ with $c\le1$ gives
\[
\beta(k)\ \ge\ \log\frac n{W\Lambda}\ =\ (\Lambda-1)-\log\Lambda
\ \ge\ \frac\Lambda4\quad(\Lambda\ge3),
\qquad
\beta(k)\ \ge\ 2\ \ge\ \frac34\ >\ \frac\Lambda4\quad(\Lambda<3),
\]
the first branch because $\tfrac34\Lambda\ge1+\log\Lambda$ for $\Lambda\ge3$. Now split on $\min\{k,\beta(k)\}$. If $k\le\beta(k)$, then, using $k\le cW\Lambda\le cW^2\Lambda$ and $A_0c=c_B/8$,
\[
V^2\delta_k\ \ge\ c_B\,\sigma^2W^2\beta(k)\ \ge\
\frac{c_B}4\,\sigma^2W^2\Lambda\ >\ A_0\sigma^2\,cW^2\Lambda\ \ge\
A_0\sigma^2k .
\]
If $k>\beta(k)$, then
\[
V^2\delta_k\ \ge\ c_B\,\sigma^2\,\frac{W^2\beta(k)^2}k\ \ge\
\frac{c_B}{16}\,\sigma^2\,\frac{\kappa^2}k\ >\ A_0\sigma^2k ,
\]
the last step because $k\le c\kappa$ and $c=c_B/(8A_0)<\sqrt{c_B}/(4\sqrt{A_0})$, so that $k^2\le c^2\kappa^2<\tfrac{c_B}{16A_0}\kappa^2$. Hence $k_0>\min\{c\kappa,K\}$.

Combining the two tests: if $\kappa\le K$, then $c\kappa<k_0\le\lceil\kappa\rceil\le2\kappa$, so $\sigma^2k_0\asymp\sigma^2\kappa=\sigma^2\min\{W\Lambda,n\}$, using $\kappa\le K\le n$; if $\kappa>K$, then $k_0>\min\{c\kappa,K\}\ge cK$ and $k_0\le K+1\le2K$, so $\sigma^2k_0\asymp\sigma^2K\asymp\sigma^2n\asymp \sigma^2\min\{W\Lambda,n\}$, using $n\le2e^2K$. In both regimes
\[
\sigma^2k_0\ \asymp\ \sigma^2\min\{W\Lambda,\ n\},
\quad\text{hence}\quad
\min\{V^2h,\ \sigma^2k_0\}\ \asymp\
\sigma^2\min\{W^2h,\ W\Lambda,\ n\}\ =\ F ,
\]
which proves the corollary for $W>1$.
\end{proof}

It remains to prove the profile bounds (5.8) for the broom $T_L$ of Section~5, the handle $o=v_0,\dots,v_L$ with $m$ leaf children at $v_L$, so that $h_{T_L}=L+1$. Its ancestor-covering numbers were computed in (5.7) there: $\Nup(0)=n$, and $\Nup(q)=\lceil(L+2)/(q+1)\rceil$ for integers $1\le q\le L$.

\begin{lemma}[Broom profile]\label{slem:broom}
For the broom $T_L$ with $L\ge1$ and every integer $k\ge1$,
\[
\delta_k(T_L)\ \le\ \max\Bigl\{1,\ \frac{2(L+2)}{ek^2}\Bigr\},
\]
and whenever $k\le(L+2)/(2e)$,
\[
\delta_k(T_L)\ \ge\ \frac{L+2}{2ek^2}\,.
\]
\end{lemma}

\begin{proof}
By the discrete formula (2.3), with $h_{T_L}=L+1$,
\[
\alpha_k(T_L)\ =\ \max_{0\le q\le L}\ (q+1)\,
\min\Bigl\{k,\ \Bigl[\log\frac{\Nup(q)}k\Bigr]_+\Bigr\}.
\]

\emph{Upper bound.} The term at $q=0$ is at most $k$ and contributes at most $1$ to $\delta_k=\alpha_k/k$. Consider $1\le q\le L$ with a nonvanishing term and put $y:=q+1$. Then $\lceil(L+2)/y\rceil>k$, hence $(L+2)/y>k$, since a real number at most the integer $k$ has ceiling at most $k$; consequently
\[
\Nup(q)\ <\ \frac{L+2}y+1\ <\ \frac{2(L+2)}y\,,
\]
and the term is at most
\[
y\,\log\frac{2(L+2)}{ky}\ \le\ \sup_{u>0}\ u\log\frac Au\ =\
\frac Ae\,,
\qquad A:=\frac{2(L+2)}k\,,
\]
the supremum attained at $u=A/e$. Hence $\alpha_k\le\max\{k,\ 2(L+2)/(ek)\}$, and dividing by $k$ proves the upper bound.

\emph{Lower bound.} Set $x:=(L+2)/(ek)$, at least $2$ in the stated range, and $y:=\lfloor x\rfloor$, so that $y\ge2$ and $y>x-1\ge x/2$. The radius $q:=y-1$ lies in $[1,L]$: $q\ge1$ since $y\ge2$, and $q\le L$ since $y\le x\le(L+2)/e<L+1$ for every $L\ge1$. Then
\[
\Nup(q)\ =\ \Bigl\lceil\frac{L+2}y\Bigr\rceil\ \ge\ \frac{L+2}y\
\ge\ \frac{L+2}x\ =\ ek,
\]
so $[\log(\Nup(q)/k)]_+\ge1$, and the term at $q$ is at least $y\cdot\min\{k,1\}=y\ge x/2$. Hence $\alpha_k\ge(L+2)/(2ek)$, and dividing by $k$ proves the lower bound.
\end{proof}

Four computations from the lower-bound proof of Theorem~5 remain, in the notation fixed there: $V=L^{1/4}$, $\sigma=1$, $m\ge e^{20L^{5/2}}$, and the statistics $W=\sum_{j=1}^L Z_{v_j}$, $Q=\sum_{j=1}^L[Z_{v_j}]_+^2$, $M=\max_{1\le i\le m}Z_{w_i}$.

\emph{The covering counts (5.7).} For the lower bound, the root-to-leaf path $v_0,\dots,v_L,w_1$ has $L+2$ vertices; an ancestor center on it covers at most $j+1$ consecutive ones, and a center at any other leaf covers none of them, so $\Nup(j)\ge\lceil(L+2)/(j+1)\rceil$. For the upper bound, the single center $v_{L+1-j}$ covers $v_{L+1-j},\dots,v_L$ together with every leaf, each at distance exactly $j$, and $\lceil(L-j+1)/(j+1)\rceil$ further centers along the prefix $v_0,\dots,v_{L-j}$ complete an ancestor $j$-net; the total is $1+\lceil(L-j+1)/(j+1)\rceil=\lceil(L+2)/(j+1)\rceil$.

\emph{The cap (5.3).} For $h_j\ge0$ the handle term obeys $2Z_{v_j}h_j-h_j^2\le[Z_{v_j}]_+^2$: the left side is nonpositive when $Z_{v_j}\le0$, and equals $Z_{v_j}^2-(h_j-Z_{v_j})^2$ otherwise. The leaf part is at most $2M\sum_i\ell_i\le2Mh_L\le MV$ by (5.1), the last step because the case under consideration has $h_L\le V/2$.

\emph{The forcing event $E=\{W\ge-2\sqrt L\}\cap\{Q\le2L\} \cap\{M\ge4L^{5/4}\}$ has probability at least $\tfrac38$.} On $E$,
\[
(L+1)V-2W+\frac QV\ \le\
L^{5/4}+L^{1/4}+4L^{1/2}+2L^{3/4}\ <\ 4L^{5/4}\ \le\ M,
\]
the middle inequality being equivalent to $L^{-1}+4L^{-3/4}+2L^{-1/2}<3$, whose left side is decreasing in $L$ and below $3$ at $L=4$; so the forcing inequality (5.4) holds on $E$. Since $W\sim N(0,L)$, Markov's inequality applied to $W^2$ gives $\mathbb P(W<-2\sqrt L)\le\tfrac14$; since $\E Q=L/2$, by $\E[Z_{v_j}]_+^2=\tfrac12$, the same inequality gives $\mathbb P(Q>2L)\le\tfrac14$; hence $\mathbb P\{W\ge-2\sqrt L,\ Q\le2L\}\ge\tfrac12$. For the leaves, set $t:=4L^{5/4}$ and integrate the standard normal density, decreasing on $[t,t+t^{-1}]$, over that interval:
\[
\mathbb P\bigl(Z_{w_1}\ge t\bigr)\ \ge\
\frac{1}{t\sqrt{2\pi}}\,e^{-(t+t^{-1})^2/2}\ \ge\
c_G\,L^{-5/4}\,e^{-8L^{5/2}},
\qquad
c_G:=\frac{e^{-33/32}}{4\sqrt{2\pi}}\ ,
\]
the exponent being $\tfrac12t^2+1+\tfrac12t^{-2}\le 8L^{5/2}+\tfrac{33}{32}$ by $t^2=16L^{5/2}\ge16$. Multiplying by $m\ge e^{20L^{5/2}}$,
\[
m\,\mathbb P\bigl(Z_{w_1}\ge t\bigr)\ \ge\
c_G\,L^{-5/4}\,e^{12L^{5/2}}\ \ge\ \log4 ,
\]
the middle expression being increasing in $L$ and, at $L=1$, equal to $c_Ge^{12}\ge e^{12-33/32-3}\ge e^{7}$, using $4\sqrt{2\pi}\le e^{3}$. The leaves are independent of one another, so $\mathbb P(M<t)=(1-\mathbb P(Z_{w_1}\ge t))^{m}\le e^{-m\mathbb P(Z_{w_1}\ge t)}\le\tfrac14$, and independent of the handle, whence $\mathbb P(E)\ge\tfrac12\cdot\tfrac34=\tfrac38$.

\emph{The crossing index satisfies $k_0\asymp\sqrt L$.} Set $\kappa:=(V^2(L+2)/(2eA_0))^{1/3}$, which is $\asymp\sqrt L$ since $V^2=\sqrt L$. Every integer $k<\kappa$ lies in the lower-bound range of \Cref{slem:broom}: $(2e)^2\sqrt L\le A_0(L+2)^2$ gives $\kappa\le(L+2)/(2e)$, and there $V^2\delta_k\ge V^2(L+2)/(2ek^2)>A_0k$, the strict inequality being $k^3<\kappa^3$ rearranged, while $n\ge e^{20L^{5/2}}$ keeps the dimension clause of (1.5) silent; hence $k_0\ge\kappa$. At $k_+:=\lceil4^{1/3}\kappa\rceil$ the upper bound of \Cref{slem:broom} gives $V^2\delta_{k_+}\le\max\{V^2,2V^2(L+2)/(ek_+^2)\}\le A_0k_+$, the first branch since $A_0k_+\ge A_0^{2/3}(2/e)^{1/3}\sqrt L\ge V^2$ and the second since $k_+^3\ge4\kappa^3$; so the crossing fires by $k_+$ and $\kappa\le k_0\le\lceil4^{1/3}\kappa\rceil$, whence $k_0\asymp\sqrt L$. Finally $\kappa\le(L+2)/(2e)$ gives $k_0\le\lceil4^{1/3}\kappa\rceil\le L+3\le\sqrt L\,(L+1)=V^2H_{T_L}$, so the minimum in Theorem~1 is the information branch and $R^*_{T_L}(L^{1/4},1)\asymp k_0\asymp\sqrt L$.

%%%%%%%%%%%%%%%%%%%%%%%%%%%%%%%%%%%%%%%%%%%%%%%%%%%%%%%%%%%%%%%%%%%%%%%%%%%%%%
\section{Exact Evaluation of the Aggregate}\label{ssec:evaluation}
%%%%%%%%%%%%%%%%%%%%%%%%%%%%%%%%%%%%%%%%%%%%%%%%%%%%%%%%%%%%%%%%%%%%%%%%%%%%%%

This section proves the statements deferred from Sections~4.1 and~4.2: the bound on the normalizing sum of the prior, the completion of the proof of Lemma~4.1, the risk of the three elementary branches of (4.3), the correctness and cost of the exact evaluation of the aggregate (4.2), and the batched evaluation behind the instance sentence of Section~4.2, including the softly linear evaluation on stars.

Throughout, $K:=\lfloor n/C_\star\rfloor$; except in the treatment of the elementary branches, $k$ is an integer with $2\le k\le K$. The active support $A_k$, the charges $\widetilde\omega$, the code length $\Gamma$, and the level weights $m_j=2^j$ for $0\le j\le J$ are those of Section~3.1; $\mathcal X_k$ is the class (4.1) of integer states of Section~4.1; and, as in Section~4.2 there, $\varphi_v(x)=e^{-(Y_v-bx)^2/(4\sigma^2)}$ with the shorthand $b:=V/k$.

\begin{lemma}[Normalizing sum]\label{slem:normalizer}
For every integer $2\le k\le K$,
\[
\sum_{x\in\mathcal X_k}e^{-2\Gamma(x)}\ \le\ e^{k}.
\]
\end{lemma}

\begin{proof}
The map $x\mapsto z$ is injective on state vectors, since the subtree sums $x_v=\sum_{u\succeq v}z_u$ recover $x$, and $\Gamma(x)$ depends only on $z$. Dropping the constraints $x_o=k$ and $0\le x_v\le3k$ therefore bounds the sum by a product over the support:
\[
\sum_{x\in\mathcal X_k}e^{-2\Gamma(x)}
\ \le\ \sum_{z\in\mathbb Z^{A_k}}
\prod_{v\in A_k}e^{-2(|z_v|+\widetilde\omega(v)\ind\{z_v\ne0\})}
\ =\ \prod_{v\in A_k}
\Bigl(1+\tfrac2{e^2-1}\,e^{-2\widetilde\omega(v)}\Bigr),
\]
the factor at $v$ being $1+e^{-2\widetilde\omega(v)}\cdot 2\sum_{t\ge1}e^{-2t}$, with $2\sum_{t\ge1}e^{-2t}=\tfrac2{e^2-1}$. The free level $R_0$, where $\widetilde\omega=0$, has $|R_0|\le ke^{m_0}=ek$ vertices by \Cref{slem:netsizes}(\ref{sit:netsize}), and $\log(1+\tfrac2{e^2-1})\le\tfrac3{10}$, so it contributes at most $\tfrac3{10}ek<0.82\,k$ to the logarithm. The charged vertices split into the groups $G_j:=\{v\in A_k\setminus R_0:\ \omega(v)=m_j\} \subseteq R_j$, $1\le j\le J$, with $|G_j|\le2ke^{m_j}$ by \Cref{slem:netsizes}(\ref{sit:union}); since $\log(1+y)\le y$, their contribution is at most
\[
\sum_{j\ge1}2ke^{m_j}\cdot\frac2{e^2-1}\,e^{-2m_j}
\ =\ \frac{4k}{e^2-1}\sum_{j\ge1}e^{-2^j}
\ \le\ \frac{4k}{e^2-1}\cdot0.16\ <\ 0.11\,k,
\]
using $\sum_{j\ge1}e^{-2^j}<e^{-2}+e^{-4}+2e^{-8}<0.16$. The logarithm of the product is below $0.93\,k$.
\end{proof}

In the setting of Lemma~4.1, write
\[
\mathcal Z(y)\ :=\ \sum_{g\in F}\pi_g\,e^{-\|y-g\|_2^2/(4\sigma^2)}
\]
for the normalizing sum of the weights, positive and smooth on $\bbR^N$.

\begin{proof}[Completion of the proof of Lemma 4.1]
\emph{The Jacobian.} The logarithm of a weight is $\log\pi_f-\|y-f\|_2^2/(4\sigma^2)-\log\mathcal Z(y)$; differentiating it in $y_i$ gives
\[
\frac{\partial\log\widehat\pi_f}{\partial y_i}
\ =\ -\frac{y_i-f_i}{2\sigma^2}
+\sum_{g\in F}\widehat\pi_g(y)\,\frac{y_i-g_i}{2\sigma^2}
\ =\ \frac{f_i-\widehat f_i(y)}{2\sigma^2}\ ,
\]
that is, $\partial\widehat\pi_f/\partial y_i=\widehat\pi_f\,(f_i-\widehat f_i)/(2\sigma^2)$. Multiplying by $f_j$ and summing over $f$,
\[
\frac{\partial\widehat f_j}{\partial y_i}
\ =\ \frac1{2\sigma^2}\Bigl(\sum_{f\in F}\widehat\pi_f\,f_if_j
-\widehat f_i\,\widehat f_j\Bigr)
\ =\ \frac1{2\sigma^2}\operatorname{Cov}_{\widehat\pi(y)}(f_i,f_j),
\]
the Jacobian identity of the main-text sketch; taking $i=j$ and summing,
\begin{equation}\label{eq:traceid}
2\sigma^2\operatorname{div}\widehat f_\pi(y)
\ =\ \operatorname{tr}\operatorname{Cov}_{\widehat\pi(y)}(f)
\ =\ \sum_{f\in F}\widehat\pi_f(y)\,
\bigl\|f-\widehat f_\pi(y)\bigr\|_2^2\ \ge\ 0 .
\end{equation}

\emph{The domain of Stein's identity.} The aggregate takes values in the convex hull of the finite set $F$, so it is bounded; its partial derivatives are posterior covariances of coordinates of $F$, bounded uniformly in $y$; and both are smooth, since $\mathcal Z>0$ everywhere. For such a map, Stein's identity \citep{stein1981} gives
\[
\E_\mu\bigl\|\widehat f_\pi(Y)-\mu\bigr\|_2^2\ =\
\E\bigl\|\widehat f_\pi(Y)-Y\bigr\|_2^2
+2\sigma^2\,\E\operatorname{div}\widehat f_\pi(Y)\ -\ N\sigma^2 .
\]

\emph{The cancellation.} At fixed $y$, expanding $\|y-f\|_2^2$ around $\widehat f_\pi(y)$ and averaging under $\widehat\pi(y)$ kills the cross term, since $\sum_f\widehat\pi_f\,(f-\widehat f_\pi)=0$:
\[
\sum_{f\in F}\widehat\pi_f(y)\,\|y-f\|_2^2
\ =\ \bigl\|y-\widehat f_\pi(y)\bigr\|_2^2
+\sum_{f\in F}\widehat\pi_f(y)\,\bigl\|f-\widehat f_\pi(y)\bigr\|_2^2 .
\]
The last sum is $2\sigma^2\operatorname{div}\widehat f_\pi(y)$ by \Cref{eq:traceid}, so substituting the expansion into Stein's identity cancels the divergence terms and leaves the risk identity of the main-text sketch.

\emph{The variational bound.} The relative entropy of probability vectors on $F$ is $\operatorname{KL}(\rho\,\|\,\pi):=\sum_f\rho_f\log(\rho_f/\pi_f)$, with $0\log0:=0$. For any such $\rho$, inserting the definition of $\widehat\pi(y)$ gives the identity
\[
\sum_{f\in F}\rho_f\,\|y-f\|_2^2
+4\sigma^2\operatorname{KL}(\rho\,\|\,\pi)
\ =\ 4\sigma^2\operatorname{KL}\bigl(\rho\,\|\,\widehat\pi(y)\bigr)
\ -\ 4\sigma^2\log\mathcal Z(y),
\]
so the left side is minimized exactly at $\rho=\widehat\pi(y)$. Comparing the minimizer with the point mass at $f_0$, and dropping the nonnegative divergence $\operatorname{KL}(\widehat\pi(y)\,\|\,\pi)$ from the minimum,
\[
\sum_{f\in F}\widehat\pi_f(y)\,\|y-f\|_2^2
\ \le\ \|y-f_0\|_2^2+4\sigma^2\log\frac1{\pi_{f_0}}\ ,
\]
the pathwise bound of the main body, whose expectation completes the proof there.
\end{proof}

The elementary branches of (4.3) are next. Their analysis rests on two facts: the root-corrected observation matches $\mu(o)=V$ exactly, and $Vp_o$ is within $V\sqrt{h_T}$ of every body point.

\begin{lemma}[Elementary branches]\label{slem:branches}
Write $k:=\kalg(T,V,\sigma)$. In each of the three elementary branches of (4.3),
\[
\sup_{\mu\in\mathcal F_V(T)}
\E_\mu\|\widehat\mu-\mu\|_2^2\ \le\
C\min\{V^2H_T,\ \sigma^2k\}.
\]
\end{lemma}

\begin{proof}
For every $\mu=\sum_u\lambda_u\,Vp_u\in\mathcal F_V(T)$, convexity of the norm and Lemma~2.2 give $\|\mu-Vp_o\|_2\le\max_u\|Vp_u-Vp_o\|_2=V\sqrt{h_T}$.

\emph{Diameter branch ($V^2H_T\le\sigma^2k$).} The estimator is the constant $Vp_o$, so its risk is at most $V^2h_T\le V^2H_T$, which is the minimum under the branch condition.

\emph{Dimension branch ($V^2H_T>\sigma^2k$ and $k>K$).} Since $k\le K+1$ always, $k=K+1$, and $K=\lfloor n/C_\star\rfloor$ gives $n<C_\star(K+1)=C_\star k$. The root coordinate of $\widetilde Y$ is exact, so
\[
\E_\mu\|\widetilde Y-\mu\|_2^2\ =\ (n-1)\,\sigma^2\ <\
C_\star\,\sigma^2k,
\]
and the minimum is $\sigma^2k$ under the branch condition.

\emph{First-budget branch ($V^2H_T>\sigma^2k$ and $k=1\le K$).} Here $n\ge C_\star$, so $n\ge2$, and, as in Section~3.3, every ancestor $r$-net with $r<h_T$ has at least two elements: the root lies in every net and covers no deepest vertex. Hence $\Nup(r)\ge2$ for all $r<h_T$, and letting $r\uparrow h_T$ in (1.4) gives $\alpha_1\ge h_T\log2$. The dimension clause of (2.2) is silent at $k=1$, since $1\le K$, so the surrogate crossing clause fired there, $V^2\overline\delta_1\le2A_0\sigma^2$, and with Lemma~2.3,
\[
\sup_{\mu}\|\mu-Vp_o\|_2^2\ \le\ V^2h_T\ \le\
\frac{V^2\alpha_1}{\log2}\ =\ \frac{V^2\delta_1}{\log2}\ \le\
\frac{V^2\overline\delta_1}{\log2}\ \le\
\frac{2A_0}{\log2}\,\sigma^2 ;
\]
the minimum is $\sigma^2$ under the branch condition.
\end{proof}

The rest of the section evaluates the interior branch. Fix $2\le k\le K$ and collect the prior into local kernels: for $z\in\mathbb Z$ put
\begin{equation}\label{eq:kernel}
\kappa_v(z)\ :=\
\begin{cases}
1, & z=0,\\
e^{-2\widetilde\omega(v)}\,e^{-2|z|}, & z\ne0\ \text{and}\
v\in A_k,\\
0, & z\ne0\ \text{and}\ v\notin A_k .
\end{cases}
\end{equation}
For a state vector $x\in\{0,\dots,3k\}^{\sfV}$ with leaks $z_v=x_v-\sum_{c\in\ch(v)}x_c$, the product $\prod_v\kappa_v(z_v)$ equals $e^{-2\Gamma(x)}$ when the leaks vanish off $A_k$ and is zero otherwise. The denominator of (4.2) is therefore the total \emph{weight} $\prod_v\kappa_v(z_v)\,\varphi_v(x_v)$ of all state vectors with $x_o=k$, and it is positive: the vector with $x_o=k$ and all other states zero has its one leak at the root, where $\widetilde\omega(o)=0$ because $o\in R_0$, so its weight is $e^{-2k}\varphi_o(k)\prod_{v\ne o}\varphi_v(0)>0$. In the same terms, the coefficients of the message (4.5) are
\begin{equation}\label{eq:semantics}
P_v[x]\ =\ \sum_{\substack{x'\in\{0,\dots,3k\}^{T_v}\\ x'_v=x}}\
\prod_{w\in T_v}\kappa_w(z'_w)\,\varphi_w(x'_w),
\qquad
z'_w:=x'_w-\textstyle\sum_{c\in\ch(w)}x'_c ,
\end{equation}
the leaks computed within $T_v$.

Messages are polynomials of degree at most $3k$, but the product of many children's messages is not, and its degree must not be allowed to grow with their number. The device is a pair representation: a polynomial $F(\zeta)=\sum_{i\ge0}F[i]\zeta^i$ with nonnegative coefficients is carried as its truncation $F_0(\zeta):=\sum_{i\le3k}F[i]\zeta^i$ together with its \emph{tail}
\[
\tau_F\ :=\ \sum_{i>3k}F[i]\,e^{-2(i-3k)} ,
\]
the evaluation of the high part at $e^{-2}$, shifted to the cap. Messages have zero tail; tails arise along partial products of messages and stay scalars.

\begin{lemma}[Product of pairs]\label{slem:pairs}
Let $F,G$ be polynomials with nonnegative coefficients, of arbitrary degrees, given by their pairs. Then $(FG)_0$ is the truncation of $F_0G_0$, and
\begin{equation}\label{eq:tailprod}
\tau_{FG}\ =\ \sum_{q=3k+1}^{6k}[\zeta^q](F_0G_0)\,e^{-2(q-3k)}
\ +\ F_0(e^{-2})\,\tau_G\ +\ \tau_F\,G_0(e^{-2})
\ +\ e^{-6k}\,\tau_F\,\tau_G .
\end{equation}
The pair of $FG$ is therefore computable from the pairs of $F$ and $G$ by one multiplication of two polynomials of degree at most $3k$ and $O(k)$ further operations.
\end{lemma}

\begin{proof}
For $q\le3k$, $[\zeta^q](FG)=\sum_{i+j=q}F[i]G[j]$ involves only indices $i,j\le3k$, which gives the truncation claim. In $\tau_{FG}=\sum_{i+j>3k}F[i]G[j]e^{-2(i+j-3k)}$, partition the index pairs according to whether $i\le3k$ and whether $j\le3k$. The low--low pairs give the first term of \Cref{eq:tailprod}. The low--high pairs factor as $\sum_{i\le3k}F[i]e^{-2i}\cdot\sum_{j>3k}G[j]e^{-2(j-3k)} =F_0(e^{-2})\tau_G$, the high--low pairs symmetrically, and the high--high pairs as $e^{-6k}\tau_F\tau_G$, since $e^{-2(i+j-3k)}=e^{-2(i-3k)}e^{-2(j-3k)}e^{-6k}$. For the cost: the single multiplication $F_0G_0$ yields all coefficients through degree $6k$, and every remaining ingredient of \Cref{eq:tailprod}, including the evaluations $F_0(e^{-2})$ and $G_0(e^{-2})$, is $O(k)$ arithmetic on them.
\end{proof}

The recursion of Section~4.2 now takes its precise form. At each vertex the children's messages are merged pairwise in the pair representation, and one local pass applies the observation weight and the prior kernel.

\begin{lemma}[Message recursion]\label{slem:recursion}
For $v\in\sfV$ let $Q_v(\zeta):=\prod_{c\in\ch(v)}P_c(\zeta)$, an empty product being $1$, with coefficients $Q_v[s]$ and tail $\tau_v$, and set
\[
\psi^-_x\ :=\ \sum_{0\le s<x}e^{-2(x-s)}\,Q_v[s],
\qquad
\psi^+_x\ :=\ \sum_{s>x}e^{-2(s-x)}\,Q_v[s]
\qquad(0\le x\le3k).
\]
Then
\begin{equation}\label{eq:localpass}
P_v[x]\ =\
\begin{cases}
\varphi_v(x)\,Q_v[x], & v\notin A_k,\\[2pt]
\varphi_v(x)\,\bigl(Q_v[x]
+e^{-2\widetilde\omega(v)}(\psi^-_x+\psi^+_x)\bigr), & v\in A_k,
\end{cases}
\end{equation}
and the two sums obey the linear recurrences
\begin{equation}\label{eq:psirec}
\psi^-_0=0,\quad
\psi^-_x=e^{-2}\bigl(Q_v[x-1]+\psi^-_{x-1}\bigr);
\qquad
\psi^+_{3k}=\tau_v,\quad
\psi^+_x=e^{-2}\bigl(Q_v[x+1]+\psi^+_{x+1}\bigr).
\end{equation}
In particular the pair of $P_v$ is computed from the pairs of the
children's messages by $(|\ch(v)|-1)_+$ products of pairs and $O(k)$
further operations, and $P_o[k]$ is the denominator of (4.2).
\end{lemma}

\begin{proof}
Assignments on the subtrees of distinct children are independent,
and weights multiply, so $Q_v[s]$ is the total weight of the
assignments on the child subtrees whose states at the children sum
to $s$; here $s$ ranges over all of $\{0,1,\dots\}$, and the pair of
$Q_v$ is assembled by \Cref{slem:pairs}. A state $x$ at $v$
completes such a family to an assignment on $T_v$ with local weight
$\varphi_v(x)\kappa_v(x-s)$, so by \Cref{eq:semantics}
\[
P_v[x]\ =\ \varphi_v(x)\sum_{s\ge0}\kappa_v(x-s)\,Q_v[s].
\]
If $v\notin A_k$, the kernel forces $s=x\le3k$, which is \Cref{eq:localpass}; child sums beyond $3k$ contribute nothing. If $v\in A_k$, the kernel splits the sum at $s=x$ into $Q_v[x]+e^{-2\widetilde\omega(v)}(\psi^-_x+\psi^+_x)$ with
\[
\psi^+_x\ =\ \sum_{x<s\le3k}e^{-2(s-x)}Q_v[s]
\ +\ e^{-2(3k-x)}\,\tau_v ,
\]
the high child sums entering exactly through the tail. The recurrences \Cref{eq:psirec} are immediate from the definitions; the seed $\psi^+_{3k}=\tau_v$ is the identity just displayed at $x=3k$. Both passes cost $O(k)$, as does applying \Cref{eq:localpass}, and the leaf base case ($Q_v\equiv1$) is contained in the general one. At the root, \Cref{eq:semantics} with $x=k$ sums the weights of all state vectors with $x_o=k$, which is the denominator of (4.2).
\end{proof}

The pruning and the chain compression described in Section~4.2 modify this recursion; neither changes the posterior.

\begin{lemma}[Pruning and compression]\label{slem:compress}
\leavevmode
\begin{enumerate}
\item\label{sit:prune} If $T_w\cap A_k=\varnothing$, then every
state vector of nonzero weight vanishes on $T_w$, and the subtree
contributes the factor $\prod_{u\in T_w}\varphi_u(0)$ to every
weight; dropping it, and every other factor common to all
assignments, changes no posterior quantity. The estimator satisfies
$\widehat\mu_k(u)=0$ for $u\in T_w$.
\item\label{sit:chain} Let $\mathcal C$ be a maximal chain of
vertices not in $A_k$, each having exactly one child whose subtree
meets $A_k$ and each, except the lowest, followed in $\mathcal C$ by
that child; write $c_{\mathcal C}$ for that child of the lowest
vertex of $\mathcal C$. Every state vector of nonzero weight is constant on
$\mathcal C\cup\{c_{\mathcal C}\}$, and the chain contributes, after
the
state-independent factor
$\exp(-\sum_{w\in\mathcal C}Y_w^2/(4\sigma^2))$ is dropped, the
diagonal kernel
\begin{equation}\label{eq:chainkernel}
x\ \longmapsto\
\exp\Bigl(\frac{b\,x}{2\sigma^2}\sum_{w\in\mathcal C}Y_w
\ -\ \frac{|\mathcal C|\,b^2x^2}{4\sigma^2}\Bigr),
\end{equation}
generated in $O(k)$ operations from $|\mathcal C|$ and
$\sum_{w\in\mathcal C}Y_w$. The estimator satisfies
$\widehat\mu_k(w)=\widehat\mu_k(c_{\mathcal C})$ for
$w\in\mathcal C$.
\end{enumerate}
\end{lemma}

\begin{proof}
For (\ref{sit:prune}): on $T_w$ all kernels force $z\equiv0$, so leaf-up induction gives $x\equiv0$ on $T_w$ for every assignment of nonzero weight; the surviving factor is $\prod_{u\in T_w}\kappa_u(0)\varphi_u(0)$ as claimed. A factor common to every assignment multiplies numerator and denominator of each posterior expectation alike. Since $x_u=0$ with posterior probability one, $\widehat\mu_k(u)=b\,\E_Y[x_u]=0$.

For (\ref{sit:chain}): a vertex $w\in\mathcal C$ has $z_w=0$ and all of its child subtrees off the chain disjoint from $A_k$, hence carrying state zero, so $x_w$ equals the state of its unique child toward $A_k$; iterating down the chain, all these states equal $x_{c_{\mathcal C}}=:x$. The chain's likelihood factor is $\prod_{w\in\mathcal C}\varphi_w(x)$, and expanding the exponent,
\[
-\sum_{w\in\mathcal C}\frac{(Y_w-bx)^2}{4\sigma^2}
\ =\ -\sum_{w\in\mathcal C}\frac{Y_w^2}{4\sigma^2}
\ +\ \frac{bx}{2\sigma^2}\sum_{w\in\mathcal C}Y_w
\ -\ \frac{|\mathcal C|\,b^2x^2}{4\sigma^2},
\]
which is \Cref{eq:chainkernel} after the constant is dropped. Pathwise equality of the states gives $\E_Y[x_w]=\E_Y[x_{c_{\mathcal C}}]$, hence the coordinate claim.
\end{proof}

The forward computation of $P_o[k]$ is a straight-line program in the $n(3k+1)$ inputs $\varphi_v(x)$ whose gates are polynomial multiplications with their tail formulas \Cref{eq:tailprod}, the linear recurrences \Cref{eq:psirec}, and the coordinatewise products \Cref{eq:localpass}. The reverse sweep differentiates it.

\begin{lemma}[Reverse sweep]\label{slem:reverse}
All $n(3k+1)$ partial derivatives $\partial P_o[k]/ \partial\varphi_v(x)$ are computable in $O(1)$ times the cost of the forward computation, and
\begin{equation}\label{eq:margid}
\widehat\mu_k(v)\ =\ b\sum_{x=0}^{3k}x\cdot
\frac{\varphi_v(x)}{P_o[k]}\,
\frac{\partial P_o[k]}{\partial\varphi_v(x)},
\end{equation}
which is the display (4.6).
\end{lemma}

\begin{proof}
By \Cref{eq:semantics} at the root, $P_o[k]$ is a sum of monomials, one per state vector $x'$ with $x'_o=k$, each containing the input $\varphi_v(x'_v)$ exactly once and otherwise only kernel constants and inputs at other vertices. Hence $\varphi_v(x)\,\partial P_o[k]/\partial\varphi_v(x)$ is the total weight of the state vectors with $x'_v=x$, so the ratio in \Cref{eq:margid} is the posterior probability of $\{x_v=x\}$ and the sum is $b\,\E_Y[x_v]=\widehat\mu_k(v)$.

For the cost, propagate the derivatives of $P_o[k]$ with respect to every intermediate quantity backward through the program, from the output to the inputs. By the chain rule, the derivative with respect to a quantity is assembled from the derivatives with respect to the quantities that consume it, so one backward pass over the acyclic program suffices; each gate is handled by the transpose of its linearization, at the cost of the gate itself. For a product of pairs, the derivative with respect to a coefficient of one factor is a correlation with the other,
\begin{equation}\label{eq:correlation}
\frac{\partial P_o[k]}{\partial F_0[i]}
\ =\ \sum_{j=0}^{3k}
\frac{\partial P_o[k]}{\partial\,[\zeta^{i+j}](F_0G_0)}\ G_0[j]
\qquad(0\le i\le3k),
\end{equation}
one multiplication of polynomials of degree at most $6k$ after reversing the coefficient order, plus the $O(k)$ coordinatewise contributions of \Cref{eq:tailprod} through its low coefficients, its evaluations, and its tails. The transposed recurrences \Cref{eq:psirec} run in the opposite direction in $O(k)$, and the coordinatewise products \Cref{eq:localpass} are differentiated in $O(k)$. Every gate's transpose therefore costs the order of the gate, and the sweep visits each gate once.
\end{proof}

\emph{The operation count of Theorem~2.} Computing $\kalg$ costs $O(n\log(n+1))$ operations by Lemma~2.3; the nets $S_j$ at the $J+1=O(\log(k+1))$ dyadic radii cost $O(n)$ each by \Cref{slem:greedy}, and the charges $\widetilde\omega$ one further pass. In the interior branch, the forward computation performs, over the whole tree, fewer than $n$ products of pairs, each $O(k\log(k+1))$ by \Cref{slem:pairs} and the multiplication model of Section~2.1, and $O(k)$ local work at each vertex (\Cref{slem:recursion}); the reverse sweep matches this cost (\Cref{slem:reverse}); and assembling \Cref{eq:margid} at every vertex costs $O(nk)$. Pruning and compression (\Cref{slem:compress}) only remove work. The elementary branches cost $O(n)$ beyond the preprocessing. Every tie break (child orders, the greedy of \Cref{slem:greedy}, the merge order) is fixed, so the whole map $(T,V,\sigma,Y)\mapsto\widehat\mu$ is deterministic, and the total is
\[
O\bigl(n\log(n+1)+nk\log(k+1)\bigr)\ \le\ O\bigl(n^2\log(n+1)\bigr)
\]
operations, since $k\le K+1\le n$: the count claimed in Theorem~2, whose proof is now complete.

The remaining two lemmas support the instance sentence of Section~4.2. Sibling messages that are dilations of a common polynomial multiply in batch: the mechanism is that logarithms turn the product into power sums of the dilation parameters. Two standard consequences of fast multiplication are used \citep{vonzurgathen2013}: Newton iteration computes truncated inverses and logarithms of power series with constant term $1$, and exponentials of those with constant term $0$, at the cost of $O(1)$ multiplications of the same degree; and a degree-$D$ polynomial is evaluated at $d$ points in $O((D+d)\log^2(D+d))$ operations by a product tree.

\begin{lemma}[Batched dilations]\label{slem:batch}
Let $F$ be a polynomial of degree at most $3k$ with $F[0]=1$, and let $\xi_1,\dots,\xi_d>0$. In $O((k+d)\log^2(k+d))$ operations one can compute
\begin{enumerate}
\item\label{sit:batchlow} the truncation of
$\prod_{i=1}^dF(\xi_i\zeta)$ modulo $\zeta^{3k+1}$;
\item\label{sit:batcheval} the values $F(e^{-2}\xi_i)$ and
$F'(e^{-2}\xi_i)$ for all $i\le d$;
\item\label{sit:batchder} for any given $y_0,\dots,y_{3k}$, the
values
$\ \xi_i\,\dfrac{\partial}{\partial\xi_i}\displaystyle
\sum_{s=0}^{3k}y_s\,[\zeta^s]\prod_{j=1}^dF(\xi_j\zeta)\ $
for all $i\le d$.
\end{enumerate}
\end{lemma}

\begin{proof}
Write $g(\zeta):=\log F(\zeta)=\sum_{m=1}^{3k}g_m\zeta^m$ modulo $\zeta^{3k+1}$, computable by Newton iteration since $F[0]=1$, and let $t_m:=\sum_{i\le d}\xi_i^m$ be the power sums. Then, modulo $\zeta^{3k+1}$,
\begin{equation}\label{eq:powersum}
\prod_{i=1}^dF(\xi_i\zeta)\ =\
\exp\Bigl(\sum_{m=1}^{3k}g_m\,t_m\,\zeta^m\Bigr),
\end{equation}
because $\log F(\xi_i\zeta)=\sum_mg_m\xi_i^m\zeta^m$ and the logarithms add. The power sums come from the auxiliary polynomial $N(\zeta):=\prod_{i\le d}(1-\xi_i\zeta)$, built by a product tree, through the logarithmic-derivative identity $-\zeta N'(\zeta)/N(\zeta)=\sum_{m\ge1}t_m\zeta^m$, one derivative, one truncated inversion, and one multiplication; the exponential in \Cref{eq:powersum} is one more Newton iteration. This proves (\ref{sit:batchlow}), and (\ref{sit:batcheval}) is multipoint evaluation of $F$ and $F'$ at the points $e^{-2}\xi_i$. For (\ref{sit:batchder}), write $P:=\exp(H)$ with $H:=\sum_mg_mt_m\zeta^m$; then $dP=P\,dH$, so the displayed quantity, as a function of the $t_m$, has
\[
\frac{\partial}{\partial t_m}\sum_{s\le3k}y_s\,[\zeta^s]P
\ =\ g_m\sum_{s=m}^{3k}y_s\,[\zeta^{s-m}]P ,
\]
one correlation as in \Cref{eq:correlation}; and since $\xi_i\,\partial t_m/\partial\xi_i=m\,\xi_i^m$, the value carried by the dilation parameter $\xi_i$ is the evaluation at $\xi_i$ of the polynomial $\sum_{m=1}^{3k}m\,g_m\bigl(\sum_{s\ge m}y_s[\zeta^{s-m}]P\bigr) \,\zeta^m$, a second multipoint evaluation. Every step is a product tree, a truncated inversion, a logarithm or exponential, a correlation, or a multipoint evaluation at degree $O(k+d)$, each within the stated bound.
\end{proof}

On the star, either the active support collapses to the root or every leaf message is a dilation of one template; in both cases the evaluation is softly linear.

\begin{lemma}[Stars]\label{slem:star}
On the star $S_m$ with $m\ge2$ leaves, the estimator (4.3) is computed in $O(n\log^2n)$ operations, $n=m+1$.
\end{lemma}

\begin{proof}
The preprocessing and the elementary branches cost $O(n\log(n+1))$, so let the interior branch be selected, with budget $2\le k\le K$ and working scale $\bar\alpha=k\overline\delta_k\ge2$ (\Cref{slem:netsizes}(\ref{sit:floor})). On the star, the minimum ancestor $r$-net with $r\ge1$ is $\{o\}$, a leaf being an ancestor only of itself, while at $r=0$ every vertex needs itself; so with $j_0:=\min\{j:\lfloor\bar\alpha/m_j\rfloor=0\}$, understood as $j_0=J+1$ if no radius vanishes, the nets are $S_j=\{o\}$ for $j<j_0$ and $S_j=\sfV$ for $j\ge j_0$, and $j_0\ge1$ because $\lfloor\bar\alpha\rfloor\ge2$.

If $j_0=J+1$, then $A_k=\{o\}$: the only state vector with nonzero weight has $x_o=k$ and all leaf states zero, the posterior is a point mass, and $\widehat\mu_k=\tfrac Vk\,x=Vp_o$ is written in $O(n)$.

If $j_0\le J$, then $A_k=\sfV$ and every leaf carries the same charge $\widetilde\omega=m_{j_0}$. Factor the local weight of leaf $v$ at state $x$ as $\varphi_v(x)=e^{-Y_v^2/(4\sigma^2)}\,e^{-b^2x^2/(4\sigma^2)}\, \xi_v^x$ with $\xi_v:=e^{bY_v/(2\sigma^2)}$; the first factor is common to every assignment and is dropped. By \Cref{slem:recursion}, the message of leaf $v$ is then $F(\xi_v\zeta)$ for the common template
\[
F(\zeta)\ :=\ 1\ +\ e^{-2m_{j_0}}\sum_{x=1}^{3k}
e^{-2x}\,e^{-b^2x^2/(4\sigma^2)}\,\zeta^x ,
\]
and the child product of the root is $Q_o(\zeta)=\prod_{v}F(\xi_v\zeta)$. \Cref{slem:batch}(\ref{sit:batchlow}) gives its truncation, \Cref{slem:batch}(\ref{sit:batcheval}) gives $Q_o(e^{-2})=\prod_vF(e^{-2}\xi_v)$, and the tail follows as
\[
\tau_o\ =\ e^{6k}\Bigl(\,Q_o(e^{-2})
-\sum_{s=0}^{3k}Q_o[s]\,e^{-2s}\Bigr),
\]
exactly, from the definition of the tail. The root pass of \Cref{slem:recursion} then yields $P_o[k]$ in $O(k)$ operations.

For the marginals, $P_o[k]$ is a linear function of $(Q_o[0],\dots,Q_o[3k],\tau_o)$ by \Cref{eq:localpass} and \Cref{eq:psirec}, with coefficients computable in $O(k)$; and every assignment's weight carries the factor $\xi_v^{x_v}$, so Euler differentiation gives
\begin{equation}\label{eq:euler}
\E_Y[x_v]\ =\ \frac{\xi_v}{P_o[k]}\,
\frac{\partial P_o[k]}{\partial\xi_v}.
\end{equation}
The derivative splits along the linear form: the low coefficients contribute one instance of \Cref{slem:batch}(\ref{sit:batchder}), with weights that combine the direct coefficients of the linear form and the term $-e^{6k-2s}$ that $Q_o[s]$ inherits from $\tau_o$, while $Q_o(e^{-2})$ contributes
\[
\xi_v\,\frac{\partial Q_o(e^{-2})}{\partial\xi_v}
\ =\ Q_o(e^{-2})\cdot
\frac{e^{-2}\xi_v\,F'(e^{-2}\xi_v)}{F(e^{-2}\xi_v)}\ ,
\]
an $O(1)$ combination of the values from \Cref{slem:batch}(\ref{sit:batcheval}), whose denominators are positive because $F$ has positive constant term and nonnegative coefficients. Then $\widehat\mu_k(v)=b\,\E_Y[x_v]$ at every leaf and $\widehat\mu_k(o)=V$. All steps beyond the preprocessing cost $O((k+m)\log^2(k+m))$, and $k\le K<n$ makes this $O(n\log^2n)$.
\end{proof}

Together, \Cref{slem:batch,slem:star} verify the instance sentence of Section~4.2: a group of sibling messages agreeing up to a dilation of the variable multiplies in batch, and on stars the entire evaluation is softly linear in $n$.

%%%%%%%%%%%%%%%%%%%%%%%%%%%%%%%%%%%%%%%%%%%%%%%%%%%%%%%%%%%%%%%%%%%%%%%%%%%%%%
\section{Adaptation}\label{ssec:adaptation}
%%%%%%%%%%%%%%%%%%%%%%%%%%%%%%%%%%%%%%%%%%%%%%%%%%%%%%%%%%%%%%%%%%%%%%%%%%%%%%

This section proves the statements deferred from Section~4.3: the selection oracle (4.8), in the weighted and random-penalty form stated there; the scale calculus for the minimax risk; the subspace approximation (4.7); the enumeration of admissible supports; the assembly of the sufficiency half of Theorem~3, including its running time; the difference statistics of Theorem~4 with their contamination bound, the concentration of the noise estimate, and the moment bookkeeping; and the noise-estimation ceiling stated at the end of Section~4.3. Throughout, $K:=\lfloor n/C_\star\rfloor$ and dyadic integers are powers of two; the selection constant of (4.8) is instantiated as $C_2:=128$. The standard Gaussian tail bound $\mathbb P\bigl(N(0,\tau^2)\ge t\bigr)\le e^{-t^2/(2\tau^2)}$, valid for $t\ge0$, is used without further comment.

The selection engine comes first. The lemma below is the oracle promised with (4.8): the noise may have unequal, even zero, coordinate variances, which is how conditioning on the root observation will enter, and the penalties may be random as long as they are bracketed on an event, which is how the estimated noise scale will enter.

\begin{lemma}[Weighted affine selection]\label{slem:selection}
Let $Y=\theta+Z\in\bbR^N$, where $\theta\in\bbR^N$ and $Z$ is a centered Gaussian vector with independent coordinates whose variances are at most $\sigma^2$, some possibly zero. Let $\mathcal M$ be a countable collection of affine models $m=c_m+W_m\subseteq\bbR^N$, each $W_m$ a linear subspace of finite dimension $D_m$, with weights $\Delta_m\ge0$ such that $\sum_{m\in\mathcal M}e^{-\Delta_m}\le1$. Let $\widehat\theta$ minimize $\|Y-\nu\|_2^2+\operatorname{pen}(m)$ over $m\in\mathcal M$ and $\nu\in m$, and over a full model whose fit is $\nu=Y$ itself. The penalties are nonnegative, possibly random, and such that the minimum defining $\widehat\theta$ is attained, as it is in particular for finite $\mathcal M$. Suppose that on an event $E$, for a constant $\bar\kappa\ge32$,
\begin{gather*}
32\,\sigma^2(D_m+\Delta_m)\ \le\ \operatorname{pen}(m)\ \le\
\bar\kappa\,\sigma^2(D_m+\Delta_m)
\qquad\text{for every }m\in\mathcal M,
\\
32\,\sigma^2N\ \le\ \operatorname{pen}(\mathrm{full})\ \le\
\bar\kappa\,\sigma^2N .
\end{gather*}
Then
\[
\E\bigl[\|\widehat\theta-\theta\|_2^2\,\ind_E\bigr]\ \le\
C_{\bar\kappa}\Bigl(\min\Bigl\{\ \inf_{m\in\mathcal M}
\bigl(\operatorname{dist}^2(\theta,m)+\sigma^2(D_m+\Delta_m+1)\bigr),
\ \ \sigma^2N\ \Bigr\}\ +\ \sigma^2\Bigr),
\]
where $C_{\bar\kappa}$ depends only on $\bar\kappa$.
\end{lemma}

\begin{proof}
The selected model is compared with one fixed model through the basic inequality of penalized least squares; a single Kraft-weighted deviation variable controls the Gaussian fluctuations over the whole collection, and the full model is handled by its own comparison.

\emph{A projection bound.} Let $U\subseteq\bbR^N$ be a fixed subspace of dimension $d$ and $P_U$ the orthogonal projection onto it. For an orthonormal basis $u_1,\dots,u_d$ of $U$, the vector $(\langle Z,u_i\rangle)_{i\le d}$ is centered Gaussian with covariance $\Sigma\preceq\sigma^2I_d$: its quadratic form at $a\in\bbR^d$ is $\operatorname{Var}\langle Z,\sum_ia_iu_i\rangle \le\sigma^2\|\sum_ia_iu_i\|_2^2=\sigma^2\|a\|_2^2$. For a centered Gaussian $g$ of variance $\tau^2\le\sigma^2$, direct integration gives $\E e^{g^2/(4\sigma^2)}=(1-\tau^2/(2\sigma^2))^{-1/2}\le\sqrt2$, the case $\tau^2=0$ included; diagonalizing $\Sigma$ and multiplying over the independent coordinates,
\[
\E\,e^{\|P_UZ\|_2^2/(4\sigma^2)}\ \le\ 2^{d/2};
\]
hence, by Markov's inequality and $2^{1/2}\le e$, for every $t\ge0$,
\[
\mathbb P\bigl(\|P_UZ\|_2^2\ge4\sigma^2(d+t)\bigr)\ \le\
2^{d/2}\,e^{-d-t}\ \le\ e^{-t}.
\]

\emph{The deviation variable.} Fix a comparison model $m_0\in\mathcal M$ and let $f_0\in m_0$ attain $a:=\operatorname{dist}(\theta,m_0)$, the metric projection of $\theta$ onto the closed affine set $m_0$. For $m\in\mathcal M$ put
\[
U_m\ :=\ W_m+W_{m_0}+\operatorname{span}(c_m-f_0),
\qquad
d_m\ :=\ \dim U_m\ \le\ D_m+D_{m_0}+1,
\]
so that every $\nu=c_m+w\in m$ has $\nu-f_0=(c_m-f_0)+w\in U_m$, and define
\[
Z^*\ :=\ \sup_{m\in\mathcal M}\
\Bigl[\frac{\|P_{U_m}Z\|_2^2}{4\sigma^2}-(d_m+\Delta_m)\Bigr]_+ .
\]
By the projection bound and a union bound, $\mathbb P(Z^*>t)\le\sum_{m}e^{-\Delta_m-t}\le e^{-t}$, so $\E Z^*=\int_0^\infty\mathbb P(Z^*>t)\,dt\le1$; and by its definition, $\|P_{U_m}Z\|_2^2\le4\sigma^2(d_m+\Delta_m+Z^*)$ for every $m$ simultaneously.

\emph{A model is selected.} Suppose the minimum is attained at some $\widehat m\in\mathcal M$ with fit $\widehat\theta\in\widehat m$, and write $d:=\widehat\theta-\theta$. Expanding the basic inequality against $(m_0,f_0)$, $\|Y-\widehat\theta\|_2^2+\operatorname{pen}(\widehat m)\le \|Y-f_0\|_2^2+\operatorname{pen}(m_0)$, at $Y=\theta+Z$ and combining the two cross terms into $2\langle Z,\widehat\theta-f_0\rangle$ gives
\[
\|d\|_2^2\ \le\ a^2+2\langle Z,\widehat\theta-f_0\rangle
-\operatorname{pen}(\widehat m)+\operatorname{pen}(m_0).
\]
Since $\widehat\theta-f_0\in U_{\widehat m}$, the inequality $2xy\le\tfrac14x^2+4y^2$ and the deviation variable give
\begin{align*}
2\langle Z,\widehat\theta-f_0\rangle
\ =\ 2\bigl\langle P_{U_{\widehat m}}Z,\
\widehat\theta-f_0\bigr\rangle
\ &\le\ \tfrac14\|\widehat\theta-f_0\|_2^2
+4\,\|P_{U_{\widehat m}}Z\|_2^2
\\
&\le\ \tfrac14\bigl(\|d\|_2+a\bigr)^2
+16\sigma^2\bigl(d_{\widehat m}+\Delta_{\widehat m}+Z^*\bigr).
\end{align*}
By $d_{\widehat m}\le D_{\widehat m}+D_{m_0}+1$, the last term splits into $16\sigma^2(D_{\widehat m}+\Delta_{\widehat m}) +16\sigma^2(D_{m_0}+1)+16\sigma^2Z^*$, and on $E$ the penalty dominates its first piece: $\operatorname{pen}(\widehat m)\ge32\sigma^2 (D_{\widehat m}+\Delta_{\widehat m})$. Together with $\tfrac14(\|d\|_2+a)^2\le\tfrac12\|d\|_2^2+\tfrac12a^2$, the expansion rearranges to
\begin{equation}\label{eq:modelcase}
\|d\|_2^2\ \le\ 3a^2+2\operatorname{pen}(m_0)
+32\sigma^2(D_{m_0}+1)+32\sigma^2Z^*
\qquad\text{on }E\cap\{\text{a model is selected}\}.
\end{equation}

\emph{The full model is selected.} Then $\widehat\theta=Y$ and $\|d\|_2^2=\|Z\|_2^2$, while the full model's own residual vanishes, so the basic inequality against $(m_0,f_0)$ reads $\operatorname{pen}(\mathrm{full})\le\|Y-f_0\|_2^2 +\operatorname{pen}(m_0)\le2a^2+2\|Z\|_2^2+\operatorname{pen}(m_0)$; on $E$ this gives $32\sigma^2N\le2a^2+2\|Z\|_2^2+\operatorname{pen}(m_0)$. On the event $\{\|Z\|_2^2\le2\sigma^2N\}$, where $2\|Z\|_2^2\le4\sigma^2N$, the preceding two inequalities force $28\sigma^2N\le2a^2+\operatorname{pen}(m_0)$, hence $\|Z\|_2^2\le2\sigma^2N\le\tfrac1{14}\bigl(2a^2 +\operatorname{pen}(m_0)\bigr)$. On the complement, the projection bound at $U=\bbR^N$ gives $\mathbb P(\|Z\|_2^2>2\sigma^2N)\le2^{N/2}e^{-N/2}\le e^{-N/8}$; moreover $\E\|Z\|_2^4=\sum_i\E Z_i^4+\sum_{i\ne j}\E Z_i^2\,\E Z_j^2 \le3\sigma^4N+\sigma^4N(N-1)\le3\sigma^4N^2$, the coordinates being independent with fourth moments at most $3\sigma^4$, so the Cauchy--Schwarz inequality gives
\[
\E\bigl[\|Z\|_2^2\,\ind\{\|Z\|_2^2>2\sigma^2N\}\bigr]
\ \le\ \sqrt3\,\sigma^2N\,e^{-N/16}\ \le\ C\sigma^2 ,
\]
the map $N\mapsto Ne^{-N/16}$ being bounded.

\emph{Assembly.} Combining \Cref{eq:modelcase} with the two estimates of the full-model case and using $\E Z^*\le1$ and the bracketing of $\operatorname{pen}(m_0)$ on $E$,
\[
\E\bigl[\|\widehat\theta-\theta\|_2^2\,\ind_E\bigr]
\ \le\ C\bigl(a^2+\bar\kappa\,\sigma^2(D_{m_0}+\Delta_{m_0}+1)\bigr)
+C\sigma^2 ,
\]
and taking the infimum over $m_0$ gives the first branch of the minimum. For the second branch, compare with the full model instead: if the full model is selected, then $\E[\|d\|_2^2\,\ind_E]\le\E\|Z\|_2^2\le\sigma^2N$; if some $\widehat m$ is selected, the basic inequality against the full model gives $\|Y-\widehat\theta\|_2^2\le\operatorname{pen}(\mathrm{full}) \le\bar\kappa\sigma^2N$ on $E$, so $\|d\|_2^2\le2\|Y-\widehat\theta\|_2^2+2\|Z\|_2^2$ has expectation at most $(2\bar\kappa+2)\sigma^2N$ there. Taking the minimum of the two comparisons completes the proof.
\end{proof}

Three features are used below. The penalties may be random, only their bracketing on $E$ entering the argument. The coordinate variances may differ and may vanish. And a point $\nu$ is the affine model $\{\nu\}$ with $D_m=0$, whose oracle value reads $\|\theta-\nu\|_2^2+\sigma^2(\Delta_{\{\nu\}}+1)$; we write $w(\nu)$ for the weight of a point model. The display (4.8) is the deterministic special case: penalties $C_2\sigma^2(\dim m+\Delta_m)$ and $C_2\sigma^2N$ satisfy the bracketing with $E$ the whole probability space and $\bar\kappa=C_2$.

The coverage argument will replace the unknown amplitude by a grid value and the unknown noise level by a rounded estimate. The next lemma prices both substitutions.

\begin{lemma}[Scale calculus]\label{slem:scale}
For every finite rooted tree, $R^*_T(V,\sigma)$ is nondecreasing in $V$ and in $\sigma$, and for every $c\ge1$,
\[
R^*_T(cV,\sigma)\ \le\ c^2\,R^*_T(V,\sigma),
\qquad
R^*_T(V,c\sigma)\ \le\ c^2\,R^*_T(V,\sigma).
\]
\end{lemma}

\begin{proof}
\emph{Monotonicity in $V$.} Let $V\le V'$. The translation $\iota(\mu):=\mu+(V'-V)p_o$ adds $V'-V$ to the leak at the root and changes no other leak, so it maps $\mathcal F_V(T)$ into $\mathcal F_{V'}(T)$ by Lemma~2.1. Given any estimator $\widehat\mu'$, define $\widehat\mu(Y):=\widehat\mu'\bigl(Y+(V'-V)p_o\bigr)-(V'-V)p_o$. Under $\mu\in\mathcal F_V(T)$, the shifted data $Y+(V'-V)p_o$ are distributed as an observation of $\iota(\mu)$, and $\|\widehat\mu-\mu\|_2 =\|\widehat\mu'(Y+(V'-V)p_o)-\iota(\mu)\|_2$, so the risk of $\widehat\mu$ at $\mu$ equals the risk of $\widehat\mu'$ at $\iota(\mu)\in\mathcal F_{V'}(T)$. Taking the supremum over $\mu$ and then the infimum over $\widehat\mu'$ gives $R^*_T(V,\sigma)\le R^*_T(V',\sigma)$.

\emph{Monotonicity in $\sigma$.} Let $\sigma\le\sigma'$ and let $W\sim N(0,(\sigma'^2-\sigma^2)I_n)$ be independent of the data. For any estimator $\widehat\mu'$, the randomized procedure $\widehat\mu'(Y+W)$ has, at every $\mu$, the risk of $\widehat\mu'$ under noise level $\sigma'$; replacing it by its average $\widehat\mu(Y):=\E_W\widehat\mu'(Y+W)$ only lowers the squared-error risk, by Jensen's inequality conditionally on $Y$, and produces a deterministic estimator. Hence $R^*_T(V,\sigma)\le R^*_T(V,\sigma')$.

\emph{The quadratic bounds.} Scaling the observation, the parameter, and the estimator by $a>0$ turns the experiment at $(V,\sigma)$ into the experiment at $(aV,a\sigma)$, since $\mathcal F_{aV}(T)=a\,\mathcal F_V(T)$; the squared error scales by $a^2$, so
\[
R^*_T(aV,a\sigma)\ =\ a^2\,R^*_T(V,\sigma)
\qquad(a>0).
\]
For $c\ge1$, combining this exact identity with the monotonicity just proved gives
\[
\begin{aligned}
R^*_T(cV,\sigma)\ &=\ c^2R^*_T(V,\sigma/c)\ \le\ c^2R^*_T(V,\sigma),\\
R^*_T(V,c\sigma)\ &=\ c^2R^*_T(V/c,\sigma)\ \le\ c^2R^*_T(V,\sigma),
\end{aligned}
\]
using $\sigma/c\le\sigma$ and $V/c\le V$.
\end{proof}

Two consequences are used repeatedly. First, for every amplitude $v>0$ and every working scale $s\in[\sigma/2,4\sigma]$, monotonicity and the quadratic bound at $c=4$ give
\begin{equation}\label{eq:workingscale}
R^*_T(v,s)\ \le\ R^*_T(v,4\sigma)\ \le\ 16\,
R^*_T(v,\sigma).
\end{equation}
Second, $\kalg(T,v,s)\le k_0(T,v,s)$ by Lemma~2.3(3), and the middle term below is at most $CR^*_T(v,s)$ by Theorem~1 at $(v,s)$, so \Cref{eq:workingscale} gives
\begin{equation}\label{eq:ratescale}
\min\bigl\{v^2H_T,\ s^2\kalg(T,v,s)\bigr\}\ \le\
\min\bigl\{v^2H_T,\ s^2k_0(T,v,s)\bigr\}\ \le\ C\,R^*_T(v,\sigma),
\end{equation}
valid for all $v>0$ and $s\in[\sigma/2,4\sigma]$.

The amplitude-free half of the candidate system rests on the subspace approximation (4.7), verified next with an explicit constant.

\begin{lemma}[Subspace approximation]\label{slem:subspace}
For every $\mu\in\mathcal F_V(T)$ and every integer $2\le k\le n/C_\star$ there is an admissible support $S$ with
\[
\operatorname{dist}^2\bigl(\mu,\ \mathcal V_{k,S}\bigr)\ \le\
22\,V^2\overline\delta_k(T).
\]
\end{lemma}

\begin{proof}
Run the proof of Proposition~3.1 on the normalized signal $f:=\mu/V$, retaining its objects: the terminal cells $L$ with masses $\lambda_L$ and exits $q_L$; the collapse $g$ with $\|f-g\|_2^2\le3\bar\alpha/k$ by (3.4); the skeleton $g_0=\sum_{h\in U}b_h\,p_h$, supported on the set $U$ consisting of $R_0$ and the heavy roots, with $g=g_0+\sum_L\lambda_Lq_L$; and the exit sampling, fixed at a realization with $W\le2k$ and $\|\sum_L(Z_L-\lambda_L)q_L\|_2^2\le8\bar\alpha/k$. Keep the skeleton unrounded and set
\[
\nu\ :=\ g_0+\sum_LZ_Lq_L ,
\]
a real linear combination of $\{p_v:\ v\in U\cup\{a_L:Z_L\ne0\}\}$. The set $S:=\bigl(U\cup\{a_L:Z_L\ne0\}\bigr)\setminus R_0$ is admissible: the heavy roots carry total activation charge below $2k$ by \Cref{slem:heavy}(3); the selected exit roots are pairwise distinct with $\widetilde\omega(a_L)\le m(L)$, so their charges total at most $W\le2k$; and a union is charged at most the sum of its parts. Hence $\sum_{v\in S}\widetilde\omega(v)\le4k$. Since $\nu\in\mathcal V_{k,S}$ and $f-\nu=(f-g)+\sum_L(\lambda_L-Z_L)q_L$,
\[
\operatorname{dist}^2\bigl(f,\ \mathcal V_{k,S}\bigr)\ \le\
\|f-\nu\|_2^2\ \le\
2\,\frac{3\bar\alpha}k+2\,\frac{8\bar\alpha}k
\ =\ 22\,\frac{\bar\alpha}k\ =\ 22\,\overline\delta_k(T).
\]
Restoring the amplitude multiplies the squared distance by $V^2$.
\end{proof}

Selecting over the capped subspaces requires listing them; this is the enumeration claim of Section~4.3.

\begin{lemma}[Support enumeration]\label{slem:enum}
For every integer $2\le k\le n/C_\star$, the admissible supports at budget $k$ number at most $e^{C_1k}$, and they can be listed in $O(|A_k|\,e^{C_1k})$ operations.
\end{lemma}

\begin{proof}
An admissible support is a subset of $A_k\setminus R_0$ of total activation charge at most $4k\le9k$, hence one of the charged parts counted in the proof of \Cref{slem:count}, where the generating-function estimate bounds their number by $e^{19k}\le e^{C_1k}$. To enumerate, fix a linear order on $A_k\setminus R_0$ and run a depth-first search over subsets that appends only higher-indexed vertices and recurses only while the accumulated charge remains at most $4k$. Every node visited by the search is an admissible support, every admissible support is visited exactly once, along the insertion order of its elements, and testing all extensions of a visited support costs $O(|A_k|)$ operations.
\end{proof}

\emph{The candidate system.} The estimator of Theorem~3 is the penalized selection of \Cref{slem:selection} over the candidates described in Section~4.3, which we now fix precisely. The system is parameterized by a working scale $s$, equal to $\sigma$ in Theorem~3 and to the rounded noise estimate in Theorem~4. Fix a universal threshold $n_\star$ such that
\[
\Bigl\lfloor\frac{\log n}{2C_1}\Bigr\rfloor\ \ge\ 2
\qquad\text{and}\qquad
\Bigl\lfloor\frac{\log n}{2C_1}\Bigr\rfloor\ \le\
\Bigl\lfloor\frac n{C_\star}\Bigr\rfloor
\qquad\text{for every }n\ge n_\star .
\]
For $n<n_\star$ the estimator returns $Y$, whose risk $n\sigma^2\le n_\star\sigma^2$ is absorbed by the additive $\sigma^2$ after enlarging the universal constant; assume $n\ge n_\star$ from here on, and set $k^\star:=\lfloor\log n/(2C_1)\rfloor$, so that $2\le k^\star\le K$. Write
\[
v_i:=(Y_o)_++2^is\ \ (0\le i\le J^+),
\qquad
v^-_j:=2^{-j}s\ \ (1\le j\le J^-),
\]
with $J^+=J^-:=\lceil\log_2(n+1)\rceil+2$, for the root-anchored and the low-amplitude grids. Throughout, $\mathcal X_k$ is the class (4.1) of integer states of Section~4.1. The candidates are:
\begin{itemize}
\item the subspaces $\mathcal V_{k,S}$ for dyadic $2\le k\le k^\star$
and admissible $S$, with weights $\Delta_{k,S}:=(C_1+1)k$, together
with $\operatorname{span}(p_o)$ at weight $1$;
\item the point models $\nu_{i,k,x}:=(v_i/k)\,x$ for $x\in\mathcal
X_k$ and dyadic $2\le k\le K$, at weights
$w(\nu_{i,k,x}):=2\Gamma(x)+k+i+\log_2k+4$, and the root points
$\nu^0_i:=v_i\,p_o$ at weights $i+4$; their low-amplitude
counterparts $\nu^-_{j,k,x}:=(v^-_j/k)\,x$ at weights
$2\Gamma(x)+k+2\log(1+j)+\log_2k+4$, and $\nu^{0-}_j:=v^-_j\,p_o$ at
weights $2\log(1+j)+4$;
\item the full model.
\end{itemize}
The penalties are $\operatorname{pen}(m):=C_2s^2(D_m+\Delta_m)$ for the affine and point models, with $\Delta_m$ the weight just listed, and $\operatorname{pen}(\mathrm{full}):=C_2s^2n$. The supports $A_k$, the charges $\widetilde\omega$, and hence $\Gamma$ and $\mathcal X_k$ depend only on $(T,k)$ (Proposition~3.1); a grid index enters only through its amplitude and its weight, and all tie breaks are fixed, so the system and its penalties are deterministic given $(Y_o,s)$.

The weights are Kraft-summable with room to spare. For the subspaces, \Cref{slem:enum} gives
\[
\sum_{\substack{k\ \mathrm{dyadic}\\2\le k\le k^\star}}\
\sum_{S\ \mathrm{admissible}}e^{-(C_1+1)k}\ +\ e^{-1}
\ \le\ \sum_{\substack{k\ \mathrm{dyadic}\\ k\ge2}}e^{-k}+e^{-1}
\ <\ 0.16+0.37\ =\ 0.53 .
\]
For the grid family, \Cref{slem:normalizer} gives $\sum_{x\in\mathcal X_k}e^{-2\Gamma(x)-k}\le1$ for every dyadic $k=2^\ell$, so, once the price of the grid index is factored out, each amplitude contributes at most $1$ from its root candidate and $\sum_{\ell\ge1}e^{-\ell}$ from its flow candidates, and
\begin{multline}\label{eq:kraftsum}
\sum_{\text{grid models}}e^{-w}
\ \le\ \Bigl(\sum_{i\ge0}e^{-i-4}
+\sum_{j\ge1}\frac{e^{-4}}{(1+j)^2}\Bigr)
\Bigl(1+\sum_{\ell\ge1}e^{-\ell}\Bigr)
\\
=\ e^{-4}\Bigl(\frac e{e-1}+\frac{\pi^2}6-1\Bigr)\frac e{e-1}
\ <\ \frac18 .
\end{multline}
The total is below $1$, as \Cref{slem:selection} requires.

The root-anchored grid is finite, and the truth can exceed its top value only after a downward root-noise fluctuation of order $n$. The next lemma quantifies both halves of this statement; its event $\mathcal E_+$ is the one announced in Section~4.3.

\begin{lemma}[Anchor event]\label{slem:anchor}
Let $s\in[\sigma/2,4\sigma]$, extend the anchored grid to $v_i=(Y_o)_++2^is$ for all integers $i\ge0$, and set
\[
i^*\ :=\ \min\{i\ge0:\ v_i\ge V\},
\qquad
\mathcal E_+\ :=\ \{i^*\le J^+\},
\qquad
Z_o\ :=\ Z(o).
\]
Then:
\begin{enumerate}
\item\label{sit:pointwise} pointwise,
$2^{i^*}s\le2s+2\sigma|Z_o|$, and consequently
\[
i^*\le1+\log_2\bigl(1+2|Z_o|\bigr),
\qquad
v_{i^*}\le V+2s+3\sigma|Z_o|,
\qquad
0\ \le\ v_{i^*}-Y_o\ \le\ 2s+3\sigma|Z_o| ,
\]
and in expectation
$\E\bigl[\sigma^2(1+i^*)+(Y_o-v_{i^*})^2\bigr]\le C\sigma^2$;
\item\label{sit:omitted}
$\mathcal E_+^{\,c}\subseteq\{Z_o<-2(n+1)\}$;
\item\label{sit:offanchor} every penalized selection of this section,
the full model carrying penalty $C_2s^2n$ and all penalties being
nonnegative, satisfies pathwise
$\|\widehat\mu-\mu\|_2^2\le2C_2s^2n+2\sigma^2\|Z\|_2^2$, and therefore
$\E\bigl[\|\widehat\mu-\mu\|_2^2\,\ind_{\mathcal E_+^c}\bigr]
\le C\sigma^2$.
\end{enumerate}
\end{lemma}

\begin{proof}
\emph{Part (1).} Since $Y_o=V+\sigma Z_o$ and $(Y_o)_+\ge Y_o$,
\[
V-(Y_o)_+\ \le\ (V-Y_o)_+\ =\ (-\sigma Z_o)_+\ \le\ \sigma|Z_o| .
\]
If $\sigma|Z_o|\le s$, then $v_0=(Y_o)_++s\ge V-\sigma|Z_o|+s\ge V$, so $i^*=0$ and $2^{i^*}s=s$. If $\sigma|Z_o|>s$ and $i^*\ge1$, minimality gives $v_{i^*-1}<V$, that is, $2^{i^*-1}s<V-(Y_o)_+\le\sigma|Z_o|$, so $2^{i^*}s<2\sigma|Z_o|$; and if $\sigma|Z_o|>s$ and $i^*=0$, then $2^{i^*}s=s<2\sigma|Z_o|$. In every case $2^{i^*}s\le2s+2\sigma|Z_o|$, and $i^*$ is finite, the grid values increasing to infinity. Dividing by $s\ge\sigma/2$ gives $2^{i^*}\le2+4|Z_o|=2(1+2|Z_o|)$, which is the bound on $i^*$. Next, $(Y_o)_+\le V+\sigma|Z_o|$, since $Y_o\le V+\sigma|Z_o|$ and $0\le V$, so $v_{i^*}=(Y_o)_++2^{i^*}s\le V+\sigma|Z_o|+2s+2\sigma|Z_o|$. For the last pointwise claim, $v_{i^*}\ge Y_o$ because $(Y_o)_+\ge Y_o$, and
\[
v_{i^*}-Y_o\ =\ \bigl((Y_o)_+-Y_o\bigr)+2^{i^*}s
\ \le\ (-Y_o)_++2s+2\sigma|Z_o|
\ \le\ 2s+3\sigma|Z_o| ,
\]
since $-Y_o=-V-\sigma Z_o\le\sigma|Z_o|$. In expectation, $\E\,i^*\le1+\log_2(1+2\,\E|Z_o|)\le3$ by concavity of $t\mapsto\log_2(1+2t)$ and $\E|Z_o|\le1$, while $\E(Y_o-v_{i^*})^2\le2(2s)^2+18\sigma^2\E Z_o^2\le146\sigma^2$, using $s\le4\sigma$.

\emph{Part (2).} On $\mathcal E_+^{\,c}$, $i^*>J^+$, so $v_{J^+}<V$ and
\[
(-\sigma Z_o)_+\ \ge\ V-(Y_o)_+\ >\ 2^{J^+}s\ \ge\
4(n+1)\cdot\frac\sigma2\ =\ 2(n+1)\,\sigma ,
\]
using $2^{J^+}\ge2^{\log_2(n+1)+2}=4(n+1)$ and $s\ge\sigma/2$; the positive part is positive here, so it equals $-\sigma Z_o$, and $Z_o<-2(n+1)$.

\emph{Part (3).} The selected criterion value is at most the full model's, whose fit is $Y$ itself, so $\|Y-\widehat\mu\|_2^2\le \|Y-\widehat\mu\|_2^2+\operatorname{pen}(\widehat m)\le \operatorname{pen}(\mathrm{full})=C_2s^2n$, and $\|\widehat\mu-\mu\|_2^2\le2\|\widehat\mu-Y\|_2^2+2\|Y-\mu\|_2^2 \le2C_2s^2n+2\sigma^2\|Z\|_2^2$. Put $a:=2(n+1)$. By part (\ref{sit:omitted}) and the independence of $Z_o$ from the nonroot coordinates,
\begin{align*}
\E\bigl[\|Z\|_2^2\,\ind_{\mathcal E_+^c}\bigr]
\ &\le\ (n-1)\,\mathbb P(Z_o<-a)
+\E\bigl[Z_o^2\,\ind\{Z_o<-a\}\bigr]
\\
&\le\ (n-1)\,e^{-a^2/2}+a\,\phi_{\mathcal N}(a)
+\Phi_{\mathcal N}(-a)\ \le\ C ,
\end{align*}
where $\phi_{\mathcal N}$ is the standard normal density, the middle step by the Gaussian tail bound and one integration by parts. With $s^2\le16\sigma^2$ and $n\,e^{-a^2/2}\le C$, the pathwise bound integrates to $\E[\|\widehat\mu-\mu\|_2^2\ind_{\mathcal E_+^c}] \le C\sigma^2$.
\end{proof}

With the anchor in place, some grid candidate always sits close to the truth at a price the selection can afford. The lemma is stated conditionally on the root observation, matching the conditional analysis of the selection to come.

\begin{lemma}[Grid coverage]\label{slem:coverage}
Condition on $Y_o=y_0$, let $s\in[\sigma/2,4\sigma]$, suppose $\mathcal E_+$ holds, and write $v:=v_{i^*}$, $k':=\kalg(T,v,s)$, and $\widetilde\theta:=(y_0,\mu_{-o})$. Let $\mathfrak O$ be the minimum, over the point models of the candidate system and the full model, of the oracle values of \Cref{slem:selection} at mean $\widetilde\theta$:
\[
\mathfrak O\ :=\ \min\Bigl\{\ \inf_{\nu}\bigl(
\|\widetilde\theta-\nu\|_2^2+\sigma^2(w(\nu)+1)\bigr),\ \
\sigma^2n\ \Bigr\}.
\]
Then, for every realization,
\begin{equation}\label{eq:anchoredcover}
\mathfrak O\ \le\ C\Bigl(R^*_T(v,\sigma)+(y_0-v)^2
+\sigma^2(1+i^*)\Bigr).
\end{equation}
If moreover $V<s/2$, set $j^*:=\lfloor\log_2(s/V)\rfloor\ge1$, so that $v^-_{j^*}\in[V,2V]$; if $j^*\le J^-$, then also
\begin{equation}\label{eq:lowcover}
\mathfrak O\ \le\ C\Bigl(R^*_T\bigl(v^-_{j^*},\sigma\bigr)
+\bigl(y_0-v^-_{j^*}\bigr)^2
+\sigma^2\bigl(1+\log(1+j^*)\bigr)\Bigr).
\end{equation}
\end{lemma}

\begin{proof}
Every candidate of anchored index $i^*$ has root coordinate exactly $v$: the flow candidates because their root state is $k$, so $\nu_{i^*,k,x}(o)=(v/k)\,k=v$, and the root candidates trivially. Since $v\ge V$, the translate $\mu':=\mu+(v-V)p_o$ lies in $\mathcal F_v(T)$ by Lemma~2.1 and agrees with $\mu$ off the root, with $\mu'(o)=v$; hence, for every such candidate $\nu$,
\begin{equation}\label{eq:rootsplit}
\|\widetilde\theta-\nu\|_2^2\ =\ (y_0-v)^2+\|\mu'-\nu\|_2^2 .
\end{equation}
Let $k_{\max}$ be the largest dyadic integer at most $K$, defined when $K\ge2$, so that $2k_{\max}>K$. There are three exhaustive cases.

\emph{Case (a): $s^2k'\le v^2H_T$, $K\ge2$, and $k'\le k_{\max}$.} Choose the smallest dyadic $k\ge\max\{k',2\}$; then $k\le2k'$ and $k\le k_{\max}\le K$, since $k_{\max}$ is itself a dyadic integer at least $\max\{k',2\}$. Proposition~3.1(1), applied to $\mu'\in\mathcal F_v(T)$ at budget $k$ and taken with the explicit constant of its error assembly in \Cref{ssec:cover}, supplies $x^\star\in\mathcal X_k$ with $\Gamma(x^\star)\le9k$ and
\[
\|\mu'-\nu_{i^*,k,x^\star}\|_2^2\ \le\ 51\,v^2\overline\delta_k(T)
\ \le\ 51\,v^2\overline\delta_{k'}(T)\ \le\ 102\,A_0\,s^2k' ,
\]
the second step by the monotonicity of $\overline\delta$ (Lemma~2.3(3)) and $k\ge k'$, the third because the dimension clause of (2.2) is silent at $k'\le K$, so its surrogate crossing clause fired there: $v^2\overline\delta_{k'}\le2A_0s^2k'$. The weight obeys
\[
w(\nu_{i^*,k,x^\star})\ =\ 2\Gamma(x^\star)+k+i^*+\log_2k+4
\ \le\ 19k+i^*+\log_2k+4\ \le\ C(k'+i^*+1).
\]
By the case hypothesis and \Cref{eq:ratescale}, $s^2k'=\min\{v^2H_T,s^2k'\}\le CR^*_T(v,\sigma)$, and $\sigma^2k'\le4s^2k'$; with \Cref{eq:rootsplit}, the candidate's oracle value is therefore at most $(y_0-v)^2+CR^*_T(v,\sigma)+C\sigma^2(1+i^*)$, which is \Cref{eq:anchoredcover}.

\emph{Case (b): $v^2H_T<s^2k'$.} Take the root candidate $\nu^0_{i^*}$, whose nonroot coordinates vanish. By \Cref{eq:rootsplit} and the convexity bound $\|\mu-Vp_o\|_2\le V\sqrt{h_T}$ established in the proof of \Cref{slem:branches},
\[
\|\widetilde\theta-\nu^0_{i^*}\|_2^2\ =\
(y_0-v)^2+\|\mu-Vp_o\|_2^2\ \le\ (y_0-v)^2+V^2h_T ,
\]
and $V^2h_T\le v^2H_T=\min\{v^2H_T,s^2k'\}\le CR^*_T(v,\sigma)$ by \Cref{eq:ratescale}. The weight is $i^*+4$, so the oracle value is at most $(y_0-v)^2+CR^*_T(v,\sigma)+\sigma^2(i^*+5)$.

\emph{Case (c): $s^2k'\le v^2H_T$, and $K\le1$ or $k'>k_{\max}$.} If $K\ge2$ and $k'>k_{\max}$, then $K<2k_{\max}<2k'$ and $n<C_\star(K+1)\le2C_\star K<4C_\star k'$, so
\[
\sigma^2n\ \le\ 4s^2n\ <\ 16\,C_\star\,s^2k'\ \le\
C\,R^*_T(v,\sigma)
\]
by the case hypothesis and \Cref{eq:ratescale}: the full-model branch of $\mathfrak O$ qualifies. If $K\le1$, then $n<2C_\star$, and $\sigma^2n\le2C_\star\sigma^2$ is absorbed by the term $\sigma^2(1+i^*)$ of \Cref{eq:anchoredcover}.

\emph{The low-amplitude bound.} Let $V<s/2$ and $j^*\le J^-$. From $2^{j^*}\le s/V<2^{j^*+1}$, the amplitude $v^-_{j^*}=2^{-j^*}s$ lies in $[V,2V)$, the translate $\mu+(v^-_{j^*}-V)p_o$ lies in $\mathcal F_{v^-_{j^*}}(T)$ and agrees with $\mu$ off the root, and every candidate of low-amplitude index $j^*$ has root coordinate exactly $v^-_{j^*}$, so \Cref{eq:rootsplit} holds with $v^-_{j^*}$ in place of $v$. The three cases above run verbatim at the amplitude $v^-_{j^*}$ and its crossing index $\kalg(T,v^-_{j^*},s)$, the weight contributions $i^*+\log_2k+4$ and $i^*+4$ replaced by $2\log(1+j^*)+\log_2k+4$ and $2\log(1+j^*)+4$; they yield \Cref{eq:lowcover}.
\end{proof}

One computational ingredient remains before the assembly: the inner minimization of the selection criterion over an integer-state class. It is the min-plus analogue of the evaluation of \Cref{slem:recursion}, on the same one-dimensional state and with its own overflow scalar.

\begin{lemma}[Min-plus evaluation]\label{slem:minplus}
Fix an amplitude $v>0$, a dyadic budget $2\le k\le K$, and a constant $c\ge0$. The minimum
\[
m_{v,k}\ :=\ \min_{x\in\mathcal X_k}
\Bigl[\bigl\|Y-\tfrac vk\,x\bigr\|_2^2+c\,\Gamma(x)\Bigr]
\]
and one minimizer are computable exactly in $O(nk^2)$ operations.
\end{lemma}

\begin{proof}
Write $b:=v/k$ and $D:=3k$. The class is nonempty: the state vector with $x_o=k$ and all other states zero belongs to $\mathcal X_k$, its one leak sitting at the root. Losses and charges are vertex-additive, and the admissible completions on the subtrees of distinct children are independent given the state at their parent, so the minimum is computed by a postorder dynamic program over one state per vertex.

As in \Cref{slem:compress}, every $x\in\mathcal X_k$ vanishes on a subtree disjoint from $A_k$, so such a subtree contributes one state-independent constant, the sum of $Y_w^2$ over its vertices, to every candidate; and a maximal chain of vertices off $A_k$ with a single child subtree meeting $A_k$ carries one common state. Call the tree obtained by pruning the zero subtrees and contracting these chains the \emph{reduced tree}. The cost of such a chain is the explicit quadratic $x\mapsto\sum_{u}(Y_u-bx)^2$ in the state $x$ at the reduced-tree vertex below it, generated in $O(k)$ operations from the chain's length, sum, and sum of squares. Unlike its sum-product counterpart, the min-plus recursion retains these additive constants: the values $m_{v,k}$ of different candidates are later compared, so nothing may be dropped. The reduction partitions the vertices: each vertex $u$ of the reduced tree owns itself, the suppressed chain entering it, and the zero subtrees pruned at $u$ or along that chain. Write $\widehat T_u$ for the union of the blocks owned by $u$ and by its descendants in the reduced tree, so that $\widehat T_o=\sfV$, and define at each such $u$ the table
\[
m_u(x)\ :=\ \min\Bigl\{\sum_{w\in\widehat T_u}(Y_w-bx'_w)^2
+c\,\Gamma\bigl(z'|_{\widehat T_u}\bigr)\Bigr\}
\qquad(0\le x\le D),
\]
the minimum over the state vectors $x'\in\{0,\dots,D\}^{\widehat T_u}$ with $x'_u=x$ whose leaks $z'$ vanish off $A_k$.

\emph{Overflow.} Child states are individually at most $D$, but their sum is not. A partial min-plus product of child tables, with conceptual value $A(t)$ at child sum $t\ge0$, is carried as its truncation $A(0),\dots,A(D)$ together with the overflow scalar
\[
\tau_A\ :=\ \min_{t>D}\ \bigl\{A(t)+c\,(t-D)\bigr\},
\]
the minimum of an empty set being $+\infty$. For the min-plus convolution of two partial products, $H(q)=\min_{i+j=q}\{A(i)+B(j)\}$, the values $H(0),\dots,H(D)$ cost $O(D^2)$ directly, and partitioning the pairs $(i,j)$ with $i+j>D$ according to whether each index exceeds $D$ gives the exact overflow
\begin{multline}\label{eq:minplustail}
\tau_H\ =\ \min\Bigl\{
\min_{\substack{0\le i,j\le D\\ i+j>D}}\{A(i)+B(j)+c(i+j-D)\},
\\
\tau_A+\min_{0\le j\le D}\{B(j)+cj\},\ \
\tau_B+\min_{0\le i\le D}\{A(i)+ci\},\ \
\tau_A+\tau_B+cD\Bigr\},
\end{multline}
since $A(i)+c(i-D)$ ranges over the overflowed side: for instance, $i>D$ and $j\le D$ give $A(i)+B(j)+c(i+j-D)=[A(i)+c(i-D)]+[B(j)+cj]$, minimized by the second block. Each merge therefore costs $O(D^2)=O(k^2)$, and there are fewer than $n$ merges in the whole tree.

\emph{The local step.} Let $(Q(0),\dots,Q(D),\tau_Q)$ be the assembled child-sum table at $u$, an empty product having $Q(0)=0$, $Q(t)=+\infty$ for $t\ge1$, and $\tau_Q=+\infty$, and let $\chi_u(x)$ be the cost of the rest of $u$'s own block: the quadratic cost of the suppressed chain entering $u$, all of whose vertices carry the state $x$, plus the constants of the block's pruned subtrees, zero if the block is the singleton $\{u\}$. If $u\notin A_k$, conservation forces the child sum to equal $x\le D$, the overflow is irrelevant, and $m_u(x)=(Y_u-bx)^2+\chi_u(x)+Q(x)$. If $u\in A_k$, splitting on $z_u=0$ against $z_u\ne0$ gives
\[
m_u(x)\ =\ (Y_u-bx)^2+\chi_u(x)
+\min\Bigl\{Q(x),\ \ c\,\widetilde\omega(u)
+\min_{t\ge0}\bigl[Q(t)+c\,|x-t|\bigr]\Bigr\};
\]
including $t=x$ in the charged branch only adds the option $Q(x)+c\widetilde\omega(u)$, dominated by the first branch, so the split is exact. The inner minimum is a one-dimensional distance transform, the min-plus mirror of the recurrences in \Cref{slem:recursion}: with
\begin{gather*}
F(0):=Q(0),\qquad F(x):=\min\{Q(x),\,F(x-1)+c\};
\\
B(D):=\min\{Q(D),\,\tau_Q\},\qquad B(x):=\min\{Q(x),\,B(x+1)+c\},
\end{gather*}
one has $F(x)=\min_{t\le x}\{Q(t)+c(x-t)\}$ and $B(x)=\min_{t\ge x}\{Q(t)+c(t-x)\}$, the seed at $D$ covering the overflowed sums because $\min_{t>D}\{Q(t)+c(t-x)\}=\tau_Q+c(D-x)$ for $x\le D$; hence $\min_{t\ge0}\{Q(t)+c|x-t|\}=\min\{F(x),B(x)\}$, at cost $O(k)$ for the two passes. At the root, $x_o=k$ is pinned and $\widetilde\omega(o)=0$, so $m_{v,k}$ is the local step evaluated at $x=k$. Storing, at every minimum, an attaining index, an attaining block of \Cref{eq:minplustail}, and an attaining overflow sum, and backtracking from the root, recovers a minimizer $x\in\mathcal X_k$. The total is $O(k^2)$ per merge and $O(k)$ per vertex, hence $O(nk^2)$ operations.
\end{proof}

\begin{proof}[Completion of the proof of Theorem~3, sufficiency]
The estimator is the penalized selection of \Cref{slem:selection} over the candidate system at working scale $s=\sigma$, with penalties $C_2\sigma^2(D_m+\Delta_m)$ and $C_2\sigma^2n$; for $n<n_\star$ it returns $Y$, a case already absorbed. Condition on $Y_o=y_0$: the candidate system is deterministic, the data decompose as $Y=\widetilde\theta+(0,\sigma Z_{-o})$ with $\widetilde\theta=(y_0,\mu_{-o})$, the noise being centered Gaussian with independent coordinates of variances at most $\sigma^2$, the root variance zero, and the penalties are deterministic, so \Cref{slem:selection} applies conditionally with $E$ the whole space, $\bar\kappa=C_2$, and $N=n$. Since $\widetilde\theta-\mu=(Y_o-V)p_o$, the pathwise root replacement
\begin{equation}\label{eq:uncondition}
\|\widehat\mu-\mu\|_2^2\ \le\
2\,\|\widehat\mu-\widetilde\theta\|_2^2+2\,(Y_o-V)^2 ,
\qquad
\E(Y_o-V)^2=\sigma^2 ,
\end{equation}
converts conditional bounds on $\|\widehat\mu-\widetilde\theta\|_2^2$ into risk bounds. Write $k_a:=\kalg(T,V,\sigma)$ and $r:=\min\{V^2H_T,\sigma^2k_a\}$; by Theorem~1 and the sandwich $k_a\le k_0\le2k_a$ of Lemma~2.3(3), $r\asymp R^*_T(V,\sigma)$. The anchor event $\mathcal E_+$ is determined by $Y_o$. Its complement contributes at most $C\sigma^2$ to the risk, by \Cref{slem:anchor}(3); on $\mathcal E_+$, taking conditional expectations in \Cref{eq:uncondition} and applying \Cref{slem:selection} conditionally bounds the contribution by $C\,\E[\mathfrak O'\,\ind_{\mathcal E_+}]+C\sigma^2$, where $\mathfrak O'$ is the conditional oracle value of \Cref{slem:selection}, the minimum over all candidates including the subspaces. It therefore suffices to prove
\begin{equation}\label{eq:oracletarget}
\E\bigl[\mathfrak O'\,\ind_{\mathcal E_+}\bigr]\ \le\ C\,(r+\sigma^2).
\end{equation}
There are three exhaustive regimes; in the first two the bound is pointwise.

\emph{Diameter regime: $V^2H_T\le\sigma^2k_a$.} The model $\operatorname{span}(p_o)$ contains $y_0p_o$, and $\widetilde\theta-y_0p_o$ vanishes at the root and equals $\mu$ off it, so
\[
\operatorname{dist}^2\bigl(\widetilde\theta,\operatorname{span}(p_o)\bigr)
\ \le\ \|\mu-Vp_o\|_2^2\ \le\ V^2h_T\ \le\ r ,
\]
by convexity as in \Cref{slem:branches}. Dimension and weight are one each, so $\mathfrak O'\le r+3\sigma^2$.

\emph{Entropy regime with small crossing: $\sigma^2k_a<V^2H_T$ and $k_a\le k^\star/2$.} Choose the smallest dyadic $k\ge\max\{k_a,2\}$, so that $k\le\max\{2k_a,2\}\le k^\star$ and $k\le4k_a$. Since $k_a\le k^\star/2\le K/2$, the dimension clause of (2.2) is silent at $k_a$, so its surrogate crossing clause fired there: $V^2\overline\delta_{k_a}\le2A_0\sigma^2k_a$. \Cref{slem:subspace} at budget $k$ and the monotonicity of $\overline\delta$ (Lemma~2.3(3)) give an admissible $S$ with
\[
\operatorname{dist}^2\bigl(\mu,\ \mathcal V_{k,S}\bigr)\ \le\
22\,V^2\overline\delta_k\ \le\ 22\,V^2\overline\delta_{k_a}\ \le\
44\,A_0\,\sigma^2k_a .
\]
The root indicator $p_o$ lies in every profile subspace, $o$ being in $R_0$, so translating an approximant of $\mu$ by $(y_0-V)p_o$ shows $\operatorname{dist}(\widetilde\theta,\mathcal V_{k,S}) \le\operatorname{dist}(\mu,\mathcal V_{k,S})$. With $\dim\mathcal V_{k,S}\le|R_0|+|S|\le(e+4)k$ and $\Delta_{k,S}=(C_1+1)k$, the oracle value of this subspace is at most $44A_0\sigma^2k_a+C\sigma^2k\le C\sigma^2k_a=Cr$.

\emph{Entropy regime with large crossing: $\sigma^2k_a<V^2H_T$ and $k_a>k^\star/2$.} Suppose first $V\ge\sigma/2$. On $\mathcal E_+$, \Cref{slem:anchor}(1) with $s=\sigma$ and $\sigma/V\le2$ gives, pointwise,
\[
\frac{v_{i^*}}V\ \le\ 1+\frac{2\sigma}V+\frac{3\sigma|Z_o|}V
\ \le\ 5+6|Z_o| ,
\]
so \Cref{slem:scale}, at the factor $c=v_{i^*}/V\ge1$, gives $R^*_T(v_{i^*},\sigma)\le(5+6|Z_o|)^2R^*_T(V,\sigma)$ pointwise, of expectation at most $CR^*_T(V,\sigma)$. Since $\mathfrak O'\le\mathfrak O$, the infimum running over more candidates, \Cref{eq:anchoredcover} and \Cref{slem:anchor}(1) together prove \Cref{eq:oracletarget} in this regime. Suppose now $V<\sigma/2$. The regime inequality and $H_T<n$ give $\sigma^2k_a<V^2H_T\le V^2n$, so $\sigma/V<\sqrt{n/k_a}\le\sqrt n$ and $j^*=\lfloor\log_2(\sigma/V)\rfloor$ satisfies $1\le j^*\le\tfrac12\log_2n\le J^-$. The low-amplitude bound \Cref{eq:lowcover} applies, and its three terms are handled in turn: $v^-_{j^*}\in[V,2V]$ and \Cref{slem:scale} give $R^*_T(v^-_{j^*},\sigma)\le4R^*_T(V,\sigma)$; the amplitude $v^-_{j^*}=2^{-j^*}\sigma$ is deterministic, so
\[
\E\bigl(Y_o-v^-_{j^*}\bigr)^2\ \le\
2\sigma^2\,\E Z_o^2+2\bigl(V-v^-_{j^*}\bigr)^2\ \le\
2\sigma^2+2V^2\ \le\ 3\sigma^2 ,
\]
using $|V-v^-_{j^*}|\le V<\sigma/2$; and, since $\log(1+t)\le t$ and $k_a>k^\star/2\ge\log n/(8C_1)$ for $n\ge n_\star$,
\[
\sigma^2\log(1+j^*)\ \le\ \sigma^2\,\frac{\log_2n}2\ \le\
C\,\sigma^2k_a\ =\ C\,r .
\]
This proves \Cref{eq:oracletarget} in every regime, hence the risk bound of Theorem~3.

\emph{Running time.} The dyadic budgets $k\le k^\star$ carry $\sum_{k}e^{C_1k}\le2e^{C_1k^\star}\le2\sqrt n$ admissible supports in total, the terms of the sum growing at least geometrically; by \Cref{slem:enum} they are listed in $O(n\sqrt n)$ operations. One $O(n\log n)$ preprocessing pass stores every vertex's $2^j$-th ancestor for $j\le\lceil\log_2n\rceil$ and every root-path sum of $Y$; lifting the deeper endpoint and then both endpoints finds the deepest common ancestor of any pair in $O(\log n)$ operations, and, by the ancestor count recorded with Lemma~2.2, $\langle p_u,p_w\rangle=\depth(u\wedge w)+1$, while $\langle Y,p_u\rangle$ is the stored root-path sum at $u$. Forming and solving the normal equations of one subspace, of dimension $O(\log n)$, therefore costs $O(\log^3n)$ operations after the preprocessing, and the subspace half costs $O(n^{3/2})$ in total, up to logarithmic factors. The grid half has $O(\log n)$ amplitudes, $O(\log n)$ budgets, and, for each pair, one run of \Cref{slem:minplus} with $c=2C_2\sigma^2$: the remaining weight $C_2\sigma^2(k+i+\log_2k+4)$ of a flow candidate is constant across $x\in\mathcal X_k$, so the penalized criterion is minimized by the returned minimizer. With $k\le n$, the grid half costs $O(n^3\log^2n)$ operations, a conservative bound, and the final comparison across all candidates is dominated by it. All tie breaks are fixed, so the estimator is a deterministic function of $(T,\sigma,Y)$, computable in polynomial time.
\end{proof}

It remains to remove the noise level; the difference statistics and their contamination bound come first. Write $\operatorname{hc}(u)$ for the size-heavy child of an internal vertex $u$, fixed in Section~4.3, pair the children of each branching vertex by a fixed rule, one child left unpaired when their number is odd, and take
\begin{equation}\label{eq:diffstats}
X_u\ :=\ Y_u-Y_{\operatorname{hc}(u)}\ \ (u\ \text{internal}),
\qquad
X_{c,c'}\ :=\ Y_c-Y_{c'}\ \ (\{c,c'\}\ \text{a pair of children}),
\end{equation}
at least $(n-1)/2$ statistics in all, each a Gaussian of variance $2\sigma^2$ around a signal difference. For every $\varepsilon>0$, the contamination bound announced in the main body reads
\begin{equation}\label{eq:contamcount}
\#\bigl\{u:\ \E X_u\ge\varepsilon\sigma\bigr\}
\ +\ \#\bigl\{\{c,c'\}:\ |\E X_{c,c'}|\ge\varepsilon\sigma\bigr\}
\ \le\ \frac{V\,w_T}{\varepsilon\sigma}\ .
\end{equation}

\begin{proof}[Proof of \Cref{eq:contamcount}]
Fix $\varepsilon>0$. Both counts rest on the identity, valid for every vertex set $B$ by the subtree sums of Lemma~2.1,
\begin{equation}\label{eq:massidentity}
\sum_{a\in B}\mu(a)\ =\ \sum_{a\in B}\ \sum_{u:\,u\succeq a}s_u
\ =\ \sum_{u\in\sfV}s_u\,\#\{a\in B:\ a\preceq u\}\ ;
\end{equation}
when every root path meets $B$ in at most $w$ vertices, the count on the right is at most $w$, and the leaks total $V$, so $\sum_{a\in B}\mu(a)\le Vw$.

\emph{Heavy differences.} Every internal vertex carries exactly one heavy edge, and every vertex is the size-heavy child of at most one vertex, its parent, so the heavy edges partition $\sfV$ into vertex-disjoint paths; the \emph{top} of such a path is the root or the lower endpoint of a light edge. The flow form of Lemma~2.1 makes every mean $\E X_u=\mu(u)-\mu(\operatorname{hc}(u))$ nonnegative, and along one heavy path the means telescope to $\mu(\mathrm{top})-\mu(\mathrm{bottom})\le\mu(\mathrm{top})$. A root path contains at most $w_{\mathrm{lt}}$ tops, one for the root and one per light edge it crosses, so \Cref{eq:massidentity}, applied to the set of tops, gives
\[
\varepsilon\sigma\,\#\bigl\{u:\ \E X_u\ge\varepsilon\sigma\bigr\}
\ \le\ \sum_{u\ \mathrm{internal}}\E X_u
\ \le\ \sum_{\mathrm{tops}\ a}\mu(a)\ \le\ V\,w_{\mathrm{lt}}\ .
\]

\emph{Sibling pairs.} The coordinates $\mu(c)$ of the children of a branching vertex $u$ are nonnegative, so a pair with $|\E X_{c,c'}|=|\mu(c)-\mu(c')|\ge\varepsilon\sigma$ contains a child whose coordinate is at least $\varepsilon\sigma$; such children number at most $\mu(u)/(\varepsilon\sigma)$, since $\sum_{c\in\ch(u)}\mu(c)\le\mu(u)$, and distinct pairs contain distinct children. Summing over the branching vertices, of which a root path contains at most $w_{\mathrm{br}}$, and applying \Cref{eq:massidentity} to their set,
\[
\#\bigl\{\{c,c'\}:\ |\E X_{c,c'}|\ge\varepsilon\sigma\bigr\}
\ \le\ \frac1{\varepsilon\sigma}\sum_{u\ \mathrm{branching}}\mu(u)
\ \le\ \frac{V\,w_{\mathrm{br}}}{\varepsilon\sigma}\ .
\]
Adding the two displays gives \Cref{eq:contamcount}, the width being $w_T=w_{\mathrm{lt}}+w_{\mathrm{br}}$.
\end{proof}

The statistic family \Cref{eq:diffstats} has at least $(n-1)/2$ members: an internal vertex with $g$ children contributes the difference along its heavy edge together with $\lfloor g/2\rfloor$ sibling pairs, $1+\lfloor g/2\rfloor\ge(g+1)/2$, and the child counts of the internal vertices sum to $n-1$. Each vertex $u$ enters at most three statistics: its own difference $X_u$ when internal, the difference $X_{\pa(u)}$ when it is the size-heavy child of its parent, and at most one pair. Scan the family in a fixed order and retain every statistic both of whose coordinates are unused by the statistics retained so far: a retained statistic blocks at most four others, two through each coordinate, so at least a fifth of the family is retained. This produces a deterministic subfamily, relabeled $X_1,\dots,X_N$, with $N\ge(n-1)/10$ and pairwise disjoint coordinate pairs; the $X_i$ are therefore independent, with $X_i-\E X_i\sim N(0,2\sigma^2)$.

\begin{lemma}[Noise estimate]\label{slem:noise}
Let $n\ge2$, write $\operatorname{med}\{|X_1|,\dots,|X_N|\}$ for the $\lceil N/2\rceil$-th smallest of the absolute values, and set
\[
\widehat\sigma\ :=\
\frac{\operatorname{med}\{|X_1|,\dots,|X_N|\}}
{\sqrt2\,\Phi_{\mathcal N}^{-1}(3/4)},
\]
the denominator being the median of $|N(0,2)|$, with $\Phi_{\mathcal N}$ the standard normal distribution function. If $V\,w_T\le\sigma n/1280$, then
\[
\mathbb P\bigl(\widehat\sigma\in[\sigma/2,\ 2\sigma]\bigr)\ \ge\
1-2e^{-c(n-1)}
\]
with a universal $c>0$.
\end{lemma}

\begin{proof}
Call $X_i$ \emph{clean} when $|\E X_i|\le\sigma/8$. The contamination count \Cref{eq:contamcount} at $\varepsilon=1/8$, applied to the full family and hence to the subfamily, bounds the number of statistics that are not clean by $8Vw_T/\sigma\le n/160\le N/8$, the last step because $N\ge(n-1)/10\ge n/20$ for $n\ge2$. A clean statistic satisfies, for every $t\ge0$,
\[
\mathbb P\bigl(|X_i|\le t\bigr)\ \le\
\mathbb P\Bigl(|N(0,2)|\le\frac t\sigma+\frac18\Bigr),
\qquad
\mathbb P\bigl(|X_i|>t\bigr)\ \le\
\mathbb P\Bigl(|N(0,2)|>\frac t\sigma-\frac18\Bigr),
\]
and two numerical evaluations of the standard normal distribution function give
\[
\mathbb P\bigl(|N(0,2)|\le0.625\bigr)\ <\ 0.35\ <\ \tfrac38,
\qquad
\mathbb P\bigl(|N(0,2)|>1.675\bigr)\ <\ 0.24\ <\ \tfrac38 .
\]
If $\operatorname{med}|X|<\sigma/2$, at least $\lceil N/2\rceil$ of the absolute values are below $\sigma/2$, of which at least $N/2-N/8=3N/8$ belong to clean statistics, and each clean statistic falls below $\sigma/2$ with probability less than $0.35$. If $\operatorname{med}|X|>9\sigma/5$, at least $N-\lceil N/2\rceil+1\ge N/2$ of the absolute values exceed $9\sigma/5$, at least $3N/8$ of them clean, each clean statistic exceeding $9\sigma/5$ with probability less than $0.24$. In either case, a count of at most $N$ independent indicators, each of success probability at most a fixed $p<3/8$, must reach $3N/8$. For $\lambda>0$, Markov's inequality applied to the exponential of $\lambda$ times this count bounds the probability by $[e^{-3\lambda/8}(1-p+pe^\lambda)]^N$, each missing indicator only lowering the moment generating function below the factor $1-p+pe^\lambda>1$; the bracket equals $1$ at $\lambda=0$ with derivative $p-\tfrac38<0$, so a fixed small $\lambda$ makes it $e^{-c}<1$, and each probability is at most $e^{-cN}\le e^{-c(n-1)/10}$; renaming $c$, the two events together have probability at most $2e^{-c(n-1)}$, the claimed exceptional probability. On the complement of the two events, $\operatorname{med}|X|\in[\sigma/2,\,9\sigma/5]$; since $\Phi_{\mathcal N}(0.6718)<\tfrac34<\Phi_{\mathcal N}(0.6788)$ by a numerical evaluation, the denominator $\sqrt2\,\Phi_{\mathcal N}^{-1}(3/4)$ lies in $(0.95,\,0.96)$, and
\[
\widehat\sigma\ \in\
\Bigl[\frac{\sigma/2}{0.96},\ \frac{9\sigma/5}{0.95}\Bigr]
\ \subseteq\ [\sigma/2,\ 2\sigma]. \qedhere
\]
\end{proof}

\begin{proof}[Completion of the proof of Theorem~4]
For $n<n_\star$ return $Y$, whose risk $n\sigma^2\le n_\star\sigma^2$ is absorbed as before; assume $n\ge n_\star$. The estimator computes $\widehat\sigma$; on the null event $\{\widehat\sigma=0\}$ it returns $Y$, and otherwise it rounds upward to $\widetilde\sigma:=2^{\lceil\log_2\widehat\sigma\rceil}$ and runs the selection of Theorem~3 at working scale $s=\widetilde\sigma$, with the weight of every subspace and point model increased by two and the penalties built from these shifted weights at scale $\widetilde\sigma$. The estimator depends only on $(T,Y)$, and its running time adds $O(n\log n)$ for the statistics and the median to the polynomial cost already accounted. We prove the risk bound with $c=1/1280$ in the regime condition.

Let $\widehat E:=\{\widehat\sigma\in[\sigma/2,2\sigma]\}$, so that $\mathbb P(\widehat E^c)\le2e^{-c(n-1)}$ by \Cref{slem:noise} under the regime condition. On $\widehat E$, the rounding gives $\widetilde\sigma\in[\sigma/2,4\sigma]$, and $\widetilde\sigma$ is a power of two in that interval, hence takes at most four values determined by $\sigma$.

\emph{On $\widehat E$.} Condition on $Y_o=y_0$. The realized candidate-and-penalty system is one of the four deterministic systems indexed by the possible values of $\widetilde\sigma$; the subspace candidates are common to all four, while the grids differ. The penalties satisfy the bracketing of \Cref{slem:selection} with respect to the true $\sigma$ on $\widehat E$: with shifted weights $\Delta_m+2$, since $\widetilde\sigma^2\in[\sigma^2/4,16\sigma^2]$ and $C_2/4=32$,
\[
32\,\sigma^2(D_m+\Delta_m+2)\ \le\
C_2\widetilde\sigma^2(D_m+\Delta_m+2)\ \le\
16C_2\,\sigma^2(D_m+\Delta_m+2),
\]
and likewise for the full model, so $E=\widehat E$ and $\bar\kappa=16C_2$. The proof of \Cref{slem:selection} is repeated with one change: the deviation variable is the supremum over the four systems, each paired with its own comparison model. The weight shift by two keeps the combined Kraft sum below one, since $4e^{-2}\bigl(0.53+\tfrac18\bigr)<1$ by \Cref{eq:kraftsum} and the subspace sum, so the expectation of the deviation variable is still at most one; the basic inequality is unchanged, the selected and the comparison model both lying in the realized subsystem. The conditional oracle bound of \Cref{slem:selection} therefore holds on $\widehat E$ for the realized system.

Coverage now runs as in the completion of Theorem~3, at the realized scale $s=\widetilde\sigma\in[\sigma/2,4\sigma]$, with $k_a(s):=\kalg(T,V,s)$ and $r_s:=\min\{V^2H_T,s^2k_a(s)\}$; by \Cref{eq:ratescale}, $r_s\le CR^*_T(V,\sigma)$ uniformly over the four values, and the weight shift adds $2\sigma^2$ to every oracle value, absorbed by the additive term. In the diameter regime, the pointwise bound through $\operatorname{span}(p_o)$ is unchanged. In the entropy regime with $k_a(s)\le k^\star/2$, the surrogate crossing at scale $s$ gives $V^2\overline\delta_{k_a(s)}\le2A_0s^2k_a(s)$, and the subspace bound reads $22V^2\overline\delta_k\le44A_0s^2k_a(s)\le Cr_s$. In the entropy regime with $k_a(s)>k^\star/2$: for $V\ge s/2$, \Cref{slem:anchor}(1) gives, pointwise,
\[
\frac{v_{i^*}}V\ \le\ 1+\frac{2s}V+\frac{3\sigma|Z_o|}V\ \le\
5+12\,|Z_o| ,
\]
using $s/V\le2$ and $\sigma\le2s$, so \Cref{eq:anchoredcover}, \Cref{slem:scale}, and \Cref{slem:anchor}(1) give the bound $C(R^*_T(V,\sigma)+\sigma^2)$ in expectation as before. For $V<s/2$, the index $j^*=\lfloor\log_2(s/V)\rfloor$ again lies in $\{1,\dots,J^-\}$, since $s^2k_a(s)<V^2H_T\le V^2n$; the amplitude $v^-_{j^*}=2^{-j^*}\widetilde\sigma$ now depends on the data through $\widetilde\sigma$, so no independence from $Z_o$ is available, and the root-coordinate term of \Cref{eq:lowcover} is bounded pathwise on $\widehat E$ instead:
\[
\bigl(Y_o-v^-_{j^*}\bigr)^2\ \le\ 2\sigma^2Z_o^2
+2\bigl(V-v^-_{j^*}\bigr)^2\ \le\ 2\sigma^2Z_o^2+2V^2\ \le\
2\sigma^2Z_o^2+8\sigma^2 ,
\]
using $v^-_{j^*}\in[V,2V]$ and $V<s/2\le2\sigma$, of expectation $O(\sigma^2)$; the weight term $\sigma^2\log(1+j^*)\le C\sigma^2k_a(s)=Cr_s$ as before. Finally, for every realized $s$ the omitted-anchor event is contained in $\{Z_o<-2(n+1)\}$ by \Cref{slem:anchor}(2), and \Cref{slem:anchor}(3) with $s^2\le16\sigma^2$ bounds its loss contribution by $C\sigma^2$. Splitting along the realized anchor event and applying the pathwise root replacement \Cref{eq:uncondition} on it, as in the completion of Theorem~3, gives $\E[\|\widehat\mu-\mu\|_2^2\,\ind_{\widehat E}]\le C(R^*_T(V,\sigma)+\sigma^2)$.

\emph{Off $\widehat E$.} On $\{\widehat\sigma=0\}$ the loss is $\sigma^2\|Z\|_2^2$. On $\{\widehat\sigma>0\}\cap\widehat E^c$, compare the selected criterion with $\operatorname{span}(p_o)$, of dimension one and shifted weight three:
\[
\|Y-\widehat\mu\|_2^2\ \le\
\bigl\|Y-P_{\operatorname{span}(p_o)}Y\bigr\|_2^2
+4C_2\widetilde\sigma^2\ \le\ \|Y\|_2^2+4C_2\widetilde\sigma^2 ,
\]
so $\|\widehat\mu-\mu\|_2\le\|\widehat\mu-Y\|_2+\|Y-\mu\|_2\le 2\|Y\|_2+\|\mu\|_2+2\sqrt{C_2}\,\widetilde\sigma$. Pathwise, $\widetilde\sigma\le2\widehat\sigma\le3\operatorname{med}|X|\le 3\max_i|X_i|\le6\|Y\|_\infty\le6\|Y\|_2$, each $X_i$ being a difference of two coordinates of $Y$, so all losses off $\widehat E$ obey $\|\widehat\mu-\mu\|_2\le C(\|Y\|_2+\|\mu\|_2)$. For the moments: $\mu(u)\le V$ at every vertex and $\sum_u\mu(u)=\sum_xs_x(\depth(x)+1)\le Vn$ by Lemma~2.1, so $\|\mu\|_2^2\le V^2n$; $\E\|Z\|_2^4\le3n^2$ gives $\E\|Y\|_2^4\le8(\|\mu\|_2^4+\sigma^4\E\|Z\|_2^4)\le 8(V^4n^2+3\sigma^4n^2)$; and the regime condition with $w_T\ge1$ gives $V\le\sigma n/1280$, hence $V^4n^2\le\sigma^4n^6$. The squared loss therefore has second moment at most $C\sigma^4n^6$ off $\widehat E$, and the Cauchy--Schwarz inequality gives
\[
\E\bigl[\|\widehat\mu-\mu\|_2^2\,\ind_{\widehat E^c}\bigr]\ \le\
\bigl(C\sigma^4n^6\bigr)^{1/2}\bigl(2e^{-c(n-1)}\bigr)^{1/2}\ \le\
C\sigma^2n^3e^{-c(n-1)/2}\ \le\ C\sigma^2 .
\]
Adding the two contributions proves Theorem~4.
\end{proof}

The width of the benchmark families is read off the definition in Section~4.3. A path has no light edge and no branching vertex, so $w_T=1$. On the complete binary tree of height $h$, every internal vertex has one light child edge, and a root-to-leaf path may cross a light edge and a branching vertex at each of its $h$ internal levels, so $w_T=(1+h)+h=2h+1$. On a star with $m\ge2$ leaves, and on the broom of Theorem~5, a root-to-leaf path crosses at most one light edge and one branching vertex, so $w_T\le3$.

The regime of Theorem~4 cannot be removed entirely by a better noise estimate. The following converse, promised at the end of Section~4.3, shows that once budgets of order $\sigma n\sqrt{\log n}$ are admitted, no estimator lands strictly within a factor $\sqrt2$ of the noise level, uniformly, on any tree.

\begin{prop}[Noise-estimation ceiling]\label{sprop:ceiling}
For every finite rooted tree and every $\sigma>0$, put
$\lambda_n:=\log\bigl(4(n+1)\bigr)$ and
$V_1:=6\,\sigma n\sqrt{\lambda_n}$. In the model $Y=\mu+sZ$,
$Z\sim N(0,I_n)$, where $\mu$ is a monotone flow of root value at
most $V_1$, that is, $\mu\in\bigcup_{0\le V'\le V_1}\mathcal
F_{V'}(T)$ with $\mathcal F_0(T):=\{0\}$, and $s\in\{\sigma,2\sigma\}$, every estimator
$\widehat s=\widehat s(Y)$ satisfies
\[
\sup_{(\mu,s)}\ \mathbb P_{\mu,s}\Bigl(\widehat s\notin
\bigl(s/\sqrt2,\ \sqrt2\,s\bigr)\Bigr)\ \ge\ \frac38 ,
\]
the supremum running over the admitted pairs. The same
bound holds for randomized estimators, by averaging over their
internal randomness.
\end{prop}

\begin{proof}
The mechanism is a scale mixture: doubling the noise of a Gaussian shift is a random perturbation of its mean, and a signal whose flow inequalities all carry enough slack absorbs the perturbation without leaving the cone.

\emph{Step 1: the mixture identity.} In one coordinate, with $a:=3\sigma^2$ and $b:=\sigma^2$, the exponent of the convolution integrand of $N(0,a)$ and $N(0,b)$ completes the square as
\[
\frac{w^2}{2a}+\frac{(y-w)^2}{2b}\ =\
\frac{a+b}{2ab}\Bigl(w-\frac{ay}{a+b}\Bigr)^2
+\frac{y^2}{2(a+b)}\ ,
\]
and integrating out $w$ leaves a Gaussian integral at variance $ab/(a+b)$; the normalizers combine into that of $N(0,a+b)$, so the convolution is $N(0,4\sigma^2)$. Coordinates being independent, with $\gamma:=N(0,3\sigma^2I_n)$ and any fixed $\mu_0\in\bbR^n$,
\begin{equation}\label{eq:mixture}
N\bigl(\mu_0,\,4\sigma^2I_n\bigr)\ =\
\int N\bigl(\mu_0+w,\ \sigma^2I_n\bigr)\,d\gamma(w):
\end{equation}
the law of $Y$ at signal $\mu_0$ and noise $2\sigma$ is a $\gamma$-mixture of laws at noise $\sigma$ and shifted signals.

\emph{Step 2: a signal with slack.} Define $\mu_0$ through its leaks,
\[
s_v\ :=\ \sigma\sqrt{6\,(1+|\ch(v)|)\,\lambda_n}
\qquad(v\in\sfV),
\qquad
V_0\ :=\ \sum_{v}s_v ,
\]
so that $\mu_0\in\mathcal F_{V_0}(T)$ by Lemma~2.1. The Cauchy--Schwarz inequality and $\sum_v(1+|\ch(v)|)=n+(n-1)\le2n$ give $V_0\le\sigma\sqrt{6\lambda_n}\,\sqrt n\,\sqrt{2n} =\sigma n\sqrt{12\lambda_n}\le V_1$, so $\mu_0$ is admitted. Define the legality event
\[
A\ :=\ \Bigl\{w\in\bbR^{\sfV}:\
w_v-\textstyle\sum_{c\in\ch(v)}w_c\ \ge\ -s_v\ \text{ for all }
v\in\sfV\Bigr\}\ \cap\ \bigl\{w_o\le\sigma\sqrt{6\lambda_n}\bigr\}.
\]
The leak map is linear, so on $A$ the leak of $\mu_0+w$ at $v$ is $s_v+(w_v-\sum_cw_c)\ge0$: the shifted signal is a monotone flow, of root value
\[
V_0+w_o\ \le\ \sigma n\sqrt{12\lambda_n}+\sigma\sqrt{6\lambda_n}
\ \le\ \bigl(\sqrt{12}+\sqrt6\bigr)\,\sigma n\sqrt{\lambda_n}
\ \le\ V_1 ,
\]
hence admitted. Under $\gamma$, the constraint at $v$ fails when a centered Gaussian of variance $3\sigma^2(1+|\ch(v)|)$ exceeds $s_v$, and the root cap fails when a centered Gaussian of variance $3\sigma^2$ exceeds $\sigma\sqrt{6\lambda_n}$; by the Gaussian tail bound each failure has probability at most $e^{-\lambda_n}$, so
\[
\gamma(A^c)\ \le\ (n+1)\,e^{-\lambda_n}\ =\ \frac14 .
\]

\emph{Step 3: two inseparable hypotheses.} Let $P:=N(\mu_0,4\sigma^2I_n)$, the law of $Y$ at the admitted pair $(\mu_0,2\sigma)$, and let
\[
Q_A\ :=\ \frac1{\gamma(A)}\int_AN\bigl(\mu_0+w,\
\sigma^2I_n\bigr)\,d\gamma(w),
\]
a mixture of laws at admitted pairs with noise level $\sigma$. Splitting \Cref{eq:mixture} over $A$ and $A^c$ gives $P=\gamma(A)Q_A+\gamma(A^c)Q_{A^c}$, so every event $B$ satisfies
\[
\bigl|P(B)-Q_A(B)\bigr|\ =\ \gamma(A^c)\,
\bigl|Q_{A^c}(B)-Q_A(B)\bigr|\ \le\ \frac14 .
\]
Given $\widehat s$, let $B:=\{\widehat s<\sqrt2\,\sigma\}$. Then $Q_A(B^c)+P(B)=1-[Q_A(B)-P(B)]\ge\tfrac34$, so at least one term is at least $\tfrac38$. If $P(B)\ge\tfrac38$: at $(\mu_0,2\sigma)$ the event $B$ reads $\{\widehat s<s/\sqrt2\}$, outside the window. If $Q_A(B^c)\ge\tfrac38$: a mixture average is at least $\tfrac38$, so some $w\in A$ has $\mathbb P_{\mu_0+w,\,\sigma}(B^c)\ge\tfrac38$, and at $(\mu_0+w,\sigma)$ the event $B^c$ reads $\{\widehat s\ge\sqrt2\,s\}$, outside the window. A randomized estimator is handled by averaging the display over its internal randomness.
\end{proof}

The statements deferred from Section~4.3 are now all in place.

%%%%%%%%%%%%%%%%%%%%%%%%%%%%%%%%%%%%%%%%%%%%%%%%%%%%%%%%%%%%%%%%%%%%%%%%%%%%%%
\section{The Universal Upper Bound for Least Squares}\label{ssec:lse}
%%%%%%%%%%%%%%%%%%%%%%%%%%%%%%%%%%%%%%%%%%%%%%%%%%%%%%%%%%%%%%%%%%%%%%%%%%%%%%

This section proves the upper bound of Theorem~5, following the three blocks sketched in Section~5: the localized projection principle (\Cref{slem:localized}), the truth-localized width bound (5.10) (\Cref{slem:truthwidth}), which rests on a noise-compatible choice of centers in the coded cover (\Cref{slem:compatible}), and the closure of the exponent, for which \Cref{slem:heightfloor} supplies the height branch. The capped order-statistics bound of \Cref{slem:capped} does the repeated work inside \Cref{slem:compatible}.

Throughout, $Z\sim N(0,I_n)$ is the noise, $\ell_n:=1+\log(en)$, and, for an integer $2\le k\le n/C_\star$, the construction of Proposition~3.1 is in force with the notation and of \Cref{ssec:cover}: the working scale $\bar\alpha=k\overline\delta_k(T)$, the nets $R_j$ with level weights $m_j=2^j$ and deepest-ancestor maps $a_j(\cdot)$, and, for a signal $f\in\mathcal F_1(T)$, the terminal cells $L$ with masses $\lambda_L$ and levels $m(L)$, the collapse $g$, the skeleton $g_0$ supported on $U$, its rounding $\bar g_0$, the exits $q_L$, and the sampling variables $Z_L$ with selection weight $W$, a notation fixed in the proof of Proposition~3.1; the unsubscripted $Z$ is always the noise. Root coordinates never contribute: every point of $\mathcal F_V(T)$ has root coordinate $V$, so every difference below vanishes there. The Gaussian tail bound $\mathbb P(N(0,\tau^2)\ge u)\le e^{-u^2/(2\tau^2)}$, valid for $u\ge0$, and Gaussian concentration for Lipschitz functions \citep{borell1975,tsirelson1976}, namely $\mathbb P\bigl(F(Z)\ge\E F(Z)+u\bigr)\le e^{-u^2/(2\Lambda^2)}$ for $\Lambda$-Lipschitz $F:\bbR^n\to\bbR$, are used without further comment.

The first lemma turns a bound on the localized Gaussian width into a risk bound for the projection. It is the fixed-point principle of Chatterjee's analysis of least squares \citep{chatterjee2014}, in the form quoted in Section~5; nothing in it refers to trees.

\begin{lemma}[Localized projection]\label{slem:localized}
Let $C\subseteq\bbR^N$ be compact and convex, let $\mu\in C$, let $Y=\mu+\sigma Z$ with $Z\sim N(0,I_N)$, and let $\widehat\mu:=\Pi_C(Y)$ be the Euclidean projection of $Y$ onto $C$. Suppose that $t\ge\sigma$ and that for every $s\ge t$
\begin{equation}\label{eq:widthcondition}
\sigma\,\E\,\sup\bigl\{\langle Z,x-\mu\rangle:\ x\in C,\
\|x-\mu\|_2\le s\bigr\}\ \le\ \frac{s^2}4\,.
\end{equation}
Then $\E\|\widehat\mu-\mu\|_2^2\le33\,t^2$.
\end{lemma}

\begin{proof}
Write $\widehat z:=\widehat\mu-\mu$ and, for $s>0$, let $X_s$ denote the supremum in \Cref{eq:widthcondition}, so that the hypothesis reads $\E X_s\le s^2/(4\sigma)$ for $s\ge t$. The projection is the nearest point of $C$ and $\mu\in C$, so $\|Y-\widehat\mu\|_2^2\le\|Y-\mu\|_2^2$; expanding both sides at $Y=\mu+\sigma Z$ and cancelling $\sigma^2\|Z\|_2^2$ gives
\begin{equation}\label{eq:projopt}
\|\widehat z\|_2^2\ \le\ 2\sigma\,\langle Z,\widehat z\rangle .
\end{equation}
Fix $s\ge t$ and suppose $\|\widehat z\|_2\ge s$. The point $\mu+s\widehat z/\|\widehat z\|_2$ lies in $C$ by convexity, on the segment from $\mu$ to $\widehat\mu$, and within distance $s$ of $\mu$, so
\[
X_s\ \ge\ \frac s{\|\widehat z\|_2}\,\langle Z,\widehat z\rangle\
\ge\ \frac{s\,\|\widehat z\|_2}{2\sigma}\ \ge\ \frac{s^2}{2\sigma}\,,
\]
the middle inequality by \Cref{eq:projopt}. As a function of $Z$ the supremum $X_s$ is $s$-Lipschitz, being a maximum of linear functionals with gradients of norm at most $s$, so concentration and $\E X_s\le s^2/(4\sigma)$ give
\[
\mathbb P\bigl(\|\widehat z\|_2\ge s\bigr)\ \le\
\mathbb P\Bigl(X_s\ge\E X_s+\frac{s^2}{4\sigma}\Bigr)\ \le\
\exp\Bigl(-\frac{(s^2/4\sigma)^2}{2s^2}\Bigr)\ =\
e^{-s^2/(32\sigma^2)}
\qquad(s\ge t).
\]
The layer-cake formula integrates the tail:
\[
\E\|\widehat z\|_2^2\ \le\ t^2+\int_t^\infty2s\,
e^{-s^2/(32\sigma^2)}\,ds\ =\ t^2+32\sigma^2e^{-t^2/(32\sigma^2)}\
\le\ t^2+32\sigma^2\ \le\ 33\,t^2,
\]
the last step by $t\ge\sigma$.
\end{proof}

For $C=\mathcal F_V(T)$ the expectation in \Cref{eq:widthcondition} is the localized width $w_\mu(s)$ of Section~5, and the lemma is the principle stated there. Bounding $w_\mu$ requires expected maxima of many correlated Gaussian terms whose weights are capped individually and in total; the next lemma is the tool. For $a\in(0,1]$ and $y\in\bbR^M$ define
\begin{equation}\label{eq:capfunctional}
\mathfrak M_a(y)\ :=\ \sup\Bigl\{\,\sum_{i=1}^Mw_iy_i:\ 0\le w_i\le
a\ \text{ for all }i,\ \ \sum_{i=1}^Mw_i\le1\Bigr\},
\end{equation}
the largest weighted sum under a cap on each weight and on the total.

\begin{lemma}[Capped order statistics]\label{slem:capped}
Let $G_1,\dots,G_M$ be jointly Gaussian, each centered with variance at most $\tau^2$, not necessarily independent, and let $a\in(0,1]$. Then
\begin{gather*}
\E\,\mathfrak M_a(G_1,\dots,G_M)\ \le\
C\tau\sqrt{1+[\log(Ma)]_+}\,,
\\
\E\,\mathfrak M_a\bigl(G_1^2,\dots,G_M^2\bigr)\ \le\
C\tau^2\bigl(1+[\log(Ma)]_+\bigr).
\end{gather*}
\end{lemma}

\begin{proof}
Write $G_{(1)}^+\ge G_{(2)}^+\ge\cdots$ for the decreasing rearrangement of the positive parts $[G_i]_+$ and $Q_{(1)}\ge Q_{(2)}\ge\cdots$ for that of the squares $G_i^2$, and set $q:=\min\{\lceil1/a\rceil,M\}$.

\emph{Reduction to order statistics.} A feasible weight vector puts total weight at most $1$, each entry at most $a$, and gains nothing from coordinates where $G_i<0$; moving weight from a smaller to a larger positive coordinate only increases the sum, so the supremum is attained by filling the largest positive coordinates to capacity $a$, with at most one fractional entry. Since $aq\ge1$ when $q=\lceil1/a\rceil$, while all $M$ coordinates are already covered when $q=M$,
\begin{equation}\label{eq:capbyorder}
\mathfrak M_a(G)\ \le\ a\sum_{j\le q}G_{(j)}^+,
\qquad
\mathfrak M_a(G_1^2,\dots,G_M^2)\ \le\ a\sum_{j\le q}Q_{(j)} .
\end{equation}

\emph{Tails of the order statistics.} If $G_{(j)}^+\ge u$ for some $u>0$, at least $j$ of the $G_i$ are at least $u$; Markov's inequality applied to the count, whose mean is at most $Me^{-u^2/(2\tau^2)}$ by the Gaussian tail bound, gives
\[
\mathbb P\bigl(G_{(j)}^+\ge u\bigr)\ \le\
\min\Bigl\{1,\ \frac Mj\,e^{-u^2/(2\tau^2)}\Bigr\},
\qquad
\mathbb P\bigl(Q_{(j)}\ge u^2\bigr)\ \le\
\min\Bigl\{1,\ \frac{2M}j\,e^{-u^2/(2\tau^2)}\Bigr\},
\]
the second by the same argument applied to $|G_i|$. Splitting the first tail at $u_j$, where $u_j^2:=2\tau^2\log(eM/j)$, gives
\begin{multline*}
\E\,G_{(j)}^+\ \le\ u_j+\frac Mj\int_{u_j}^\infty
e^{-u^2/(2\tau^2)}\,du
\\
\le\ u_j+\frac Mj\cdot\frac{\tau^2}{u_j}\,
e^{-u_j^2/(2\tau^2)}\ =\ u_j+\frac{\tau^2}{e\,u_j}\ \le\
C\tau\sqrt{\log\frac{eM}j}\,,
\end{multline*}
where the integrand is at most $(u/u_j)\,e^{-u^2/(2\tau^2)}$. Splitting the second tail at $2\tau^2\log(2eM/j)$ gives
\[
\E\,Q_{(j)}\ \le\ 2\tau^2\log\frac{2eM}j+\frac{2\tau^2}e\ \le\
C\tau^2\log\frac{eM}j\,.
\]

\emph{Summation.} Summing the first bound over $j\le q$ and applying the Cauchy--Schwarz inequality to the average,
\[
a\sum_{j\le q}\E\,G_{(j)}^+\ \le\ Ca\tau\,q\,\Bigl(\frac1q
\sum_{j\le q}\log\frac{eM}j\Bigr)^{1/2}\ \le\
Ca\tau\,q\sqrt{2+\log(M/q)}\,,
\]
using $\sum_{j\le q}\log(eM/j)=q\log(eM)-\log q!\le q\bigl(2+\log(M/q)\bigr)$, by $\log q!\ge q\log q-q$. Now $aq\le a\lceil1/a\rceil\le1+a\le2$; and either $q=\lceil1/a\rceil$, in which case $M/q\le Ma$, or $q=M$, in which case $\log(M/q)=0$. In both cases the right side is at most $C\tau\sqrt{1+[\log(Ma)]_+}$, which with \Cref{eq:capbyorder} proves the first claim. The second is identical with $\tau^2\log(eM/j)$ in place of $\tau\sqrt{\log(eM/j)}$, without the Cauchy--Schwarz step.
\end{proof}

The heart of the matter is the next lemma: for each signal and noise realization, take the best accurate member of the coded cover; the resulting pairing with the noise, maximized over the body, is small in expectation. Its proof prices the three residuals of the construction separately, and only the last one, the sampled exits, needs the noise-dependent choice.

\begin{lemma}[Noise-compatible centers]\label{slem:compatible}
Fix an integer $2\le k\le n/C_\star$ and let $\mathcal D_k$ be the coded cover at amplitude $V=1$. For $f\in\mathcal F_1(T)$ and a noise realization $Z$ set
\begin{equation}\label{eq:compatibleinf}
\mathfrak r_Z(f)\ :=\ \inf\Bigl\{\langle Z,\ f-\nu\rangle:\
\nu\in\mathcal D_k,\ \|f-\nu\|_2^2\le51\,\overline\delta_k(T)\Bigr\}.
\end{equation}
The infimum runs over a nonempty set for every $f$ and every $Z$, and
\begin{equation}\label{eq:compatiblebound}
\E\,\sup_{f\in\mathcal F_1(T)}\mathfrak r_Z(f)\ \le\
C\Bigl[\sqrt{\bar\alpha}\,\log(2+\ell_n)+
\sqrt{\overline\delta_k(T)\,\ell_n}\,\Bigr].
\end{equation}
\end{lemma}

\begin{proof}
Fix $f$ and run the proof of Proposition~3.1 on it, retaining its objects. Each realization of the sampling defines the candidate $\nu:=\bar g_0+\sum_LZ_Lq_L$, and the error decomposes as there:
\begin{equation}\label{eq:threeresiduals}
f-\nu\ =\ (f-g)\ +\ (g_0-\bar g_0)\ +\ \sum_L(\lambda_L-Z_L)\,q_L .
\end{equation}
On the sampling event of that proof, call it $A$, of probability at least $\tfrac14$, where $W\le2k$ and $\|\sum_L(Z_L-\lambda_L)q_L\|_2^2\le8\bar\alpha/k$, the candidate lies in $\mathcal D_k$, and the three residuals assemble, as in \Cref{ssec:cover}, to $\|f-\nu\|_2^2\le3(3+6+8)\bar\alpha/k=51\,\overline\delta_k(T)$: the feasible set of \Cref{eq:compatibleinf} is nonempty, for every $f$ and irrespective of $Z$. Pairing \Cref{eq:threeresiduals} with the noise and taking on $A$ the realization that minimizes the third term,
\begin{equation}\label{eq:pairingsplit}
\mathfrak r_Z(f)\ \le\ \langle Z,f-g\rangle+\langle Z,g_0-\bar
g_0\rangle+\min_{A}\,\Bigl\langle
Z,\sum_L(\lambda_L-Z_L)q_L\Bigr\rangle .
\end{equation}
The three terms are bounded uniformly over $f$ in turn; the first carries the main contribution, the second is a union bound over few subspaces, and the third is where choosing the center after the noise pays.

\emph{The collapse residual.} Let $\tau(u)$ be the level of the terminal cell containing $u$, so that $g=\sum_u\lambda_u\,p_{a_{\tau(u)}(u)}$ and the difference telescopes through the levels:
\begin{gather*}
f-g\ =\ \sum_u\lambda_u\bigl(p_u-p_{a_{\tau(u)}(u)}\bigr),
\\
p_u-p_{a_{\tau(u)}(u)}\ =\
\sum_{j=\tau(u)}^{J-1}\bigl(p_{a_{j+1}(u)}-p_{a_j(u)}\bigr)
+\bigl(p_u-p_{a_J(u)}\bigr).
\end{gather*}
Fix a level $j<J$ and group the contributions by the level-$(j{+}1)$ root: by \Cref{slem:cells}(\ref{sit:refine}), $a_j(u)=a_j(b)$ whenever $a_{j+1}(u)=b$, so the level-$j$ increment of every $u$ in the group of $b$ is the same vector $p_b-p_{a_j(b)}$, of squared norm at most $d_T(a_j(b),b)\le\bar\alpha/m_j$ by \Cref{slem:cells}(\ref{sit:radius}). The group mass
\[
\lambda_b\ :=\ \sum\bigl\{\lambda_u:\ a_{j+1}(u)=b,\
\tau(u)\le j\bigr\}
\]
is capped: each member of the group lies in the level-$(j{+}1)$ cell of $b$ and in its own terminal cell, which is coarser and hence contains that whole cell; two terminal cells with a common vertex coincide, the terminal cells being a partition, so the group shares one terminal cell, stopped at a level at most $j<J$ and hence light, whence $\lambda_b\le m_j/k$ by the monotonicity of the level weights. The group masses are nonnegative and total at most one, each $\lambda_u$ entering at most one group, so, with $M_j\le\min\{n,\,|R_{j+1}|\}\le\min\{n,\,2ke^{2m_j}\}$ possible groups (\Cref{slem:netsizes}(\ref{sit:union})),
\[
\E\,\sup_f\,\sum_b\lambda_b\,\langle Z,p_b-p_{a_j(b)}\rangle
\le
\E\,\mathfrak M_{m_j/k}\Bigl(\bigl(\langle Z,p_b-p_{a_j(b)}\rangle
\bigr)_b\Bigr)
\le
C\sqrt{\bar\alpha}\,\min\bigl\{1,\sqrt{\ell_n/m_j}\bigr\},
\]
by \Cref{slem:capped} with variance bound $\bar\alpha/m_j$ and cap $m_j/k$: its logarithmic factor $1+[\log(M_jm_j/k)]_+$ is at most $\min\{5m_j,\,\ell_n\}$, being at most $1+\log2+\log m_j+2m_j\le5m_j$ from $M_jm_j/k\le2m_je^{2m_j}$, and at most $1+\log n\le\ell_n$ from $M_jm_j/k\le n$. Summing the dyadic levels: those with $m_j\le\ell_n$ number at most $C\log(2+\ell_n)$ and contribute $C\sqrt{\bar\alpha}$ each, while over the rest $\sqrt{\ell_n/m_j}$ is geometrically decaying from below one, and sums to a constant. The final increments $p_u-p_{a_J(u)}$ number at most $n$ and have squared norms at most $\bar\alpha/m_J\le2\overline\delta_k(T)$, and their mass-weighted sum is a feasible value of $\mathfrak M_1$; so \Cref{slem:capped} with cap $1$ and variance bound $2\overline\delta_k(T)$ bounds its expected supremum by $C\sqrt{\overline\delta_k(T)}\,\sqrt{1+\log n}\le C\sqrt{\overline\delta_k(T)\,\ell_n}$. Altogether
\begin{equation}\label{eq:collapseresidual}
\E\,\sup_f\ \langle Z,\ f-g\rangle\ \le\
C\Bigl[\sqrt{\bar\alpha}\,\log(2+\ell_n)+
\sqrt{\overline\delta_k(T)\,\ell_n}\,\Bigr].
\end{equation}

\emph{The rounding residual.} The difference $g_0-\bar g_0$ has squared norm at most $6\bar\alpha/k$, by the constant recorded after \Cref{slem:heavy}, and lies in the span of $\{p_h:h\in U\}$. The set $U\setminus R_0$ consists of heavy roots, of total activation charge below $2k$ by \Cref{slem:heavy}(3); it is therefore one of the charged parts counted in the proof of \Cref{slem:count}, so at most $e^{19k}$ sets $U$, hence at most $e^{19k}$ subspaces $V_U:=\operatorname{span}\{p_h:h\in U\}$, occur as $f$ varies. Each has dimension $|U|\le|R_0|+k\le ek+k\le4k$, since a heavy root outside $R_0$ carries charge at least $m_1=2$. For any fixed $U$,
\[
\sup\bigl\{\langle Z,z\rangle:\ z\in V_U,\
\|z\|_2^2\le6\bar\alpha/k\bigr\}\ =\
\sqrt{6\bar\alpha/k}\,\bigl\|P_{V_U}Z\bigr\|_2 ,
\]
and $\|P_{V_U}Z\|_2$ is a $1$-Lipschitz function of $Z$ with mean at most $\sqrt{\dim V_U}\le2\sqrt k$. Concentration and a union bound give, for $u\ge0$, $\mathbb P(\max_U\|P_{V_U}Z\|_2\ge2\sqrt k+\sqrt{38k}+u)\le e^{19k}e^{-(\sqrt{38k}+u)^2/2}\le e^{-u^2/2}$, so $\E\max_U\|P_{V_U}Z\|_2\le2\sqrt k+\sqrt{38k}+2\le C\sqrt k$ and
\begin{equation}\label{eq:roundingresidual}
\E\,\sup_f\ \langle Z,\ g_0-\bar g_0\rangle\ \le\
\sqrt{6\bar\alpha/k}\cdot C\sqrt k\ =\ C\sqrt{\bar\alpha}\,.
\end{equation}

\emph{The exit pairing.} Condition on $Z$ and on $f$, so that only the sampling is random, and write $X:=\langle Z,\sum_L(\lambda_L-Z_L)q_L\rangle$. The sampling variables are independent with $\E Z_L=\lambda_L$ and $\operatorname{Var}(Z_L)\le\lambda_Lm(L)/k$, so $X$ is centered with
\[
\E\,X^2\ =\ \sum_L\operatorname{Var}(Z_L)\,\langle Z,q_L\rangle^2\
\le\ \sum_L\lambda_L\,\frac{m(L)}k\,\langle Z,q_L\rangle^2\ =:\
S_f(Z)^2 .
\]
Centeredness and the Cauchy--Schwarz inequality give
\[
\E\bigl[X\ \big|\ A\bigr]\ =\
-\frac{\E\bigl[X\,\ind_{A^c}\bigr]}{\mathbb P(A)}\ \le\
\frac{(\E X^2)^{1/2}\,\mathbb P(A^c)^{1/2}}{\mathbb P(A)}\ \le\
4\,S_f(Z),
\]
and some realization in $A$ is no larger than this conditional expectation. Hence the third term of \Cref{eq:pairingsplit} is at most $4S_f(Z)$, and it remains to prove
\begin{equation}\label{eq:exitsecond}
\E\,\sup_f\ S_f(Z)^2\ \le\ C\bar\alpha ,
\end{equation}
which gives $\E\sup_fS_f(Z)\le\sqrt{\E\sup_fS_f(Z)^2}\le C\sqrt{\bar\alpha}$ by Jensen's inequality. An exit $q_L$ is determined by the root $a_L$ of its cell: its other endpoint is the root of the parent cell (\Cref{slem:cells}(\ref{sit:refine})), so at level $m$ the possible exits number at most $\min\{n,2ke^m\}$, with $\|q_L\|_2^2\le2\bar\alpha/m$ by the parent cell's radius. Split $S_f(Z)^2$ by levels. At a level $m\le\ell_n$, the weights $\lambda_L\le m/k$ are again capped with total at most one, so \Cref{slem:capped}, applied to the squares with variance bound $2\bar\alpha/m$ and cap $m/k$, bounds the expected supremum of the level's sum by $C(2\bar\alpha/m)\bigl(1+[\log(2me^m)]_+\bigr)\le C\bar\alpha$, since $1+\log(2me^m)\le4m$; with the prefactor $m/k$ the level contributes $C\bar\alpha\,m/k$, and the dyadic sum over $m<k$ is at most $C\bar\alpha$. At the levels $m>\ell_n$, the prefactor obeys $m/k\le1$ and the weights total at most one, so the combined contribution is a feasible value of $\mathfrak M_1$ over the possible exits at these levels, at most $n(1+\log_2k)$ vectors of variance at most $2\bar\alpha/\ell_n$; \Cref{slem:capped} with cap $1$, applied to the squares, bounds its expectation by $C(\bar\alpha/\ell_n)\bigl(1+\log(n(1+\log_2k))\bigr)\le C\bar\alpha$, since $1+\log\bigl(n(1+\log_2k)\bigr)\le2\ell_n$. This proves \Cref{eq:exitsecond}.

Combining \Cref{eq:pairingsplit} with \Cref{eq:collapseresidual,eq:roundingresidual,eq:exitsecond},
\[
\E\,\sup_f\mathfrak r_Z(f)\ \le\
C\Bigl[\sqrt{\bar\alpha}\,\log(2+\ell_n)
+\sqrt{\overline\delta_k(T)\,\ell_n}\,\Bigr]
+C\sqrt{\bar\alpha}+4\,C\sqrt{\bar\alpha}\,,
\]
and $\sqrt{\bar\alpha}\le\sqrt{\bar\alpha}\log(2+\ell_n)$ absorbs the last two terms into the first.
\end{proof}

The width bound follows by recentering the local ball at a deterministic member of the cover.

\begin{lemma}[Truth-localized width]\label{slem:truthwidth}
For every integer $2\le k\le n/C_\star$, every $\mu\in\mathcal F_V(T)$, and every $t>0$,
\[
w_\mu(t)\ \le\
C\,\bigl(t+V\sqrt{\overline\delta_k(T)}\,\bigr)\sqrt k\ +\
C\,V\Bigl[\sqrt{k\overline\delta_k(T)}\,\log(2+\ell_n)
+\sqrt{\overline\delta_k(T)\,\ell_n}\,\Bigr],
\]
with $w_\mu$ the localized Gaussian width defined in Section~5; this is the display (5.10) there.
\end{lemma}

\begin{proof}
Abbreviate $\overline\delta:=\overline\delta_k(T)$. Amplitudes rescale: for $\theta\in\mathcal F_V(T)$ the normalized signal $\theta/V$ lies in $\mathcal F_1(T)$, and multiplying a member of the unit-amplitude cover by $V$ gives a member of $\mathcal D_k$ at amplitude $V$. The feasible set of \Cref{eq:compatibleinf} does not depend on $Z$ and is never empty, so fix $\nu_\mu$, $V$ times one of its members at $\mu/V$: then $\|\mu-\nu_\mu\|_2\le V\sqrt{51\overline\delta}$ and $\nu_\mu$ is deterministic. For each $\theta$ with $\|\theta-\mu\|_2\le t$ and each realization of $Z$, let $\nu_\theta$ be $V$ times a minimizer of \Cref{eq:compatibleinf} at $\theta/V$, so that $\|\theta-\nu_\theta\|_2\le V\sqrt{51\overline\delta}$ and $\langle Z,\theta-\nu_\theta\rangle=V\,\mathfrak r_Z(\theta/V)$. Decompose
\[
\langle Z,\ \theta-\mu\rangle\ =\
\langle Z,\ \nu_\theta-\nu_\mu\rangle
+\langle Z,\ \theta-\nu_\theta\rangle
-\langle Z,\ \mu-\nu_\mu\rangle .
\]
The last term does not depend on $\theta$ and has expectation zero, $\nu_\mu$ being deterministic, so it drops from the expected supremum. The middle term is at most $V\sup_{f\in\mathcal F_1(T)}\mathfrak r_Z(f)$, whose expectation \Cref{slem:compatible} bounds by the second group of the claim. For the first term, the triangle inequality confines every center met to a deterministic ball:
\[
\|\nu_\theta-\nu_\mu\|_2\ \le\
\|\nu_\theta-\theta\|_2+\|\theta-\mu\|_2+\|\mu-\nu_\mu\|_2\ \le\
t+2V\sqrt{51\overline\delta}\ =:\ \rho ,
\]
so the first term is at most the maximum of $\langle Z,\nu-\nu_\mu\rangle$ over $\{\nu\in\mathcal D_k:\|\nu-\nu_\mu\|_2\le\rho\}$, a fixed set of at most $e^{C_1k}$ centered Gaussians of standard deviation at most $\rho$. The maximum is a feasible value of $\mathfrak M_1$, so \Cref{slem:capped} with cap $1$ bounds its expectation by $C\rho\sqrt{1+C_1k}\le C(t+V\sqrt{\overline\delta}\,)\sqrt k$, the first group of the claim.
\end{proof}

The second comparison branch of Section~5 rests on the profile remembering the height of the tree.

\begin{lemma}[Height floor]\label{slem:heightfloor}
For every finite rooted tree and every integer $1\le k\le n/C_\star$,
\[
\delta_k(T)\ \ge\ \frac{\log2}8\cdot\frac{H_T}{k^2}\,.
\]
\end{lemma}

\begin{proof}
Write $h:=h_T$, so that $H_T\le2h$, and fix a root-to-leaf path with $h$ edges; every ancestor of a vertex of the path lies on the path, so ancestor covering restricted to it is interval covering of its $h+1$ vertices, as on the path of \Cref{slem:path}.

If $h\ge4k$, set $q:=\lfloor h/(2k)\rfloor-1\ge1$. A center covers at most $q+1$ consecutive path vertices, so
\[
\Nup(q)\ \ge\ \frac{h+1}{q+1}\ =\ \frac{h+1}{\lfloor
h/(2k)\rfloor}\ \ge\ 2k ,
\]
whence $[\log(\Nup(q)/k)]_+\ge\log2$ and the term of the discrete formula (2.3) at $q$ is at least $(q+1)\log2$. Since $q+1=\lfloor h/(2k)\rfloor>h/(2k)-1\ge h/(4k)$,
\[
\delta_k\ \ge\ \frac{(q+1)\log2}k\ \ge\ \frac{h\log2}{4k^2}\ \ge\
\frac{H_T\log2}{8k^2}\,.
\]

If $h<4k$, the term of (2.3) at $q=0$ is $\min\{k,[\log(n/k)]_+\}\ge\min\{k,2\}\ge1$, since $n/k\ge C_\star=e^2$; hence $\delta_k\ge1/k>H_T\log2/(8k^2)$, using $H_T\le2h<8k$ and $\log2<1$.
\end{proof}

The blocks are in place, and it remains to assemble them along the case analysis of Section~5.

\begin{proof}[Completion of Theorem~5, upper bound]
The skeleton, the meeting of the two branches at $k_0=\ell_n^{1/5}$, and the resulting bound $C(1+\log(en))^{2/5}$ are in Section~5; deferred here are the fixed-point verification behind (5.11), the height branch (5.12), and the endpoint cases.

\emph{The interior case: the width branch \textup{(5.11)}.} Let $2\le k_0\le K:=\lfloor n/C_\star\rfloor$, so that the crossing clause fired at $k_0$: $V^2\delta_{k_0}\le A_0\sigma^2k_0$, and $R^*_T\asymp\sigma^2k_0$ by the interior case of Section~3.3. Lemma~2.3 gives $\overline\delta_{k_0}\le2\delta_{k_0}$, so
\[
V\sqrt{\overline\delta_{k_0}}\ \le\ \sqrt{2A_0}\,\sigma\sqrt{k_0}\,,
\qquad
V\sqrt{k_0\overline\delta_{k_0}}\ \le\ \sqrt{2A_0}\,\sigma\,k_0\,,
\]
and \Cref{slem:truthwidth} at $k=k_0$ yields, for every $s>0$,
\[
w_\mu(s)\ \le\ Cs\sqrt{k_0}\ +\ C\sigma B,
\qquad
B\ :=\ k_0\log(2+\ell_n)+\sqrt{k_0\,\ell_n}\,,
\]
with one universal constant $C$, uniformly over $\mu\in\mathcal F_V(T)$: the term $CV\sqrt{\overline\delta_{k_0}}\,\sqrt{k_0}\le C'\sigma k_0$ from the first group, and both terms of the second group after the two displayed substitutions, are absorbed into $C\sigma B$, since $k_0\le B/\log2$. Set $t^2:=C_{\mathrm{loc}}\sigma^2B$ with $C_{\mathrm{loc}}:=64C^2/\log2+8C+1$. Then $t\ge\sigma$, since $B\ge k_0\log2\ge2\log2>1$; and for every $s\ge t$,
\[
\sigma\,w_\mu(s)\ \le\ C\sigma s\sqrt{k_0}+C\sigma^2B\ \le\
\frac{s^2}8+\frac{s^2}8\ =\ \frac{s^2}4\,,
\]
the first term because $s\ge t\ge\sqrt{C_{\mathrm{loc}}k_0\log2}\,\sigma\ge 8C\sigma\sqrt{k_0}$, and the second because $C\sigma^2B=Ct^2/C_{\mathrm{loc}}\le t^2/8\le s^2/8$. \Cref{slem:localized} applies at every $\mu$ and gives
\[
R^{\mathrm{worst}}_{\mathrm{LSE}}(T,V,\sigma)\ \le\ 33\,t^2\ =\
33\,C_{\mathrm{loc}}\,\sigma^2\bigl[k_0\log(2+\ell_n)+\sqrt{k_0\,\ell_n}\,\bigr];
\]
dividing by $R^*_T\asymp\sigma^2k_0$ proves (5.11).

\emph{The interior case: the height branch \textup{(5.12)}.} Still for $2\le k_0\le K$, \Cref{slem:heightfloor} applies at $k_0$ and combines with the crossing:
\[
\frac{\log2}8\cdot\frac{V^2H_T}{k_0^2}\ \le\ V^2\delta_{k_0}\ \le\
A_0\sigma^2k_0,
\qquad\text{so}\qquad
V^2H_T\ \le\ \frac{8A_0}{\log2}\,\sigma^2k_0^3 .
\]
Both $\mu$ and its projection lie in the body, so $\|\widehat\mu_{\mathrm{LSE}}-\mu\|_2^2\le\diam^2\mathcal F_V(T)=V^2H_T$ pathwise, and dividing by $R^*_T\asymp\sigma^2k_0$ proves (5.12).

\emph{The endpoints.} The projection onto a convex set is a contraction, so $\|\widehat\mu_{\mathrm{LSE}}-\mu\|_2= \|\Pi_{\mathcal F_V(T)}(Y)-\Pi_{\mathcal F_V(T)}(\mu)\|_2\le \|Y-\mu\|_2$ and $R^{\mathrm{worst}}_{\mathrm{LSE}}\le\sigma^2n$ always; together with the diameter,
\begin{equation}\label{eq:lsecaps}
R^{\mathrm{worst}}_{\mathrm{LSE}}(T,V,\sigma)\ \le\
\min\bigl\{V^2H_T,\ \sigma^2n\bigr\}.
\end{equation}
If $k_0=K+1$ with $n\ge C_\star$, the dimension case of Section~3.3 gives $R^*_T\asymp\sigma^2n$, and \Cref{eq:lsecaps} bounds the ratio by a constant. If $k_0=1$, the two-point bound (3.8) gives $R^*_T\ge c\min\{V^2H_T,\sigma^2\}$ for $n\ge2$, and there are two subcases. When $n<C_\star$, the right side of \Cref{eq:lsecaps} is at most $C_\star\min\{V^2H_T,\sigma^2\}$: its second branch is $\sigma^2n\le C_\star\sigma^2$, and its first is unchanged. When the crossing clause fired at $k=1$, the second case of Section~3.3 gives $V^2H_T\le2V^2h_T\le(2A_0/\log2)\,\sigma^2$, so the right side of \Cref{eq:lsecaps} is again at most $C\min\{V^2H_T,\sigma^2\}$: if $V^2H_T\le\sigma^2$ it equals the first branch, and otherwise it is at most $C\sigma^2$. In both subcases $R^{\mathrm{worst}}_{\mathrm{LSE}}\le C\,R^*_T$. The three cases are exhaustive: $k_0=K+1$ with $n<C_\star$ means $K=0$ and $k_0=1$, which the last case covers. Together with the two interior branches and the case analysis of Section~5, the upper bound of Theorem~5 is proved; with the lower bound established there, so is the theorem.
\end{proof}

\clearpage
\bibliographystyle{imsart-nameyear}
\bibliography{refs}

\end{document}